\documentclass[11pt,a4paper,reqno,
]{amsart}
\usepackage{a4wide}
\usepackage{amsfonts}
\usepackage{amsmath,latexsym,amssymb,amsfonts,amsbsy, amsthm}
\usepackage[pagebackref,
colorlinks,
allcolors=blue
]{hyperref}
\usepackage[usenames]{color}
\usepackage{scalerel}

\usepackage[backgroundcolor=yellow,linecolor=white]{todonotes}
\usepackage{enumitem}
\usepackage{xspace,colortbl}
\usepackage[utf8]{inputenc}
\allowdisplaybreaks
\usepackage{mathrsfs}
\usepackage{bbm}

\usepackage{lmodern}
\usepackage{microtype}

\newcommand{\beq}{\begin{equation}}
\newcommand{\eeq}{\end{equation}}
\newcommand{\ben}{\begin{eqnarray}}
\newcommand{\een}{\end{eqnarray}}
\newcommand{\beno}{\begin{eqnarray*}}
\newcommand{\eeno}{\end{eqnarray*}}

\theoremstyle{plain}
\newtheorem{theorem}{Theorem}[section]
\newtheorem{definition}[theorem]{Definition}
\newtheorem{corollary}[theorem]{Corollary}
\newtheorem{proposition}[theorem]{Proposition}
\newtheorem{lemma}[theorem]{Lemma}
\theoremstyle{definition}   
\newtheorem{remark}[theorem]{Remark}

\numberwithin{theorem}{section} \numberwithin{equation}{section}

\def\R{\mathbb{R}}
\def\loc{\text{loc}}

\newcommand{\del}{\partial }

\renewcommand{\phi}{\varphi}

\newcommand{\1}{\mathbbm{1} }

\newcommand{\scrC }{\mathscr{C}}

\newcommand{\N}{\mathbb{N}}

\newcommand{\cH}{{\mathcal H}}

\newcommand{\cL}{{\mathcal L}}

\newcommand{\cQ}{{\mathcal Q}}

\newcommand{\cS}{{\mathcal S}}

\newcommand{\dist}{{\rm dist}}
\newcommand{\supp}{{\rm supp}}
\newcommand{\weakto}{\rightharpoonup}

\newcommand{\B}{{\bf B}}

\newcommand{\eps}{\varepsilon}

\renewcommand{\epsilon}{\varepsilon}

\newcommand{\A}{{\bf A}}

\begin{document}
\title[Rotating waves in nonlinear media]{Rotating waves in nonlinear media and critical degenerate Sobolev inequalities}
\date{\today}

	\author{Joel K\"ubler} 
\address{Institut f\"ur Mathematik,
	Goethe-Universit\"at Frankfurt,
	Robert-Mayer-Str. 10,
	D-60629 Frankfurt am Main, Germany}
\email{kuebler@math.uni-frankfurt.de}

	\author{Tobias Weth} 
\address{Institut f\"ur Mathematik,
	Goethe-Universit\"at Frankfurt,
	Robert-Mayer-Str. 10,
	D-60629 Frankfurt am Main, Germany}
\email{weth@math.uni-frankfurt.de}

%
%
\date{\today}

\subjclass[2020]{Primary:  35J20, Secondary: 35J70, 35B33}  	
\keywords{Nonlinear wave equation,
degenerate Sobolev inequality,
symmetry breaking,
concentration-compactness} 
%
%
\begin{abstract}
  We investigate the presence of rotating wave solutions of the nonlinear wave 
  equation $\del_t^2 v - \Delta v +m v  = |v|^{p-2} v$ in $\R \times \B$, where 
  $\B \subset \R^N$ is the unit ball, complemented with Dirichlet boundary 
  conditions on $\R \times \del\B$. Depending on the prescribed angular 
  velocity 
  $\alpha$ 
  of the rotation, this leads to a Dirichlet problem for a semilinear elliptic 
  or degenerate elliptic equation. We show that this problem is governed by an 
  associated critical degenerate Sobolev inequality in the half space. After 
  proving this inequality and the existence of associated extremal functions, 
  we then deduce necessary and sufficient conditions for the existence of 
  ground state solutions. Moreover, we analyze under which conditions on 
  $\alpha$, $m$ and $p$   
  these  ground states are nonradial and therefore give rise to truly rotating 
  waves. Our approach carries over to the corresponding Dirichlet problems in 
  an annulus and in more general Riemannian models with boundary, including the 
  hemisphere.
  We briefly discuss these problems and show that they are related to a larger 
  family of associated critical degenerate Sobolev inequalities.
\end{abstract}

\maketitle
%
%
%
%
%
\section{Introduction}
Within a standard model, the analysis of wave propagation in an ambient medium 
with nonlinear response leads to the study of a nonlinear wave equation of the 
type
\begin{equation}
	\label{eq: abstract nonlinear wave equation} 
	\del_t^2 v - \Delta v +m v = f(v)  \qquad \text{in $\R \times \Omega$,} 
\end{equation}
in an ambient domain $\Omega \subset \R^N$ with mass parameter $m \ge 0$ and 
nonlinear response function $f$. In the case $m=0$, \eqref{eq: abstract 
nonlinear wave equation}  is the 
classical nonlinear wave equation, while the case $m>0$ is also known as 
a nonlinear Klein-Gordon equation. For nonlinearities of the form 
$f(v)=g(|v|^2)v$ with a real-valued function $g$, standing wave solutions 
can be found by the ansatz 
\begin{equation}
	\label{psi-ansatz}
	v(t,x)=e^{-i k t}u(x), \qquad k >0
\end{equation}
with a real-valued function $u$. Depending on the frequency parameter $k$, this 
reduces (\ref{eq: abstract nonlinear wave equation}) either to a stationary 
nonlinear Schrödinger or a 
nonlinear Helmholtz equation (see e.g. \cite{evequoz-weth} for more details). 
The resulting stationary nonlinear Schrödinger equation has been studied 
extensively in the past four decades by variational methods, see e.g. the 
monograph \cite{Ambrosetti-malchiodi} and the references therein. Due to a lack 
of a direct variational framework, the nonlinear Helmholtz equation requires a 
different approach and has been studied more recently e.g. in 
\cite{gutierrez,evequoz-weth,chen-evequoz-weth,mandel-scheider-yesil,mandel-montefusco-pellaci}
by dual variational methods and bifurcation theory.

Clearly, the amplitude $|v|$ of a solution $v$ of (\ref{eq: abstract nonlinear wave equation}) given by the ansatz (\ref{psi-ansatz}) remains time-independent. As a consequence, the analysis of standing wave solutions does not lead to a full understanding of (\ref{eq: abstract nonlinear wave equation}) from a dynamical point of view and should be complemented, in particular, by the study of non-stationary real-valued time-periodic solutions, travelling wave solutions and scattering solutions. We stress that the ansatz (\ref{psi-ansatz}) does not give rise to non-stationary real-valued time-periodic solutions since the nonlinearity of the problem does not allow to pass to real and imaginary parts. 

In the case where $\Omega= \R^N$ and $f(v)$ in (\ref{eq: abstract nonlinear 
wave equation}) is replaced by $q(x)f(v)$ with a compactly supported weight 
function $q$, spatially localized real-valued time-periodic solutions, also 
called breathers, have attracted increasing attention recently, see e.g. 
\cite{hirsch-reichel,mandel-scheider} and the references therein. In the case 
where $\Omega$ is a radial domain, a further interesting type of real-valued 
time-periodic solution is given by {\em rotating wave solutions}. In 
particular, if $\Omega$ is a bounded radial domain and (\ref{eq: abstract 
nonlinear wave equation}) is complemented with the Dirichlet boundary condition 
$v = 0$ on $\R \times \del \Omega$, the existence of rotating waves and their 
variational characterization arises as a natural question which, up to our 
knowledge, has not been addressed systematically so far.

The main purpose of the present paper is to provide such a systematic study. 
While we mainly focus on the case where $\Omega= \B$ is the unit ball in 
$\R^N$, we will also address the case where $\Omega$ is an annulus or a general 
Riemannian model with boundary, see Sections \ref{sec:case-an-annulus} and 
\ref{sec:riemannian-models} below. Specifically, we study the case of a 
focusing nonlinearity of the form $f(v)=|v|^{p-2}v$, which leads to the 
superlinear problem 
\begin{equation} \label{NLKG}
\left\{ 
\begin{aligned}
 \del_t^2 v - \Delta v +m v & = |v|^{p-2} v \quad && \text{in $\R \times \B$} 
 \\
 v & = 0 && \text{on $\R \times \del\B$}
\end{aligned}
\right.
\end{equation}
for $N \geq 2$, where $2<p<2^*$ and 
$m>-\lambda_1(\B)$. Here, $\lambda_1(\B)$ denotes the first Dirichlet 
eigenvalue 
of $-\Delta$ on $\B$ and $2^*$ denotes the critical Sobolev exponent given by 
$2^*= \frac{2N}{N-2}$ for $N \ge 3$ and $2^* = \infty$ for $N=2$. 
The ansatz for time-periodic rotating solutions of \eqref{NLKG} is given by 
\begin{equation} \label{Spiraling wave ansatz}
v(t,x)=u(R_{\alpha t} (x))
\end{equation}
where, for $\theta \in \R$, we let $R_{\theta} \in O(N)$ denote a planar 
rotation in $\R^N$ with angle $\theta$, so the constant $\alpha>0$ in 
(\ref{Spiraling wave ansatz}) is the angular velocity of the rotation. Without 
loss of generality, we may assume that 
$$
R_{\theta}(x)= (x_{1} \cos \theta  + x_2 \sin \theta , -x_{1} \sin \theta  + x_2 \cos \theta ,x_3,\dots,x_N) \qquad \text{for $x \in \R^N$},
$$
so $R_{\theta}$ is the rotation in the $x_{1}$-$x_2$-plane with fixed point set 
$\{0_{\R^2}\} \times \R^{N-2}$. In the following, we call a function $u$ on the 
unit ball \textit{$x_1$-$x_2$-nonradial} if it is not $R_\theta$-invariant for at least 
one angle $\theta \in \R$. If the profile function $u$ in (\ref{Spiraling wave 
ansatz}) is $x_1$-$x_2$-nonradial, then the corresponding solution $v$ can be 
interpreted as a rotating wave in a medium with nonlinear response given by the 
right hand side of (\ref{NLKG}). The ansatz (\ref{Spiraling wave ansatz}) 
reduces 
\eqref{NLKG} to 
\begin{equation} \label{Reduced equation}
\left\{ 
\begin{aligned} 
-\Delta u + \alpha^2 \del_\theta^2 u + m u & = |u|^{p-2} u \quad && \text{in $\B$} \\
u & = 0 && \text{on $\del \B$}
\end{aligned}
\right.
\end{equation}
where $\del_\theta = x_{1} \del_{x_2} - x_{2} \del_{x_{1}}$ denotes the 
associated angular derivative operator. We point out that a seemingly closely 
related equation, with the term $\alpha^2 \del_\theta^2 u$ replaced by 
$-\alpha^2 \del_\theta^2 u$, arises in an ansatz for solutions of nonlinear 
Schr\"odinger equations in $\R^3$ with invariance with respect to screw motion, 
see \cite{Agudelo-Kuebler-Weth} and also \cite{del Pino-Musso-Pacard} for a related work on 
Allen-Cahn equations. Note, however, that the positive sign of the term 
$\alpha^2 \del_\theta^2 u$ results in a drastic change of the nature of the 
problem, as the operator $-\Delta + \alpha^2 \del_\theta^2$ loses uniform 
ellipticity in $\B$ if $\alpha \ge 1$. This also distinguishes the study of 
(\ref{Reduced equation}) from the related study of rotating solutions to 
nonlinear Schrödinger equations, where the angular velocity $\alpha$ appears 
within a first order term which does not affect the ellipticity of the 
associated Schrödinger operator, see e.g. \cite{Seiringer,Lieb-Seiringer} and 
the references therein.

If a solution $u$ of (\ref{Reduced equation}) satisfies $\del_\theta u \equiv 
0$ in $\B$, then $u$ solves the classical stationary nonlinear Schrödinger 
equation
$-\Delta u + m u= |u|^{p-2} u$ in $\B$ with Dirichlet boundary conditions on 
$\partial \B$, so it satisfies \eqref{Reduced equation} with $\alpha=0$. If, in 
addition, $u$ is positive, then $u$ has to be a radial function as a 
consequence of the symmetry result of Gidas, Ni and Nirenberg 
\cite{Gidas-Ni-Nirenberg}. Thus, the ansatz (\ref{Spiraling wave ansatz}) then 
merely gives rise to a radial stationary solution of (\ref{NLKG}). We mention 
here that radially symmetric 
non-stationary solutions of (\ref{eq: abstract nonlinear wave equation}) in 
$\Omega = \B$ were first studied by Ben-Naoum and 
Mahwin~\cite{Ben-Naoum-Mahwin} for sublinear 
nonlinearities and more recently by Chen and 
Zhang~\cite{Chen-Zhang,Chen-Zhang-2,Chen-Zhang-3}.  In this problem, the 
spectral 
properties of the 
radial wave operator lead to delicate assumptions on the dimension as 
well as the ratio between the radius of the ball and the period length. The main purpose of the present paper is to analyze for which range of parameters $\alpha$, $m$ and $p$ 
ground state solutions of \eqref{Reduced equation} exist and to distinguish under which assumptions on $\alpha$, $m$ and $p$ they are radial or $x_1$-$x_2$-nonradial and therefore correspond to rotating waves via the ansatz (\ref{Spiraling wave ansatz}).

By a ground state solution of \eqref{Reduced equation}, we mean a solution characterized as a minimizer of the minimization problem
for 
\begin{equation}
  \label{eq:intro-minimization-problem}
\scrC_{\alpha,m,p}(\B) :=\inf_{u \in H^1_0(\B) \setminus \{0\}} R_{\alpha,m,p}(u),
\end{equation}
where, for $m \in \R$, $\alpha \geq 0$ and $p \in [2,2^*)$, we consider the 
associated Rayleigh quotient $R_{\alpha,m,p}$ given by 
\begin{equation}
  \label{eq:def-Raleigh-quotient}
R_{\alpha,m,p}(u)=  \frac{\int_\B \left( |\nabla u|^2 - \alpha^2  |\del_\theta u|^2 + m u^2 \right) \, dx}{\left(\int_\B |u|^p \, dx \right)^\frac{2}{p}}, \qquad u \in H^1_0(\B) \setminus \{0\}.
\end{equation}
As we shall see in Remark~\ref{remark-alpha-greater-1} below, this minimization problem is only meaningful for $0 \leq \alpha \le 1$, since for every $p \in [2,2^*)$ and $m \in \R$ we have 
\begin{equation*}
\scrC_{\alpha,m,p}(\B)= -\infty \quad \text{for $\alpha >1$.}
\end{equation*}
Moreover, for every $p \in [2,2^*)$ and $m \in \R$,
\begin{equation}
  \label{eq-C-intro-new-1-1}
\text{the function $\alpha \mapsto \scrC_{\alpha,m,p}(\B)$ is continuous and nonincreasing on $[0,1]$.}
\end{equation}
In the case $0 < \alpha < 1$, the operator $-\Delta+ \alpha^2 \del_\theta^2$ is uniformly elliptic, as can be seen by writing the operator in polar coordinates as
\begin{equation}
  \label{eq:op-polar-coord}
-\Delta+ \alpha^2 \del_\theta^2 = -\Delta_r u -\frac{1}{r^2} 
\Delta_{\mathbb{S}^{N-1}} u + \alpha^2 \del_\theta^2 u,
\end{equation}
where $\Delta_{\mathbb{S}^{N-1}}$ denotes the Laplace-Beltrami operator on the 
unit sphere $\mathbb{S}^{N-1}$. In this case the existence of minimizers of 
$R_{\alpha,m,p}$ on $H^1_0(\B) \setminus \{0\}$ follows by a standard 
compactness and weak lower semicontinuity argument. However, even in this case 
it is difficult to decide in general whether minimizers are radial or nonradial 
functions.
This is due to competing effects. Firstly, the additional term $-\alpha^2 
\|\del_\theta u\|_{L^2(\B)}^2$ favours $x_1$-$x_2$-nonradial functions as energy 
minimizers. On the other hand, the P\'olya-Szegö inequality 
yields $\int_\B |\nabla u^*|^2\,dx \le \int_\B |\nabla u|^2\,dx$, where $u^*$ denotes the (radial) Schwarz symmetrization of a function $u \in H^1_0(\B)$.

Since $R_{\alpha,m,p}(u)= R_{0,m,p}(u)$ for every radial function $u \in 
H^1_0(\B) \setminus \{0\}$ and every $\alpha \in [0,1]$, a sufficient condition 
for the $x_1$-$x_2$-nonradiality of all ground state solutions is the inequality
\begin{equation}
  \label{eq-C-sufficient-ineq-intro}
\scrC_{\alpha,m,p}(\B)< \scrC_{0,m,p}(\B). 
\end{equation}
In particular, we will be interested in proving this inequality for $\alpha$ 
close to $1$. We point out that the borderline case $\alpha = 1$ differs 
significantly from the case $0\le \alpha <1$, as the differential operator 
$-\Delta+ \del_\theta^2$ is no longer uniformly elliptic on $\B$. In fact, it 
follows from the representation (\ref{eq:op-polar-coord}) in the case $\alpha 
=1$ that the operator $-\Delta+ \del_\theta^2$ fails to be uniformly elliptic 
in a neighborhood of the great circle $\{ x \in \del \B: x_3 = \cdots = 
x_{N}=0\}$ (which equals $\del \B$ in the case $N=2$). We shall see in this 
paper that the minimization problem in the case $\alpha =1$ is essentially 
governed by a degenerate anisotropic critical Sobolev inequality in the half 
space. The corresponding critical exponent in this Sobolev inequality is given 
by 
\begin{equation*}
2_{\text{\tiny $1$}}^*:=\frac{4N+2}{2N-3}.  
\end{equation*}
The relevance of this exponent is indicated by our first main result which yields the following characterization.
\begin{theorem} \label{main theorem}
  Let $m > -\lambda_1(\B)$ and $p \in (2,2^*)$.
  \begin{itemize}
  \item[(i)] If $\alpha \in (0,1)$, then there exists a ground state solution of (\ref{Reduced equation}).
  \item[(ii)] We have
        \begin{equation}
          \label{eq:main-theorem-dist}
        \scrC_{1,m,p}(\B) =0 \quad \text{for} \   p>2_{\text{\tiny $1$}}^*, \qquad \text{and} \quad \scrC_{1,m,p}(\B)>0 \quad \text{for} \  p \leq 2_{\text{\tiny $1$}}^*.	
        \end{equation}
	Moreover, for any $p \in (2_{\text{\tiny $1$}}^*,2^*)$, there exists $\alpha_p \in (0,1)$ with the property that
        $$
        \scrC_{\alpha,m,p}(\B)< \scrC_{0,m,p}(\B) \qquad \text{for $\alpha \in (\alpha_p,1]$}
        $$
        and therefore every ground state solution of (\ref{Reduced equation}) is $x_1$-$x_2$-nonradial for $\alpha \in (\alpha_p,1)$.
      \end{itemize}
  \end{theorem}

The following new degenerate Sobolev inequality is an immediate consequence of the special case $m=0$, $\alpha =1$ in Theorem~\ref{main theorem}.

\begin{corollary}
  \label{main-theorem-corollary}
\begin{equation*}
\Bigl(\int_\B |u|^{2_{\text{\tiny $1$}}^*} \, dx 
\Bigr)^{\frac{2}{2^*_{\scaleto{1}{3pt}}}} \le 
\frac{1}{{\scrC_{1,0,p}(\B)}}\int_\B\bigl( |\nabla u|^2 - |\del_\theta 
u|^2 \bigr)dx \qquad \text{for $u \in H^1_0(\B)$.} 
\end{equation*}
Moreover, the exponent $2_{\text{\tiny $1$}}^*$ is optimal in the sense that no such inequality holds for $p>2_{\text{\tiny $1$}}^*$.
\end{corollary}

Theorem~\ref{main theorem} yields symmetry breaking of ground states for suitable parameter values of $p$, $\alpha$ and $m$, but the precise parameter range giving rise to this symmetry breaking remains largely open. 
To shed further light on this question, we state the following result which establishes uniqueness and radial symmetry of ground state solutions for $\alpha$ close to zero and every  $m \geq 0$, $2<p<2^*$.

\begin{theorem}
  \label{sec:introduction-second-theorem}
  Let $m \geq 0$ and $2<p<2^*$. Then there exists $\alpha_0>0$ such that
\begin{equation*}
\scrC_{\alpha,m,p}(\B)= \scrC_{0,m,p}(\B) \qquad \text{for $\alpha \in [0,\alpha_0)$.} 
\end{equation*}
Moreover, for $\alpha \in [0,\alpha_0)$, there is, up to sign, a unique ground state solution of (\ref{Reduced equation}) which is a radial function.
\end{theorem}

Combining Theorems~\ref{main theorem} and 
\ref{sec:introduction-second-theorem}, we find that, for fixed $p 
>2_{\text{\tiny $1$}}^*$, symmetry breaking of ground state solutions occurs 
when passing a critical parameter $\alpha=\alpha(p)$ which lies in the 
intervall $[\alpha_0,\alpha_*]$. However, so far it remains unclear whether 
symmetry breaking also occurs in the case $p \leq 2_{\text{\tiny $1$}}^*$.
Before stating a partial answer to this question for $2<p< 2_{\text{\tiny $1$}}^*$, we first note that symmetry breaking does not occur in the linear case $p=2$. More precisely, we shall observe in Section~\ref{preliminaries} below that
$$
\scrC_{\alpha,m,2}(\B) =\scrC_{0,m,2}(\B)= \lambda_1(\B)+m \qquad \text{for all $\alpha \in [0,1]$, $m \in \R$.}
$$
Moreover, every Dirichlet eigenfunction of (\ref{eq:intro-minimization-problem}) is radial in this linear case. On the other hand, for every $p$ strictly greater than $2$, symmetry breaking occurs for sufficiently large values of the parameter $m$, as the following result shows.
\begin{theorem} \label{Thm: m large}
  Let $\alpha \in (0,1)$ and $2<p<2^*$. Then there exists $m_0>0$ with the property that (\ref{eq-C-sufficient-ineq-intro}) holds for $m \ge m_0$ and therefore every ground state solution of (\ref{Reduced equation}) is $x_1$-$x_2$-nonradial for $m \ge m_0$.
\end{theorem}	

Next, we discuss the limit case $\alpha=1$ in the minimization problem 
(\ref{eq:intro-minimization-problem}). We may study this limit case based on 
Corollary~\ref{main-theorem-corollary}, but we need to look for minimizers in a 
space larger than $H^1_0(\B)$. More precisely, we let $\cH$ be given as the 
closure of $C_c^1(\B)$ in
$$
\left\{ u \in L^{{2_{\text{\tiny $1$}}^*}}(\B): \|u\|_{\cH}^2 := \int_{\B} 
\left(|\nabla u|^2 - |\del_\theta u|^2 \right) \, dx < \infty \right\} 
$$
with respect to the norm $\|\cdot \|_{\cH}$.
We then have the following result, which complements Theorems~\ref{main 
theorem} and~\ref{Thm: m large} in the case $\alpha=1$.

\begin{theorem} \label{Thm: m large-alpha-1}
  Let $2<p<2_{\text{\tiny $1$}}^*$ and $\alpha=1$.
  \begin{itemize}
  \item[(i)] For every $m>-\lambda_1(\B)$, there exists a ground state solution of (\ref{Reduced equation}).
  \item[(ii)] There exists $m_0>0$ with the property that 
  (\ref{eq-C-sufficient-ineq-intro}) holds for $m \ge m_0$ and therefore every 
  ground state solution $u \in \cH$ of (\ref{Reduced equation}) is 
  $x_1$-$x_2$-nonradial for $m \ge m_0$.  
  \end{itemize}
\end{theorem}	

The critical case $\alpha=1$, $p=2_{\text{\tiny $1$}}^*$ remains largely open, but we have a partial result on the existence of ground state solutions which relates problem (\ref{Reduced equation}) to a degenerate Sobolev inequality of the form
        \begin{equation}
          \label{eq:def-S_n-intro-ineq}
\|u\|_{L^{2_s^*}(\R^N_+)} \le C \Bigl(\int_{\R^N_+} \sum_{i=1}^{N-1} |\del_i 
u|^2 + x_1^s |\del_N u|^2 \, dx\Bigr)^{1/2} 
        \end{equation}
in the half space
$$
\R^N_+:=\left\{ x \in \R^N: x_1>0 \right\}.
$$
This inequality seems new and of independent interest, and it is the key ingredient in the proof of Theorem~\ref{main theorem}.
Our main result related to this half space inequality is the following. 

\begin{theorem} \label{existence of minimizers - whole space - intro}
	Let $s>0$ and set $2_s^* := \frac{4N+2s}{2N-4+s}.$
	Then we have
        \begin{equation}
          \label{eq:def-S_n-intro}
	\cS_s(\R^N_+) := \inf_{ u \in C^1_c(\R^N_+)} 
	\frac{\displaystyle\int_{\R^N_+} \sum \limits_{i=1}^{N-1} |\del_i u|^2 +  
	x_1^s |\del_N u|^2 \, dx}{\left(\int_{\R^N_+} |u|^{2_s^*} \, 
	dx\right)^\frac{2}{2_s^*}} > 0.
        \end{equation}
	Moreover, the value $\cS_s(\R^N_+)$ is attained in $H_s \setminus \{0\}$, where
        $H_s$ denotes the closure of $C_c^1(\R^N_+)$ in the space  
       \begin{equation}
\label{existence of minimizers - whole space - intro-eq}          
\left\{ u \in L^{2_s^*}(\R^N_+):\|u\|_{H_s}^2:= \int_{\R^N_+} \sum_{i=1}^{N-1} 
|\del_i u|^2 + 
 x_1^s |\del_N u|^2 \, dx < \infty \right\} 
        \end{equation}
with respect to the norm $\|\cdot \|_{H_s}$.
\end{theorem}
Here, distributional derivatives are considered in (\ref{existence of 
minimizers - whole space - intro-eq}). Several remarks regarding 
Theorem~\ref{existence of minimizers - whole space - intro} are in order.
First, we point out that the criticality of the exponent $2_s^* := 
\frac{4N+2s}{2N-4+s}$ in \ref{existence of minimizers - whole space - intro} 
corresponds to the fact that the quotient in \eqref{eq:def-S_n-intro} is 
invariant under an anisotropic rescaling given by $u \mapsto u_\lambda$ for 
$\lambda>0$ with $u_\lambda(x):=u(\lambda x_1, \lambda x_2, \ldots, \lambda 
x_{N-1},\lambda^{1+\frac{s}{2}} x_N)$. This invariance leads to a lack of 
compactness, and we have to apply concentration-compactness methods to deduce 
the existence of minimizers. We further note that the existence of minimizers 
in 
the half space problem is in striking contrast to the case $s=0$ which is 
excluded in Theorem~\ref{existence of minimizers - whole space - intro}. 
Indeed, the case $s=0$ corresponds to the classical Sobolev inequality which 
only admits extremal functions in the entire space $\R^N$.

We have already noted that the case $s=1$ in Theorem~\ref{existence of 
minimizers - whole space - intro} is of key importance in the proof of 
Theorem~\ref{main theorem}. The more general case $s \in (0,2]$ arises in a 
similar way when (\ref{Reduced equation}) is studied in Riemannian models with 
boundary in place of $\B$, and we will discuss this case in 
Section~\ref{sec:riemannian-models} below. We point out that the setting of 
Riemannian models includes hypersurfaces of revolution with boundary in 
$\R^{N+1}$, and that the particular case of a hemisphere corresponds to the 
case $s=2$. The latter is no surprise in view of the recent work of Taylor 
\cite{Taylor} and Mukherjee \cite{Mukherjee,Mukherjee2}, who studied the 
problem of rotating solutions on the unit sphere. In particular, their work 
relies on degenerate Sobolev embeddings on the unit sphere where also the value 
$2_{\text{\tiny $2$}}^* = \frac{2(N+1)}{N-1}$ appears as a critical exponent. 
In fact, our approach allows to use the case $s=2$ in Theorem~\ref{existence of 
minimizers - whole space - intro} and the corresponding inequality in $\R^N$ 
(see Theorem~\ref{claim-ineq-lemma} below) to give new proofs of these 
degenerate Sobolev embeddings which does not rely on Fourier analytic and 
pseudodifferential arguments as in \cite{Taylor}. 

Next we remark that degenerate Sobolev type inequalities have been studied extensively in the context of Grushin operators which take the form
$$
\cL = \Delta_x + c |x|^{2s} \Delta_y 
$$
on $\R^N=\R^m \times \R^k$, where $x \in \R^m$, $y \in \R^k$ and $s>0$. For a 
comprehensive survey of the properties of these operators, see e.g. 
\cite{Hajlasz-Koskela}. In particular, an associated Sobolev type inequality of 
the type 
\begin{equation} \label{Classical Grushin Quotient}
\|u\|_{L^{\frac{2m+2k(s+1)}{m+k(s+1)-2}}(\R^N)} \le  C \Bigl(\int_{\R^N} |\nabla_x u|^2 + c |x|^{2s} |\nabla_y u|^2 \, d(x,y)\Bigr)^{1/2}, \qquad u \in C^1_c(\R^N)
\end{equation}
has been established. Here, the associated critical exponent is related to the homogeneous dimension in the context of more general weighted Sobolev inequalities. We also mention symmetry results for positive entire solutions to semilinear problems involving $\cL$ in \cite{Monti-Morbidelli}, as well as the existence of extremal functions on $\R^N$ shown in \cite{Beckner} and \cite{Monti}. 

We point out that the restriction of inequality (\ref{Classical Grushin Quotient}) to the half space coincides with the inequality (\ref{eq:def-S_n-intro-ineq}) in the case $N=2$. On the other hand, for $N \geq 3$, the inequality (\ref{eq:def-S_n-intro-ineq}) is not associated to a Grushin operator in the sense above. Nonetheless, it is worth noting that for $m=N-1$, $k=1$ and $s=\frac{1}{2}$, the critical exponents coincide. 

More closely related to Theorem~\ref{existence of minimizers - whole space - intro} in the case $N \ge 3$ is \cite[Theorem 1.7]{Filippas et al} where a more general family of Grushin type operators and their associated inequalities has been considered.
However, the inequality (\ref{eq:def-S_n-intro-ineq}) associated to (\ref{eq-C-sufficient-ineq-intro}) is a limit case which is not part of the family of inequalities considered in \cite[Theorem 1.7]{Filippas et al}. 

Coming back to the existence of ground state solutions of (\ref{Reduced equation}) in the critical case $\alpha =1$, $p=2_{\text{\tiny $1$}}^*$, we state the following result.

\begin{theorem} \label{Theorem: Minima comparison - introduction}
  If
  \begin{equation}
    \label{eq:critica-threshold}
  \scrC_{1,m,{2^*_{\scaleto{1}{3pt}}}}(\B) < 
  2^{\frac{1}{2}-\frac{1}{2^*_{\scaleto{1}{3pt}}}} 
  \cS_1(\R^N_+)  
  \end{equation}
  for some $m>-\lambda_1(\B)$, then the value $\scrC_{1,m,{2_{\text{\tiny 
  $1$}}^*}}(\B)$ is attained in $\cH \setminus \{0\}$ by a ground state 
  solution of (\ref{Reduced equation}). Moreover, there exists $\eps>0$ with 
  the property that (\ref{eq:critica-threshold}) holds for every $m \in 
  (-\lambda_1(\B),-\lambda_1(\B)+\eps)$.
\end{theorem}
Here, the factor $2^{\frac{1}{2}-\frac{1}{2^*_{\scaleto{1}{3pt}}}}$ is due to 
the scaling properties 
of a more general quotient related to \eqref{eq:def-S_n-intro},
see Remark~\ref{Remark: Optimality of exponents}(ii) below.

The paper is organized as follows. 
We first study the degenerate Sobolev inequality~\eqref{eq:def-S_n-intro-ineq} 
and hence prove Theorem~\ref{existence of minimizers - whole space - intro} in 
Section~\ref{sec-Halfspace}.
This is subsequently used in Section~\ref{sec:degen-sobol-ineq} to prove the 
second part of Theorem~\ref{main theorem}.
In Section~\ref{preliminaries} we then discuss the properties of ground state 
solutions of \eqref{Reduced equation} in detail and give the 
proofs of Theorems \ref{sec:introduction-second-theorem} and \ref{Thm: m 
large}. This also includes the degenerate case $\alpha=1$ and the proof of 
Theorem~\ref{Thm: m large-alpha-1}.
Section~\ref{sec:case-an-annulus} is then devoted to the properties of rotating 
waves when $\B$ is replaced by an annulus. In this case, our methods give rise 
to an analogue of 
Theorem~\ref{main theorem} with more explicit conditions for 
$x_1$-$x_2$-nonradiality of ground states.
In Section~\ref{sec:riemannian-models} we discuss how the general degenerate 
Sobolev inequality \eqref{eq:def-S_n-intro-ineq} can be used to give an 
analogue of Theorem~\ref{main theorem} for Riemannian models.
Finally, in the appendix, we prove uniform $L^\infty$-bounds for weak solutions 
of 
\eqref{Reduced equation} in the case $\alpha= 1$.

   \section{A family of degenerate Sobolev inequalities} \label{sec-Halfspace}
In this section, we give the proof of Theorem~\ref{existence of minimizers - whole space - intro}. More precisely, in the first part of the section, we prove the corresponding degenerate Sobolev inequality   
\begin{equation} \label{S_R positive-ineq}
  \Bigl(\int_{\R^N} |u|^{2_s^*} \, dx\Bigr)^{\frac{2}{2_s^*}}\le C \int_{\R^N} 
  \Bigl( \sum \limits_{i=1}^{N-1} |\del_i u|^2 + |x_1|^s |\del_N u|^2 \Bigr) 
  \, dx \qquad \text{for $u \in C^1_c(\R^N)$}
	\end{equation}
        in the entire space with a constant $C>0$, from which the positivity of 
        $\cS_s(\R^N_+)$ in (\ref{eq:def-S_n-intro}) follows.
        
        In the second part of the section, we then prove the existence of minimizers of the quotient in (\ref{eq:def-S_n-intro}) in the larger space $H_s$ defined in Theorem~\ref{existence of minimizers - whole space - intro}.\\ 
        
\subsection{Degenerate Sobolev inequality on $\R^N$}
The first step in the proof of (\ref{S_R positive-ineq}) is the following key inequality.

\begin{lemma}
  \label{claim-ineq-lemma}
Let $\alpha >0$ and $p > 2$ be given. Then we have 
\begin{equation}
  \label{eq:claim-ineq-lemma}
\int_{\R^N}|u|^{p}\,dx  \le \kappa  \Bigl(\int_{\R^N}|x_1|^{\alpha} |u|^{q}\,dx\Bigr)^{\frac{2}{2+\alpha}} \Bigl( \int_{\R^N} |\partial_1 u|^2 \,dx\Bigr)^{\frac{\alpha}{2+\alpha}} \qquad \text{for $u \in C^1_c(\R^N)$}
\end{equation}
with
\begin{equation}
\label{eq:def-q-lemma}
q = \frac{p(2+\alpha)-2\alpha}{2} \qquad \text{and}\qquad \kappa >0 .
\end{equation}
\end{lemma}
\begin{proof}
	Let $u \in C^1_c(\R^N)$. By Hölder's inequality, we have 
	\begin{equation}
	\label{eq:hoelder-1}
	\int_{\R^N} |u|^p\,dx \le \Bigl(\int_{\R^N}|x_1|^{{s} {\sigma}'} |u|^{{r}{\sigma}'}\,dx\Bigr)^{\frac{1}{{\sigma}'}} \Bigl(\int_{\R^N}|x_1|^{-{s} {\sigma}}|u|^{{(p-r)}{\sigma}}\,dx\Bigr)^{\frac{1}{\sigma}}
	\end{equation}
	for $s>0$, $\sigma \in (1,\infty)$ and ${r} \in (0,p)$.
        It is convenient to write $s = \frac{t}{\sigma}$ and $m= (p-r) \sigma$, then (\ref{eq:hoelder-1}) becomes
        \begin{equation}
	\label{eq:hoelder-1-1}
	\int_{\R^N} |u|^p\,dx \le \Bigl(\int_{\R^N}|x_1|^{\frac{t}{\sigma-1}} |u|^{p \sigma' - \frac{m}{\sigma-1}}\,dx\Bigr)^{\frac{1}{{\sigma}'}} \Bigl(\int_{\R^N}|x_1|^{-t}|u|^{m}\,dx\Bigr)^{\frac{1}{\sigma}}
	\end{equation}         
	for $t>0$, $\sigma \in (1,\infty)$ and $m \in (0,p\sigma)$. If, more specifically,
                \begin{equation}
          \label{eq:admissibility-con}
          t \in (0,1),\quad \sigma \in (1,\infty),\quad m \in (1,p\sigma),\quad 
          \tau>1 \quad \text{and}\quad \theta \in (0,1),
      \end{equation}
       we may integrate by parts and use Hölder's inequality to get
	\begin{align}
	\int_{\R^N}&|x_1|^{-t}|u|^{m}\,dx =
	- \frac{m}{1-t} \int_{\R^N}x_1 |x_1|^{-t}|u|^{m-1}\partial_1 u \,dx \nonumber\\
	&\le c \int_{\R^N}|x_1|^{1-t}|u|^{m-1}|\partial_1 u| \,dx \le c
	\Bigl(\int_{\R^N}|x_1|^{2(1-t)}|u|^{2(m-1)}\,dx \Bigr)^{\frac{1}{2}} \Bigl( \int_{\R^N} |\partial_1 u|^2 \,dx\Bigr)^{\frac{1}{2}} \nonumber \\
	&\leq c \Bigl(\int_{\R^N}|x_1|^{2(1-t)\tau}|u|^{2\theta(m-1)\tau}\,dx \Bigr)^\frac{1}{2\tau} \Bigl( \int_{\R^N} |u|^{2(1-\theta)(m-1)\tau'} \, dx \Bigr)^\frac{1}{2\tau'}
	\Bigl( \int_{\R^N} |\partial_1 u|^2 \,dx\Bigr)^{\frac{1}{2}}. \label{int-parts-1} 
	\end{align}
	We now restrict our attention to values
        \begin{equation}
          \label{eq:t-cond}
1> t > \frac{2\sigma-2}{2 \sigma-1}
        \end{equation}
and choose, specifically,
        \begin{equation}
          \label{eq:tau-def}
\tau  =\frac{t}{2(1-t)(\sigma-1)} 
\end{equation}
which satisfies $\tau>1$ by (\ref{eq:t-cond}) and $2(1-t) \tau = \frac{t}{\sigma-1}.$
        Therefore (\ref{int-parts-1}) reduces to
        \begin{align}
          &\int_{\R^N}|x_1|^{-t}|u|^{m}\,dx           \label{eq:int-parts-2} \\
&\leq c \Bigl(\int_{\R^N}|x_1|^{\frac{t}{\sigma-1}}|u|^{2\theta(m-1)\tau}\,dx \Bigr)^\frac{1}{2\tau} \Bigl( \int_{\R^N} |u|^{2(1-\theta)(m-1)\tau'} \, dx \Bigr)^\frac{1}{2\tau'}
	\Bigl( \int_{\R^N} |\partial_1 u|^2 \,dx\Bigr)^{\frac{1}{2}}. \nonumber    
      \end{align}
Next we define
\begin{equation}
  \label{eq:def-m}
  m:=   \frac{p(\sigma-1)(\tau-1) + \sigma p- 1}{2\tau (\sigma-1) +1} +1
  = \frac{p(\sigma-1)(\tau-1) + \sigma p +2\tau (\sigma-1)}{2\tau (\sigma-1) +1}
\end{equation}
and
\begin{equation}
\label{eq-def-theta}
        \theta = \frac{(m-1)-\frac{p}{2 \tau'}}{m-1} .
      \end{equation}
      A short computation shows that these values are chosen such that the conditions
      \begin{equation}
        \label{eq:m-theta-con}
2 \theta(m-1) \tau = p \sigma' - \frac{m}{\sigma-1} \qquad \text{and}\qquad
2 (1-\theta)(m-1) \tau' = p
      \end{equation}
  hold for the exponents in (\ref{eq:int-parts-2}). In order to use the 
  inequalities with these values of $\theta$ and $m$, we have to ensure that 
  these values are admissible in the sense of (\ref{eq:admissibility-con}). By 
  definition, we have $m >1$. Moreover, we note that $m < \sigma p$ since  
$$
\sigma \ge 1 \ge \frac{1}{2\tau'} + \frac{1}{p},\qquad \text{i.e.,}\qquad p(\tau-1) +2\tau 
\le 2 \sigma p \tau ,
$$
and hence
$$
p(\sigma-1)(\tau-1) + \sigma p +2\tau (\sigma-1)
\le \sigma p \bigl(2\tau (\sigma-1) +1\bigr). 
$$
Hence $m \in (1,\sigma p)$, as required. Moreover, we have $\theta < 1$ by definition. To see that $\theta>0$,
we note that, since $p >2$, we have $\tau' > 1 \ge \frac{p}{2(p -1)} \ge \frac{p}{2(\sigma p -1)}$
and therefore 
$$
2 \Bigl(p(\sigma-1)\tau + \tau'(\sigma p- 1)\Bigr) > p \Bigl(2\tau (\sigma-1) +1\Bigr), 
$$
which shows that 
$$
2(m-1)\tau' = 2 \frac{p(\sigma-1)\tau + \tau'(\sigma p- 1)}{2\tau (\sigma-1) +1}
> p .
$$
Consequently, $\theta>0$, and thus $\theta \in (0,1)$, as required in (\ref{eq:admissibility-con}). 
So we may consider these values of $\tau$, $m$ and $\theta$ in \eqref{eq:hoelder-1-1} and \eqref{eq:int-parts-2}. With  (\ref{eq:m-theta-con}), this yields the inequalities
        \begin{equation*}
	\int_{\R^N} |u|^p\,dx \le \Bigl(\int_{\R^N}|x_1|^{\frac{t}{\sigma-1}} |u|^{q}\,dx\Bigr)^{\frac{1}{{\sigma}'}} \Bigl(\int_{\R^N}|x_1|^{-t}|u|^{m}\,dx\Bigr)^{\frac{1}{\sigma}}
	\end{equation*}         
and 
        \begin{equation*}
\int_{\R^N}|x_1|^{-t}|u|^{m}\,dx         
\leq c \Bigl(\int_{\R^N}|x_1|^{\frac{t}{\sigma-1}}|u|^{q}\,dx \Bigr)^{\frac{1}{2\tau}} \Bigl( \int_{\R^N} |u|^{p} \, dx \Bigr)^{\frac{1}{2\tau'}}
	\Bigl( \int_{\R^N} |\partial_1 u|^2 \,dx\Bigr)^{\frac{1}{2}}
      \end{equation*}
      with
      \begin{equation}
        \label{eq:def-q}
q:= 2 \theta (m-1) \tau = p \sigma' - \frac{m}{\sigma-1} .
      \end{equation}
Combining these inequalities yields
   \begin{equation*}
	\int_{\R^N} |u|^p\,dx \le c \Bigl(\int_{\R^N}|x_1|^{\frac{t}{\sigma-1}} |u|^{q}\,dx\Bigr)^{\frac{1}{{\sigma}'}+\frac{1}{2\tau \sigma}} \Bigl( \int_{\R^N} |u|^{p} \, dx \Bigr)^{\frac{1}{2\tau' \sigma}} 	\Bigl( \int_{\R^N} |\partial_1 u|^2 \,dx\Bigr)^{\frac{1}{2\sigma}}
	\end{equation*}              
        and therefore
   \begin{equation}
	\label{eq:hoelder-combined-reduced}
        		\int_{\R^N} |u|^p\,dx 
		\le c \Bigl(\int_{\R^N}|x_1|^{\frac{t}{\sigma-1}} |u|^{q}\,dx\Bigr)^{\frac{2\sigma \tau-2\tau+1}{2\sigma \tau-\tau+1}}\Bigl( \int_{\R^N} |\partial_1 u|^2 \,dx\Bigr)^{\frac{\tau}{2\sigma\tau-\tau+1}}.
  \end{equation}
  To obtain (\ref{eq:claim-ineq-lemma}), it is convenient to set $\alpha:= \frac{t}{\sigma-1}>0$, noting that the admissibility condition    \eqref{eq:t-cond} translates to
\begin{equation}
\label{eq:alpha-cond}
            \frac{1}{\sigma-1}> \alpha > \frac{2}{2 \sigma-1} .
        \end{equation}
        Note that, if $\alpha>0$ is given, we always find $\sigma \in (1,\infty)$ with the property that \eqref{eq:alpha-cond} holds.
        Moreover, the exponents in (\ref{eq:hoelder-combined-reduced}) then satisfy
	$$
	{\frac{\tau}{2\sigma\tau-\tau+1}}= \frac{\alpha}{2+\alpha},\qquad {\frac{2\sigma \tau-2\tau+1}{2\sigma \tau-\tau+1}}= \frac{2}{2+\alpha},
	$$
        so (\ref{eq:hoelder-combined-reduced}) becomes
$$        
\int_{\R^N}|u|^{p}\,dx  \le c  \Bigl(\int_{\R^N}|x_1|^{\alpha} 
|u|^{q}\,dx\Bigr)^{\frac{2}{2+\alpha}} \Bigl( \int_{\R^N} |\partial_1 u|^2 
\,dx\Bigr)^{\frac{\alpha}{2+\alpha}} .
$$
This is already the inequality in (\ref{eq:claim-ineq-lemma}). So it only remains to show that the two definitions of $q$ given in (\ref{eq:def-q}) and (\ref{eq:def-q}) are consistent, i.e., we have the identity 
$$
2 \theta(m-1) \tau =\frac{p(2+\alpha)-2\alpha}{2}
$$
The latter follows by a somewhat tedious but straightforward computation, so 
the proof of the lemma is complete.
\end{proof}
We may now complete the proof of the main result of this section, given as follows.
\begin{theorem} 
	\label{general-degenerate-Sobolev-inequality-whole space}
	Let $s >0$ and $2^*_s= \frac{4N+2s}{2N-4+s}$ as in Theorem~\ref{existence of minimizers - whole space - intro}.
	Then inequality (\ref{S_R positive-ineq}) holds with some constant $C>0$.
\end{theorem}
We remark that this may be proven by combining  the previous results with a 
suitable adaption of the inequality on the halfspace given in \cite[Theorem 
1.7]{Filippas et al} to the setting of the entire space $\R^N$.
For the convenience of the reader, we give a self-contained proof.
\begin{proof}
  In the following, the letter $c>0$ stands for a constant which may change from line to line.
	Let $\alpha = \frac{s}{2(N-1)}$. Then Lemma~\ref{claim-ineq-lemma} yields
	$$
	\int_{\R^N}|u|^{2_s^*}\,dx  \le \kappa  \Bigl(\int_{\R^N}|x_1|^{\alpha} |u|^{q_s}\,dx\Bigr)^{\frac{2}{2+\alpha}} \Bigl( \int_{\R^N} |\partial_1 u|^2 \,dx\Bigr)^{\frac{\alpha}{2+\alpha}} \qquad \text{for $u \in C^1_c(\R^N)$}
	$$
	with $q_s:= \frac{N(2_s^*+2)}{2(N-1)}$. To estimate the term $\int_{\R^N} |x_1|^\alpha |u|^{q_s}\,dx$, we define, for $i=1,\dots,N$, the functions $a_i \in C_c(\R^{N-1})$ by
        $$
        a_i (\hat x_i):= \int_{\R} |u|^{\frac{q(N-1)}{N}-1} |\del_i u| \, dx_i 
        $$
        where 
	$$
	\hat x_i  :=(x_1, \ldots,x_{i-1},x_{i+1},\ldots,x_N) \in \R^{N-1}\qquad 
	\text{for $x \in \R^N$ and $i=1,\ldots,N$.}
	$$
	Integrating the derivative $\partial_i |u|^{\frac{q_s(N-1)}{N}}$ in the $x_i$-direction, we find that 
	$|u(x)|^{\frac{q_s(N-1)}{N}} \leq c  a_i (\hat x_i)$ for all $x \in \R^N$, 
	$i=1,\ldots,N$ and therefore
        $$
        |u(x)|^{q_s(N-1)} \le c \prod_{i=1}^{N} a_i(\hat x_i) \qquad \text{for $x \in \R^N$.}
	$$
        Applying Gagliardo's Lemma~\cite[Lemma 4.1]{Gagliardo} to the functions 
        $a_1^{\frac{1}{N-1}}, \dots,a_{N-1}^{\frac{1}{N-1}}$ and $x \mapsto 
        |x_1|^\alpha a_N^{\frac{1}{N-1}}(x)$, we thus find that 
	\begin{align*}
	&\int_{\R^N} |x_1|^\alpha |u|^{q_s}\,dx \le c \left( \int_{\R^{N-1}} |x_1|^{(N-1)\alpha} a_N(\hat x_N) d\hat x_N  \prod_{{i=1}}^{N-1} \int_{\R^{N-1}} a_i(\hat x_i)  d \hat x_i \right)^{\frac{1}{N-1}} \\
	&= c \left( \int_{\R^{N}} |x_1|^\frac{s}{2}|u|^{\frac{q_s(N-1)}{N}-1} |\del_N u| d x  \prod_{{i=1}}^{N-1} \int_{\R^{N}} |u|^{\frac{q_s(N-1)}{N}-1} |\del_i u|dx \right)^\frac{1}{N-1} \\
	&\le c \left(\int_{\R^N} |u|^{2\frac{q_s(N-1)}{N}-2} \,dx \right)^\frac{N}{2(N-1)}  \left( \int_{\R^N} |x_1|^s |\del_N u|^2 \, dx  \prod_{{i=1}}^{N-1} \int_{\R^N}  |\del_i u|^2 \, dx \right)^\frac{1}{2(N-1)}  .
	\end{align*}
	Since $\frac{2(N-1)q_s}{N} -2 = 2_s^*$, we conclude that 
	\begin{align*}
	&\int_{\R^N} |u|^{2_s^*}\,dx \le 
	c 	\Bigl(\int_{\R^N} |x_1|^\alpha |u|^{q_s}\,dx\Bigr)^\frac{2}{2+\alpha}  \Bigl(\int_{\R^N} |\del_1 u|^2 \,dx\Bigr)^{\frac{\alpha}{2+\alpha}}  \\
          &\le c \Biggl(\left(\int_{\R^N} |u|^{2_s^*} \,dx \right)^\frac{N}{2(N-1)}  \left( \int_{\R^N} |x_1| |\del_N u|^2 \, dx \prod_{{i=1}}^{N-1} \int_{\R^N}  |\del_i u|^2 \, dx \right)^\frac{1}{2(N-1)} \Biggr)^{\frac{2}{2+\alpha} }\\
         &\:\times  
	\Bigl(\int_{\R^N} |\del_1 u|^2 \,dx\Bigr)^{\frac{\alpha}{2+\alpha}}\\
	&= c \left(\int_{\R^N} |u|^{2_s^*} \,dx \right)^\frac{N}{2(N-1)+\frac{s}{2}}  \left( \int_{\R^N} |x_1| |\del_N u|^2 \, dx   \prod_{{i=2}}^{N-1}  \int_{\R^N}  |\del_i u|^2 \, dx \right)^\frac{1}{2(N-1)+\frac{s}{2}}\\  
	&\:\times \Bigl(\int_{\R^N} |\del_1 u|^2 \,dx\Bigr)^\frac{1+\frac{s}{2}}{2(N-1)+\frac{s}{2}}
	\end{align*}
	and therefore 
	\begin{align*}
	&\Bigl(\int_{\R^N} |u|^{2_s^*}\,dx\Bigr)^\frac{N-2+\frac{s}{2}}{2(N-1)+\frac{s}{2}} \\
	\le& c  \Bigl( \int_{\R^N} |x_1| |\del_N u|^2 \, dx  \prod_{{i=2}}^{N-1} \int_{\R^N}  |\del_i u|^2 \, dx \Bigr)^\frac{1}{2(N-1)+\frac{s}{2}}
	\Bigl(\int_{\R^N} |\del_1 u|^2 \,dx\Bigr)^\frac{1+\frac{s}{2}}{2(N-1)+\frac{s}{2}} .
	\end{align*}
	Finally, Young's inequality gives
	\begin{align*}
	\Bigl(\int_{\R^N} |u|^{2_s^*}\,dx\Bigr)^\frac{2}{2_s^*} &\leq c \Bigl( \int_{\R^N} |x_1| |\del_N u|^2 \, dx   \prod_{{i=2}}^{N-1}  \int_{\R^N}  |\del_i u|^2 \, dx \Bigr)^\frac{2}{2N+s}   \Bigl(\int_{\R^N} |\del_1 u|^2 \,dx\Bigr)^\frac{2+s}{2N+s} \\
	&\leq  c \Bigl(  \int_{\R^N} 
	|x_1| |\del_N u |^2 \,dx + \sum \limits_{i=1}^{N-1} \int_{\R^N} |\del_i u|^2 \,dx \Bigr) .
	\end{align*}
\end{proof}
In particular, this implies
$$
\cS_s(\R^N_+) = \inf_{ u \in C^1_c(\R^N_+)} 
	\frac{\displaystyle\int_{\R^N_+} \sum \limits_{i=1}^{N-1} |\del_i u|^2 +  
		x_1^s |\del_N u|^2 \, dx}{\left(\int_{\R^N_+} |u|^{2_s^*} \, 
		dx\right)^\frac{2}{2_s^*}} > 0
$$
and thus the first part of Theorem~\ref{existence of minimizers - whole space - 
intro}.
\begin{remark} \label{Remark: Optimality of exponents}
	{\bfseries(Optimality and Variants)} 
	\begin{itemize}
		\item[(i)]
	The exponent $2_s^*$ in (\ref{eq:def-S_n-intro}) is optimal in the sense 
	that
        \begin{equation}
         \label{Remark: Optimality of exponents-eq}
	\inf_{u \in C^1_c(\R^N)} \frac{\int_{\R^N} \Bigl(\sum \limits_{i=1}^{N-1}|\del_i u|^2 +   |x_1|^{s} |\del_N u|^2  \Bigr) \, dx}{\|u\|_{L^p(\R^N)}^2} = 0 \qquad \text{for $p \neq 2_s^*$.}
          \end{equation}
        This follows by considering the rescaling $u \mapsto u_\lambda$, $\lambda>0$ with 
	$$
	u_\lambda(x):=u(\lambda x_1, \lambda x_2, \ldots, \lambda x_{N-1},\lambda^{1+\frac{s}{2}} x_N).
	$$
Indeed, for $u \in C_c^1(\R^N)$ we have 
$$
\int_{\R^N_+}\Bigl( \sum_{i=1}^{N-1}|\del_i u_\lambda|^2 + x_1^s  |\del_N u_\lambda|^2 \Bigr) dx = \lambda^{-\frac{2N+s-4}{2}} \int_{\R^N_+}\Bigl( \sum_{i=1}^{N-1} |\del_i v|^2 + x_1^s  |\del_N u|^2 \Bigr) dx
$$
and, for $1 < p< \infty$, 
$$
\Bigl( \int_{\R^N_+} |u_\lambda|^{p} \, dx \Bigr)^\frac{2}{p} = \lambda^{-\frac{2}{p}(N+\frac{s}{2})}\Bigl( \int_{\R^N_+} |u|^{p} \, dx \Bigr)^\frac{2}{p}.
$$
Since $\frac{2N+s-4}{2} =\frac{2}{p}(N+\frac{s}{2})$ if and only if $p=2_s^*$, 
(\ref{Remark: Optimality of exponents-eq}) follows.
%
%
\item[(ii)] For $\kappa>0$, $u \in C_c^1(\R^N)$, we may consider a rescaled 
function of the form
$$
v(x)=u\left(x_1,\ldots,x_{N-1},\frac{x_n}{\sqrt{\kappa}}\right) .
$$
Comparing the associated quotients then yields
\begin{equation} \label{eq:Degenerate Inequality Scaling}
\begin{aligned} 
 \inf_{u \in C^1_c(\R^N)} \frac{\int_{\R^N} \Bigl(\sum 
\limits_{i=1}^{N-1}|\del_i u|^2 + \kappa |x_1|^{s} |\del_N u|^2  \Bigr) \, 
dx}{\|u\|_{L^{2^*_s}(\R^N)}^2}  = & {\kappa^{\frac{1}{2}-\frac{1}{2^*_s}}} 
\cS_s(\R^N_+).
\end{aligned}
\end{equation}
In the special case $\kappa = 2$, this quotient will appear later when we 
connect $ 
\scrC_{1,m,{2^*_{\scaleto{1}{3pt}}}}(\B)$ and $\cS_1(\R^N_+)$, in particular in 
the proof of Theorem~\ref{Theorem: Minima comparison - introduction}.
\end{itemize}
\end{remark}
%
%
%
Recalling the space $H_s$ defined in Theorem~\ref{existence of minimizers - whole space - intro}, we see that Theorem~\ref{general-degenerate-Sobolev-inequality-whole space}
immediately implies that $H_s$ is continuously embedded into 
$L^{2_s^*}(\R^N_+)$. 
\subsection{Existence of minimizers}
In the following, we fix $s>0$ and study minimizing sequences for
		$$
		\cS := \cS_s(\R^N_+)= \inf_{ u \in H_s \setminus \{0\}} 
		\frac{\displaystyle\int_{\R^N_+} \left( \sum \limits_{i=1}^{N-1} 
		|\del_i u|^2 + 
		x_1^s |\del_N u|^2 \right) \, dx}{\left(\int_{\R^N_+} |u|^{2_s^*} \, 
		dx\right)^\frac{2}{2_s^*}} > 0.
		$$
First, consider the following classical lemma due to Lions~\cite{Lions}, which 
we give in the form presented in \cite{Struwe}:
\begin{lemma} \label{Concentration-Compactness Lemma}
	{\bf (Concentration-Compactness Lemma)} \\
	Suppose $(\mu_n)_n$ is a sequence of probability measures on $\R^N$. Then, after passing to a subsequence, one of the following three conditions holds:
	\begin{enumerate}[font=\itshape]
		\item[(i)]
		{\bfseries(Compactness)}
		There exits a sequence $(x_n)_n \subset \R^N$ such that for any $\eps>0$ there exists $R>0$ such that
		$$
		\int_{B_R(x_n)} \, d\mu_n \geq 1-\eps .
		$$
		\item[(ii)] 
		{\bfseries(Vanishing)} 
		For all $R>0$ it holds that
		$$
		\lim_{n \to \infty} \left(\sup_{x \in \R^N} \int_{B_R(x)} \, d\mu_n \right) =0.
		$$
		\item[(iii)]
		{\bfseries(Dichotomy)} 
		There exists $\lambda \in (0,1)$ such that for any $\eps>0$ there exists $R>0$ and $(x_n)_n \subset \R^N$ with the following property:
		Given $R'>R$ there are nonnegative measures $\mu_n^1, \mu_n^2$ such that
		\begin{align*}
		& 0 \leq \mu_n^1 + \mu_n^2 \leq \mu_n \\
		& \supp \, \mu_n^1 \subset B_R(x_n), \quad \supp \, \mu_n^2 \subset \R^N \setminus B_{R'}(x_n) \\
		& \limsup_{n \to \infty} \left( \left| \lambda - \int_{\R^N} d\mu_n^1 \right| +  \left|(1- \lambda) - \int_{\R^N} d\mu_n^2 \right|\right) \leq \eps .
		\end{align*}
	\end{enumerate}
\end{lemma}
A characterization of minimizing sequences in the sense of measures is given in 
the following lemma, which is a straightforward adaption of \cite[Lemma 
4.8]{Struwe}:	
\begin{lemma} \label{Concentration-Compactness Lemma II}
	{\bf (Concentration-Compactness Lemma II)} \\
	Let $s>0$ and suppose $u_n \weakto u$ in $H_s$ and $\mu_n := \left(\sum 
	\limits_{i=1}^{N-1}|\del_i u_n|^2 + x_1^s |\del_N u_n|^2 \right)  dx 
	\weakto \mu$, $\nu_n := |u_n|^{2_s^*} dx \weakto\nu$ weakly in the sense 
	of measures where $\mu$ and $\nu$ are finite measures on $\R^N_+$. Then: 
\begin{itemize}
	\item[(i)]
	There exists an at most countable set $J$, a set $\{x^j: j \in J\} \subset \R^N_+$ and $\{\nu^j: j \in J\} \subset (0,\infty)$ such that
	$$
	\nu = |u|^{2_s^*} dx + \sum \limits_{ j \in J} \nu^j \delta_{x^j} .
	$$
	\item[(ii)]
	There exists a set $\{\mu^j: j \in J\} \subset (0,\infty)$ such that
	$$
	\mu \geq  \left(\sum \limits_{i=1}^{N-1}|\del_i u|^2 +  x_1^s |\del_N 
	u|^2 \right)  dx  + \sum \limits_{j \in J} \mu^j \delta_{x^j}
	$$
	where
	$$
	\cS (\nu^j)^\frac{2}{2_s^*} \leq \mu^j
	$$
	for $j \in J$. In particular, $\sum \limits_{j \in J} (\nu^j)^\frac{2}{2_s^*} < \infty$.
\end{itemize}
\end{lemma}
Our main result then states that $\cS$ is attained in $H_s$ and completes the 
proof of Theorem~\ref{existence of minimizers - whole space - intro}.
\begin{theorem} \label{Theorem: Existence of Minimizers on Halfspace}
	Let $s>0$ and suppose $(u_n)_n$ is a minimizing sequence for 
	$$
	\cS = \inf_{ u \in H_s \setminus \{0\}} \frac{\displaystyle\int_{\R^N_+} 
	\left( \sum \limits_{i=1}^{N-1} |\del_i u|^2 +   x_1^s |\del_N u|^2 
	\right) 
	\, dx}{\left(\int_{\R^N_+} |u|^{2_s^*} \, dx\right)^\frac{2}{2_s^*}} 
	$$
	with $\|u_n\|_{L^{2_s^*}}=1$. Then, up to translations orthogonal to $x_1$ and anisotropic scaling, $(u_n)_n$ is relatively compact in $H_s$. 
\end{theorem}
\begin{proof}
	For $r>0$ we define the family of rectangles
	$$
	\cQ_r := \left\{ (0,r^2) \times \left(y + (-r^2,r^2)^{N-2}\times (-r^{2+s},r^{2+s})  \right): y \in \R^{N-1} \right\}.
	$$
	It is important to note that for $R>0$, with respect to the transformation
	\begin{equation} \label{anisotropic scaling}
	\tau_R(x)=(R^2 x_1, R^2 x_2,\ldots,R^2 x_{N-1},R^{2+s} x_N) ,
	\end{equation}
	these sets satisfy
	$$
	\tau_R (\cQ_r)= \cQ_{rR} .
	$$
	Moreover, the functions
	$$
	Q_n(r):=\sup_{E \in \cQ_r} \int_{E} |u_n|^{2_s^*} \, dx 
	$$
	are continuous on $[0,\infty)$ and satisfy
	$$
	\lim_{r \to 0} Q_n(r)=0, \quad \lim_{r \to \infty} Q_n(r)=1 .
	$$
	Hence we may choose $A_n>0$, $y_n \in \R^{N-1}$ such that the rescaled 
	sequence $v_n \in H_s$ given by
	$$
	v_n(x):=A_n^\frac{2N-4+s}{2} u_n(A_n^2 x_1,A_n^2 (x_2+(y_n)_1), \ldots,A_n^{2+s} (x_N+(y_n)_{N-1})
	$$
	satisfies
	$$
	Q_n(1)=\sup_{E \in \cQ_1} \int_{E} |v_n|^{2_s^*} \, dx =  \int_{(0,1) \times (-1,1)^{N-1}} |v_n|^{2_s^*} \, dx = \frac{1}{2} .
	$$
	After passing to a subsequence, we may assume $v_n \weakto v$ in $H_s$ and 
	in $L^{2_s^*}(\R^N_+)$. We now consider the measures
	$$
	\mu_n :=  \left(\sum \limits_{i=1}^{N-1} |\del_i v_n|^2 +  x_1^s |\del_N 
	v_n|^2 \right) \, dx, \quad \nu_n := |v_n|^{2_s^*} \, dx 
	$$
	and apply Lemma~\ref{Concentration-Compactness Lemma} to $(\nu_n)_n$, where 
	we note that $\mu_n$ and $\nu_n$ are initially measures on $\R^N_+$ but can 
	trivially be extended to $\R^N$. By our normalization, vanishing cannot 
	occur. We assume that we have dichotomy and thus let $\lambda \in (0,1)$ be 
	as in Lemma~\ref{Concentration-Compactness 
	Lemma}(iii). Then, considering a sequence $\eps_n \downarrow 0$, for any $n 
	\in \N$
	there exist $R_n>0$, $x_n \in \R^N_+$ as well as nonnegative 
	measures $\nu_n^1, \nu_n^2$ on $\R^N_+$ such that
	\begin{align*}
	& 0 \leq \nu_n^1 + \nu_n^2 \leq \nu_n \\
	& \supp \, \nu_n^1 \subset \R^N_+ \cap B_{R_n}(x_n), \quad \supp \, \nu_n^2 
	\subset \R^N_+ \setminus B_{2R_n^\frac{2+s}{2}+1}(x_n) \\
	&  \left| \lambda - \int_{\R^N_+} d\nu_n^1 \right| +  \left|(1- \lambda) - 
	\int_{\R^N_+} d\nu_n^2 \right| \leq 2\eps_n 
	\end{align*}
	and thus
	\begin{align*}
	& \limsup_{n \to \infty} \left( \left| \lambda - \int_{\R^N_+} d\nu_n^1 
	\right| +  \left|(1- \lambda) - \int_{\R^N_+} d\nu_n^2 \right|\right) =0 .
	\end{align*}
	From the proof of the Lemma~\ref{Concentration-Compactness Lemma} (see 
	\cite{Struwe}) we can assume $R_n \to \infty$ and, in particular, $R_n \geq 
	1$.
	
	For $r>0$, let the anisotropic scaling $\tau_r$ be defined as in \eqref{anisotropic scaling}. 
	We crucially note that
	$$
	B_{R_n}(0) \subset \tau_{\sqrt{R_n}}(B_1(0)) 
	$$
	and 
	$$
	\R^N_+ \setminus B_{2 R_n^\frac{2+s}{2}+1}(0) \subset \R^N_+ \setminus 
	\tau_{\sqrt{R_n}}(B_2(0)) .
	$$
	We take $\phi \in C_c^\infty(B_2(0))$ with $0 \leq \phi \leq 1$ and $\phi \equiv 1 $ in $B_1(0)$. For $n \in \N$, let $\phi_n(x):=\phi(\tau_{\sqrt{R_n}}^{-1}(x-x_n))$, so that
	$$
	\phi_n \equiv 1 \quad \text{on $ x_n+\tau_{\sqrt{R_n}}(B_1(0))$}, \quad \phi_n \equiv 0 \quad  \text{on $\R^N \setminus (x_n+\tau_{\sqrt{R_n}}(B_2(0)))$},
	$$
	and thus, in particular,
	$$
	\phi_n \equiv 1 \quad \text{on $\supp \, \nu_n^1$}, \quad \phi_n \equiv 0 \quad  \text{on $\supp \, \nu_n^2$}.
	$$
	Note that
	$$
	|\del_1 v_n|^2 +  x_1^s |\del_2 v_n|^2 \geq \left( |\del_1 v_n|^2 +  
	x_1^s |\del_2 v_n|^2 \right) \left( \phi_n^2 + (1-\phi_n)^2 \right) .
	$$
	We have
	\begin{align*}
	& \left( \int_{\R^N_+}  \left( \sum \limits_{i=1}^{N-1} |\del_i (\phi_n 
	v_n)|^2 +  x_1^s |\del_N (\phi_n v_n)|^2 \right) \, dx 
	\right)^\frac{1}{2} 
	\\
	\leq
	& \left( \int_{\R^N_+} \phi_n^2 \left( \sum \limits_{i=1}^{N-1} |\del_i 
	v_n|^2 + x_1^s |\del_N v_n|^2 \right) \, dx \right)^\frac{1}{2}+ \left( 
	\int_{\R^N_+} v_n^2 \left(\sum \limits_{i=1}^{N-1} |\del_i \phi_n|^2 + 
	 x_1^s |\del_N \phi_n|^2 \right) \, dx \right)^\frac{1}{2}
	\end{align*}
	and analogously for $(1-\phi_n)$ instead of $\phi_n$. 
	Squaring and adding these estimates gives
	\begin{align*}
	& 
	\|\phi_n v_n\|_{H_s}^2 + \|(1-\phi_n) v_n\|_{H_s}^2
	\\
	\leq &  \int_{\R^N_+} \left(\sum \limits_{i=1}^{N-1} |\del_i  v_n|^2 + 
	 x_1^s|\del_N  v_n|^2 \right) \, dx + 2 \int_{\R^N_+} v_n^2 \left(\sum 
	\limits_{i=1}^{N-1} |\del_i \phi_n|^2 +  x_1^s |\del_N \phi_n|^2 \right) 
	\, dx\\
	& +  4 \left(  \int_{\R^N_+} \left(\sum \limits_{i=1}^{N-1} |\del_i  v_n|^2 
	+  x_1^s |\del_N  v_n|^2 \right) \, dx \right)^\frac{1}{2} \left( 
	\int_{\R^N_+} v_n^2 \left( \sum \limits_{i=1}^{N-1} |\del_i \phi_n|^2 + 
	 x_1^s |\del_N \phi_n|^2 \right) \, dx \right)^\frac{1}{2} .
	\end{align*}
	Setting
	\begin{align*}
	\beta_n & := 2 \int_{\R^N_+} v_n^2 \left( \sum \limits_{i=1}^{N-1} |\del_i 
	\phi_n|^2 +  x_1^s |\del_N \phi_n|^2 \right) \, dx \\
	& \ \  +  4 \left(  \int_{\R^N_+} \left( \sum \limits_{i=1}^{N-1} |\del_i  
	v_n|^2 +  x_1^s |\del_N  v_n|^2 \right) \, dx \right)^\frac{1}{2} \left( 
	\int_{\R^N_+} v_n^2 \left( \sum \limits_{i=1}^{N-1} |\del_i \phi_n|^2 + 
	 x_1^s |\del_N \phi_n|^2 \right) \, dx \right)^\frac{1}{2}
	\end{align*}
	we thus have
	\begin{align*}
	&\int_{\R^N_+} \left( \sum \limits_{i=1}^{N-1} |\del_i  v_n|^2 +  x_1^s 
	|\del_N  v_n|^2 \right) \, dx 
	\geq \|\phi_n v_n\|_{H_s}^2 + \|(1-\phi_n) v_n\|_{H_s}^2 - \beta_n .
	\end{align*}
	Next, we define the anisotropic annulus
	$$
	A_n:=x_n + \tau_{\sqrt{R_n}}(B_{2}(0)) \setminus \tau_{\sqrt{R_n}}(B_{1}(0))
	$$
	and consider $\delta>0$.
	Using Young's inequality and the fact that any derivative of $\phi_n$ 
	vanishes outside of $A_n$, we can estimate
	\begin{align*}
	\beta_n \leq & \delta \int_{\R^N_+} \left( \sum \limits_{i=1}^{N-1} 
	|\del_i  v_n|^2 +  x_1^s |\del_N  v_n|^2 \right) \, dx
	+C(\delta) \int_{A_n} v_n^2 \left( \sum \limits_{i=1}^{N-1} |\del_i 
	\phi_n|^2 +  x_1^s |\del_N \phi_n|^2 \right) \, dx.
	\end{align*}
	Note that
	\begin{align*}
\sum \limits_{i=1}^{N-1} |\del_i 
\phi_n|^2 +  x_1^s |\del_N \phi_n|^2  & = R_n^{-2} \sum \limits_{i=1}^{N-1} 
|[\del_i \phi](\tau_n(x))|^2 + x_1^s R_n^{-2-s} |[\del_N 
\phi](\tau_n(x))|^2  
\\
	&= R_n^{-2} \bigg( \sum \limits_{i=1}^{N-1}  |[\del_i \phi]|^2 + 
	(\cdot)_1^s |[\del_N \phi|^2  \bigg) \circ \tau_{\sqrt{R_n}}^{-1} ,
	\end{align*}
	and thus
	$$
	\sum \limits_{i=1}^{N-1} |\del_i \phi_n|^2 +  x_1^s |\del_N \phi_n|^2 
	\leq C R_n^{-2}
	$$
	for some $C>0$ independent of $n$.
	This gives
	\begin{align*}
	\int_{A_n} v_n^2 \left( \sum \limits_{i=1}^{N-1} |\del_i \phi_n|^2 + 
	 x_1^s |\del_N \phi_n|^2 \right) \, dx \leq C R_n^{-2} 
	\|v_n\|_{L^2(A_n)}^2 .
	\end{align*}
	Using H\"older's inequality then further yields
	\begin{align*}
	R_n^{-1} \|v_n\|_{L^2(A_n)} & \leq R_n^{-1} |A_n|^\frac{2}{2N+s} \|v_n\|_{L^{2_s^*}(A_n)} \leq C \|v_n\|_{L^{2_s^*}(A_n)} \\
	& \leq C \left( \int_{\R^N} \, d\nu_n - \left( \int_{\R^N} \, d\nu_n^1 + \int_{\R^N} \, d\nu_n^2 \right) \right)^\frac{1}{2_s^*} \to 0
	\end{align*}
	as $n \to \infty$. Here we used 
	$$
	|A_n|=|\tau_{\sqrt{R_n}}(B_{2}(x_n))|-|\tau_{\sqrt{R_n}}(B_{1}(x_n))| = R_n^\frac{2N+s}{2} \big(|B_2(0)|-|B_1(0)|\big) .
	$$
	Overall, we find that, for any $\delta>0$,
	$$
	\limsup_{n \to \infty}  \beta_n \leq \delta \sup_n \|v_n\|_H^2 ,
	$$
	and since $(v_n)_n$ remains bounded in $H_s$, we conclude
	\begin{align*}
	& \int_{\R^N_+} \left( \sum \limits_{i=1}^{N-1} |\del_i  v_n|^2 + x_1^s 
	|\del_N v_n|^2 \right) \, dx 
	 \geq \|\phi_n v_n\|_{H_s}^2 + \|(1-\phi_n) v_n\|_{H_s}^2 - \beta_n \\
	&\geq  \cS \left( \|\phi_n v_n\|_{L^{2_s^*}(\R^N_+)}^2 + \|(1-\phi_n) 
	v_n\|_{L^{2_s^*}(\R^N_+)}^2 \right) + o(1) \\
	&\geq  \cS \left( \left( \int_{B_{R_n}(x_n)} \, d\nu_n 
	\right)^\frac{2}{2_s^*} +  \left( \int_{\R^N_+ \setminus B_{R_n'}(x_n)} \, 
	d\nu_n \right)^\frac{2}{2_s^*}  \right) + o(1) \\
	&\geq   \cS \left( \left( \int_{\R^N_+} \, d\nu_n^1 \right)^\frac{2}{2_s^*} 
	+  \left( \int_{\R^N_+ } \, d\nu_n^2 \right)^\frac{2}{2_s^*}  \right) + 
	o(1)  \geq  \cS \left(\lambda^\frac{2}{2_s^*} + 
	(1-\lambda)^\frac{2}{2_s^*}\right) + o(1) .
	\end{align*}
	But since $\lambda \in (0,1)$, we have $\lambda^\frac{2}{2_s^*} + (1-\lambda)^\frac{2}{2_s^*} >1$ and thus
	$$
	\begin{aligned} 
	\cS & =\lim_{n \to \infty} \int_{\R^N_+} \left( \sum \limits_{i=1}^{N-1} 
	|\del_i  v_n|^2 +  x_1^s |\del_N v_n|^2 \right) \, dx \\
	& \geq \liminf_{n \to \infty} \left( \cS \left(\lambda^\frac{2}{2_s^*} + 
	(1-\lambda)^\frac{2}{2_s^*}\right) + o(1) \right) > \cS, 
	\end{aligned}
	$$
	a contradiction. Hence we cannot have dichotomy.
	
	Since we are therefore in case (i) of the 
	Lemma~\ref{Concentration-Compactness Lemma}, there exists a sequence 
	$(x_n)_n$ such that for any $\eps>0$ there exists $R=R(\eps)>0$ with
	$$
	\int_{B_R(x_n)} \, d\nu_n \geq 1- \eps .
	$$
	Since we normalized so that
	$$
	\int_{(0,1) \times (-1,1)^{N-1}} |v_n|^{2_s^*} \, dx = \frac{1}{2},
	$$
	we must have $(0,1) \times (-1,1)^{N-1} \cap B_R(x_n) \neq \varnothing$ if $\eps<\frac{1}{2}$.
	By making $R$ larger if necessary, we can thus assume
	$$
	\int_{B_R(0)} \, d\nu_n \geq 1-\eps .
	$$
	In particular, we may therefore pass to a subsequence such that $\nu_n 
	\weakto \nu$ 
	weakly in the sense of measure, where $\nu$ is a finite measure on $\R^N_+$.
	By weak lower (and upper) semicontinuity (of measures), we then have
	$$
	\int_{\R^N_+} \, d\nu =1 .
	$$
	By Lemma~\ref{Concentration-Compactness Lemma II} we may now assume 
	$$
	\mu_n  \weakto \mu \geq \sum \limits_{i=1}^{N-1} \left(|\del_i v|^2 +  
	x_1^s |\del_N v|^2 \right)  dx  + \sum \limits_{j \in J} \mu^j 
	\delta_{x^j} \quad \text{and}\quad 
	\nu_n  \weakto |v|^{2_s^*} dx + \sum \limits_{ j \in J} \nu^j \delta_{x^j}
	$$
	for points $x^j \in \R^N_+$ and positive $\mu^j$, $\nu^j$ satisfying $\cS 
	(\nu^j)^\frac{2}{2_s^*} \leq \mu^j .$ 	We have
	\begin{equation} \label{Concavity inequality}
	\begin{aligned}
	\cS + o(1) &= \|v_n\|_{H_s}^2 = \int_{\R^N_+} \, d\mu_n \geq \int_{\R^N_+} 
	\, 
	d\mu 
	+o(1) \geq  \|v\|_{H_s}^2  + \sum \limits_{j \in J} \mu^j  + o(1) \\
	& \geq \cS \left( \|v\|_{L^{2_s^*}(\R^N_+)}^2 + \sum \limits_j 
	(\nu^j)^\frac{2}{2_s^*} \right) + o(1) \\
	& \geq \cS \left( \|v\|_{L^{2_s^*}(\R^N_+)}^{2_s^*} + \sum \limits_j \nu^j 
	\right)^\frac{2}{2_s^*} + o(1) \\
	&= \cS \left( \int_{\R^N_+} \, d\nu \right)^\frac{2}{2_s^*} + o(1) = S + 
	o(1) 
	\end{aligned}
	\end{equation}
	as $n \to \infty$. In the second inequality, we used the fact that the map $t \mapsto t^\frac{2}{2_s^*}$ is strictly concave and hence subadditive. Moreover, the strict concavity implies that equality can only hold, if at most one of the terms
	$ \|v\|_{L^{2_s^*}(\R^N_+)}^{2_s^*}$ and $ \nu^j, \ j \in J $ is nonzero.
	
	{\bf Claim:} $\nu^j=0$ for all $j$. \\
	Assuming that this is false, we have $\nu_n \weakto \delta_{x^1}$ for some 
	$x^1 \in \overline{\R^N_+}$. By our normalization and weak lower 
	semicontinuity (of measures), $x^1 \not \in Q:= (0,1) \times(-1,1)^{N-1}$ 
	since
	$$
	\delta_{x^1}(Q) \leq \liminf_{n \to \infty} \nu_n(Q)=\frac{1}{2}.
	$$
	Moreover, if $\dist(x^1,Q)>0$, there exists $\eps>0$ such that $B_\eps(x_1) \cap Q \neq \varnothing$ and thus
	$$
	1=\delta_{x^1}(B_\eps(x^1)) \leq \liminf_{n \to \infty} \nu_n(B_\eps(x^1))\leq \frac{1}{2} ,
	$$
	which is a contradiction. Hence it only remains to consider the case $x^1 \in \del Q$.
	Due to the normalization
	$$
	\sup_{E \in \cQ_1} \int_{E} |v_n|^{2_s^*} \, dx =  \int_{(0,1) \times (-1,1)^{N-1}} |v_n|^{2_s^*} \, dx = \frac{1}{2} ,
	$$
	we have $x^1 \not \in ((0,y)+Q)$ for all $y \in \R^{N-1}$, so $x^1$ must be of the form $x^1=( 1, y)$ or $(0,y)$ for some $y \in (-1,1)^{N-1}$. The latter case can be excluded, since, for $\eps \in (0,\frac{1}{2})$, 
	$$
	\delta_{x^1}(B_\eps(0,y)) \leq \liminf_{n \to \infty} \nu_n(B_\eps(0,y)) \leq \liminf_{n \to \infty} \nu_n((0,y)+Q) \leq\frac{1}{2}.
	$$
	After a translation orthogonal to the $x_1$-direction, we may therefore 
	assume $x^1=(1,0, \ldots,0)$ and first note that $v \equiv 0$ and hence 
	$\mu \geq \cS \delta_{x^1}$ by \eqref{Concavity inequality}. On the other 
	hand,
	$$
	\int_{\R^N} \, d\mu \leq \liminf_{n \to \infty} \int_{\R^N} \, d\mu_n=\cS ,
	$$
	whence we conclude $\mu = \cS \delta_{x^1}$.
	
	For any $0<\delta< \frac{1}{2}$, $B_\delta:=B_\delta(x_1)$ is a continuity set of $\nu=\delta_{x_1}$, hence
	$$
	\nu_n(B_\delta) \to 1
	$$
	and similarly
	$$
	\mu_n(B_\delta) \to S
	$$
	as $n \to \infty$. In particular, for fixed $\eps>0$ we find $n_0=n_0(\eps,\delta)$ such that 
	$$
	\int_{B_\delta} |v_n|^{2_s^*} \, dx \geq 1-\eps, \qquad \cS-\eps \leq 
	\int_{B_\delta} \left( \sum \limits_{i=1}^{N-1} |\del_i v_n|^2 +  x_1^s 
	|\del_N v_n|^2 \right) \, dx  \leq \cS+ \eps 
	$$
	for $n \geq n_0$. 
	Furthermore, 
	$$
	\frac{1}{1+\delta}\int_{B_\delta}  \left( \sum \limits_{i=1}^{N-1} |\del_i 
	v_n|^2 +  x_1^s |\del_N v_n|^2 \right) \, dx \leq \int_{B_\delta} \sum 
	\limits_{i=1}^N |\del_i v_n|^2  \, dx 
	$$
	and
    $$
    \int_{B_\delta} \sum \limits_{i=1}^N |\del_i v_n|^2  \, dx \leq 
    \frac{1}{1-\delta} \int_{B_\delta}  \left( \sum \limits_{i=1}^{N-1} |\del_i 
    v_n|^2 +  x_1^s |\del_N v_n|^2 \right) \, dx 
    $$
	imply
	$$
	\frac{\cS-\eps}{1+\delta} \leq \int_{B_\delta} \sum \limits_{i=1}^N |\del_i v_n|^2  \, dx \leq \frac{\cS+\eps}{1-\delta}
	$$	
	for $n \geq n_0$.
	It is important to note that the weak convergence $\nu_n \weakto 
	\delta_{x^1}$ implies that, for any $t \in (0,\delta)$ and $q\in 
	(2_s^*,2^*)$, we have
	$$
	\begin{aligned} 
	1 & = \liminf_{n \to \infty} \int_{B_t} |v_n|^{2_s^*} \, dx \leq 
	|B_t|^{1-\frac{2_s^*}{q}} \liminf_{n \to \infty} \left(  \int_{B_t} 
	|v_n|^{q} \, dx \right)^\frac{2_s^*}{q} \\
	&\leq |B_t|^{1-\frac{2_s^*}{q}}  \liminf_{n \to \infty} \left(  
	\int_{B_\delta} |v_n|^{q} \, dx \right)^\frac{2_s^*}{q}.
	\end{aligned}	
	$$
	In particular, this implies
	\begin{equation} \label{Hoelder blowup argument}
	\liminf_{n \to \infty} \left(  \int_{B_\delta} |v_n|^{q} \, dx \right)^\frac{2_s^*}{q} \geq |B_t|^{\frac{2_s^*}{q}-1} ,
	\end{equation}
	and since $t \in (0,\delta)$ was arbitrary, we conclude that $\|v_n\|_{L^q(B_\delta)} \to \infty$ as $n \to \infty$ for any $q \in (2_s^*,2^*)$.
	
	Now let $\phi \in C_c^\infty(\R^N)$ such that $\phi \equiv 1$ on $B_1(0)$ and $\phi \equiv 0$ on $\R^N \setminus B_2(0)$, and set
	$$
	\phi_\delta(x):=\phi \left(\frac{x-x^1}{\delta}\right)
	$$
	so that $\phi_\delta \equiv 1$ on $B_\delta(x^1)$, $\phi \equiv 0$ on $\R^N \setminus B_{2\delta}(x^1)$. Then, by Sobolev's inequality
	\begin{equation} \label{localized inequality}
	\left( \int_{\R^N_+} |\phi_\delta \, v_n|^{q} \, dx \right)^\frac{2}{q}  
	\leq C_q \left(  \int_{\R^N_+} \sum \limits_{i=1}^N |\del_i(\phi_\delta \,  
	v_n)|^2 \, dx  + \int_{\R^N_+} |\phi_\delta \,  v_n|^2  \, dx \right) .
	\end{equation}
	Note that \eqref{Hoelder blowup argument} implies that the left hand side 
	goes to infinity as $n \to \infty$ since
	$$
	\int_{B_\delta} |v_n|^{q} \, dx \leq  \int_{\R^N} |\phi_\delta \, v_n|^{q} \, dx.
	$$
	On the other hand,
	$$
	\int_{\R^N_+} |\phi_\delta \,  v_n|^2   \, dx \leq |B_{2 
	\delta}|^{1-\frac{2}{2_s^*}} \left( \int_{B_{2\delta}}|v_n|^{2_s^*} \, dx 
	\right)^\frac{2}{2_s^*} \leq |B_{2}|^{1-\frac{2}{2_s^*}} ,
	$$
	and, noting that $\nabla \phi_\delta(x) = \delta^{-1} [\nabla \phi](\frac{x-x^1}{\delta})$,
	\begin{align*}
	\left(\int_{\R^N_+} \sum \limits_{i=1}^N |\del_i(\phi_\delta \,  v_n)|^2 \, 
	dx \right)^\frac{1}{2} &\leq   \left(\int_{\R^N_+} \phi_\delta^2 \sum 
	\limits_{i=1}^N |\del_i  v_n|^2 \, dx \right)^\frac{1}{2} 
	+  \left(\int_{\R^N_+} v_n^2 \sum \limits_{i=1}^N |\del_i  \phi_\delta|^2 
	\, dx \right)^\frac{1}{2} \\
	&\leq \left(\int_{B_{2\delta}} \sum \limits_{i=1}^N |\del_i  v_n|^2 \, dx 
	\right)^\frac{1}{2}  + \sqrt{N} \delta^{-1} \|\nabla \phi\|_\infty 
	\left(\int_{ B_{2 \delta} \setminus B_\delta } |v_n|^2\right)^\frac{1}{2} \\
	&\leq \sqrt{  \frac{\cS+\eps}{1-2\delta}} + \sqrt{N}\delta^{-1} \|\nabla 
	\phi\|_\infty |B_{2 \delta} \setminus 
	B_\delta|^{\frac{1}{2}-\frac{1}{2_s^*}} \left(\int_{B_{2 \delta} \setminus 
	B_\delta }  |v_n|^{2_s^*}\right)^\frac{1}{2_s^*} \\
	&\leq  \sqrt{  \frac{\cS+\eps}{1-2\delta}} + \sqrt{N} \delta^{-1} \|\nabla 
	\phi\|_\infty |B_{2 \delta} \setminus 
	B_\delta|^{\frac{1}{2}-\frac{1}{2_s^*}} .
	\end{align*}
	This implies that the right hand side of \eqref{localized inequality} 
	remains bounded as $n \to \infty$, a contradiction. 
	
	We conclude $\nu^j=0$ for all $j$ and hence $\|v\|_{L^{2_s^*}(\R^N_+)}=1$.
	Since $L^{2_s^*}(\R^N_+)$ is uniformly convex, this implies $v_n \to v$ in 
	$L^{2_s^*}(\R^N_+)$. Moreover, since $\|v\|_{H_s}^2 \geq \cS$, weak lower 
	semicontinuity gives $\|v_n\|_{H_s}^2 \to \cS=\|v\|_{H_s}^2$ and hence $v_n 
	\to v$ in ${H_s}$ again by uniform convexity of the Hilbert space ${H_s}$. 
	This completes the proof.
\end{proof}
\begin{remark} {\bf (Existence of minimizers on $\R^N$)} \\
	We note that Theorem~\ref{general-degenerate-Sobolev-inequality-whole 
	space} implies  
	$$
	\cS_s(\R^N) := 	\inf_{ u \in C^1_c(\R^N)} \frac{\displaystyle\int_{\R^N} 
	\sum 
	\limits_{i=1}^{N-1} |\del_i u|^2 + |x_1|^s |\del_N u|^2 \, 
	dx}{\left(\int_{\R^N} |u|^{2_s^*} \, dx\right)^\frac{2}{2_s^*}} > 0. 
	$$
	Consequently, we can look for minimizers in the closure of 
	$C_c^1(\R^N)$ in  
	$$
	\Bigl\{ u \in L^{2_s^*}(\R^N): \int_{\R^N} \sum_{i=1}^{N-1} |\del_i u|^2 + 
	|x_1|^s |\del_N u|^2 \, dx < \infty \Bigr\} .
	$$
    The previous arguments can then easily be adapted to prove the existence of 
    minimizers of $\cS_s(\R^N)$ similar to
	Theorem~\ref{Theorem: Existence of Minimizers on Halfspace}.	
\end{remark}

\section{A degenerate Sobolev inequality on $\B$}
\label{sec:degen-sobol-ineq}

In this section we shall prove the second part of Theorem~\ref{main theorem}, 
namely the properties of $\scrC_{1,m,p}(\B)$ given in 
(\ref{eq:main-theorem-dist}).

We first use the scaling properties discussed in Remark~\ref{Remark: Optimality 
of exponents}(i) to prove the following.
\begin{proposition} \label{Prop: p>pstar}
	Let $p>{2_{\text{\tiny $1$}}^*}$ and $m>-\lambda_1(\B)$. Then $\scrC_{1,m,p}(\B)=0$, i.e.
	$$
	\inf_{u \in C_c^1(\B) \setminus \{0\}} \frac{\|\nabla u\|_2^2 - \|\del_\theta u\|_2^2 + m \|u\|_2^2}{\|u\|_p^2} = 0.
	$$
\end{proposition}
\begin{proof}
Let $\eps>0$. By (\ref{Remark: Optimality of exponents-eq}), there exists $v \in C_c^1(\R^N_+)$ with the property that 
$$
\int_{\R^N_+} \Bigl(\sum \limits_{i=1}^{N-1} |\del_i v|^2 + 2x_1  |\del_N v|^2\Bigr) dx  < \frac{\eps}{2} \, \Bigl( \int_{\R^N_+} |v|^p \, dx \Bigr)^{\frac{2}{p}}.
$$
For $\lambda \in (0,1)$, let 
\begin{equation}
\tau_\lambda : \B \to \R^N_+, \quad \tau_\lambda(x)=(\lambda^{-2} (x_1+1),\lambda^{-2} x_3, \ldots,\lambda^{-2} x_{N-1},  \lambda^{-3} x_2) 
\end{equation}
and set $u_\lambda:=v \circ \tau_\lambda$. If $\lambda$ is chosen sufficiently small, we have $u \in C_c^1(\B)$ and
\begin{align*} 
& \|\nabla u\|_{L^2(\B)}^2 - \|\del_\theta u\|_{L^2(\B)}^2  = \int_\B \Bigl( \sum \limits_{i=1}^{N}|\del_i u|^2 - |x_{1} \del_2 u- x_2 \del_1 u|^2    \Bigr)\, dx  \\
=& \int_\B \left( \sum \limits_{i=1}^{N-1}|\lambda^{-2}[\del_i v] \circ \tau_\lambda|^2 + |\lambda^{-3 } [\del_N v] \circ \tau_\lambda|^2 - \left| x_{1} \lambda^{-3}  [\del_N v] \circ \tau_\lambda - x_2 \lambda^{-2} [\del_{1} v] \circ \tau_\lambda \right|^2    \right)  \, dx  \\
= &  \lambda^{2N+1} \int_{\R^N_+} \left( \sum \limits_{i=1}^{N-1} \lambda^{-4 }|\del_i v |^2 + \lambda^{-6} | \del_N v |^2 - |(\lambda^2 x_{1}-1) \lambda^{-3} \del_N v - \lambda^3 x_2 \lambda^{-2} \del_{1} v |^2 \right)  \, dx 	 \\
= & \lambda^{2N-3}  \int_{\R^N_+} \left(\sum \limits_{i=1}^{N-1}|\del_i v |^2 + 2 x_{1} |\del_N v|^2 \right) dx 	 \\
& +  \lambda^{2N-3}\int_{\R^N_+}\left(- \lambda^2 x_{1}^2 |\del_N v|^2 - 2x_2 
\lambda^2 (\lambda^2 x_{1} -1) \del_{1} v \, \del_N v +\lambda^6 x_{2}^2 
|\del_{1} v|^2  \right)  \, dx, 
\end{align*} 
while
$$
\|u\|_{L^2(\B)}^2  = \lambda^{2N+1} \|v\|_{L^2(\R^N_+)}^2 \qquad \text{and}\qquad 
\|u\|_{L^{p}(\B)}^2 = \lambda^\frac{4N+2}{p} \|v\|_{L^{p}(\R^N_+)}^2 .
$$
We conclude that
\begin{align*}
\scrC _{1,m,p}(\B)& \leq \frac{\|\nabla u\|_{L^2(\B)}^2 - \|\del_\theta u\|_{L^2(\B)}^2 + m \|u\|_{L^2(\B)}^2}{\|u\|_{L^{p}(\B)}^2} \\
 & = \lambda^\frac{p(2N-3)-(4N+2)}{p}\frac{\int_{\R^N_+} \left(\sum \limits_{i=1}^{N-1}|\del_i v |^2 + 2 x_{1} |\del_N v|^2   \right)  \, dx 	}{\|v\|_{L^{p}(\R^N_+)}^2} + o\left(\lambda^\frac{p(2N-3)-(4N+2)}{p}\right) < \eps
\end{align*}
for $\lambda>0$ small enough, since $p>{2_{\text{\tiny $1$}}^*}=\frac{4N+2}{2N-3}$. Recalling that $\eps>0$ was arbitrary, this yields the claim.
\end{proof}

To prove the second assertion on $\scrC_{1,m,p}(\B)$ in (\ref{eq:main-theorem-dist}), we now transfer the information given by Theorem~\ref{existence of minimizers - whole space - intro} in the case $s=1$ to the ball $\B$. To this end, we consider the great circle
\begin{equation} \label{Def: gamma}
	\gamma := \left\{ x \in \del \B: x_3 = \cdots = x_{N}=0 \right\}.
      \end{equation}
      
We have the following key lemma.

\begin{lemma} \label{lemma: inequality close to boundary}
	Let $\eps >0$. Then there exists $\delta>0$ with the property that, for any $x_0 \in \gamma$, 
	$$
	\frac{\int_{\Omega_{x_0,\delta}} \left( |\nabla u|^2 - |\del_\theta u|^2 
	\right) \, dx}{	
	\|u\|_{L^{{2_{\scaleto{1}{3pt}}^*}}(\Omega_{x_0,\delta})}^2} \geq 
	\left(1-\eps\right) 
	2^{\frac{1}{2}-\frac{1}{2^*_{\scaleto{1}{3pt}}}} \cS_1(\R^N_+) \qquad 
	\text{for $u \in 
	C_c^1(\Omega_{x_0,\delta}) \setminus \{0\}$,}
       $$
	where $\cS_1(\R^N_+)$ is given in Theorem~\ref{existence of minimizers - whole space - intro} and  
	\begin{equation} \label{Def: Omega} 
		\Omega_{x_0,\delta} := \B \cap B_\delta(x_0)= \{x \in \B: 
		|x-x_0|<\delta \}.
	\end{equation}

\end{lemma}
\begin{proof}
  We may assume $x_0=e_2= (0,1,0,\dots,0)$ is the second coordinate vector. We fix $\delta>0$ and consider a function $u \in C_c^1(\Omega_{e_2,\delta})$ which we extend trivially to a function $u \in C_c^1(\R^N)$. Moreover, we write $u$ in $N$-dimensional polar coordinates, so we consider $U:= [0,1] \times (-\pi,\pi) \times (0,\pi)^{N-2}$ and the function 
  $$
  v= u \circ P\::\: U \to \R
  $$
  with $P:U \to \R^N$ given by 
  \begin{align}
    P(r,\theta,\vartheta_1,\dots,\vartheta_{N-2})&= (r \cos \theta \sin\vartheta_1 \cdots \sin \vartheta_{N-2},\: r \sin \theta \sin\vartheta_1 \cdots \sin \vartheta_{N-2},\nonumber\\
    &r \cos \vartheta_{1}, r \sin \vartheta_1 \cos \vartheta_2, \dots, r \sin \vartheta_1 \dots \sin \vartheta_{N-3} \cos \vartheta_{N-2},r \sin \vartheta_1 \dots  \vartheta_{N-2} )\label{polar-coordinates}
  \end{align}
  We then have
	\begin{align}
		& \int_{\B} \left( |\nabla u|^2 - |\del_\theta u|^2 \right) \, dx  \label{polar-coordinates-calc}\\
		= & \int_0^1 \int_{-\pi}^{\pi} \int_0^\pi \cdots  \int_0^\pi  \left( 
		|\del_r u|^2 + \frac{1}{r^2} \sum \limits_{i=1}^{N-2} g_i 
		|\del_{\vartheta_i} u|^2 + \left(\frac{g_{N-1}}{r^2}-1\right) 
		|\del_\theta u|^2 \right) g \, d\vartheta_1 \cdots d\vartheta_{N-2} \, 
		d\theta \, dr 
	\end{align}
	with the functions $g,g_i:U \to \R$, $i=1,\dots,N-1$ given by   
        \begin{equation}
          \label{eq:def-g-g-i}
        g(r,\theta,\vartheta_1,\dots,\vartheta_{N-2}) = r^{N-1}  \prod_{k=1}^{N-2} \sin^{N-1-k}\vartheta_k ,\qquad 
		g_i(r,\theta,\vartheta_1,\dots,\vartheta_{N-2}) = \prod_{k=1}^{i-1} \frac{1}{\sin^2 \vartheta_k}
        \end{equation}
In particular, we have $g \le 1$ and $g_i \geq 1$ in $U$ for $i=1,\dots,N-1$. Moreover, since $P^{-1}(e_2)=(1,\frac{\pi}{2},\dots,\frac{\pi}{2})$ and $g(1,\frac{\pi}{2},\dots,\frac{\pi}{2})=1$, we can choose $\delta>0$ sufficiently small so that
\begin{equation}
  \label{eq:sufficient-geometry}
P^{-1}(\Omega_{e_2,\delta}) \subset (0,1) \times (0,\pi)^{N-1} \qquad
\text{and}\qquad g \geq (1-\eps) \; \text{in $P^{-1}(\Omega_{e_2,\delta})$.}
\end{equation}
Therefore
	\begin{align*}
		& \int_{\B} \left( |\nabla u|^2 - |\del_\theta u|^2 \right) \, dx  \\
		\geq & (1-\eps) \int_0^1 \int_{-\pi}^{\pi} \int_0^\pi \cdots  
		\int_0^\pi  \left( |\del_r u|^2 + \sum \limits_{i=1}^{N-2}  
		|\del_{\vartheta_i} u|^2 + \frac{(1-r)(1+r)}{r^2} |\del_\theta u|^2 
		\right)  \, d\vartheta_1 \cdots d\vartheta_{N-2} \, d\theta \, dr .
	\end{align*}
	Noting that
	$$
	\frac{(1-r)(1+r)}{r^2} \geq \frac{(2-\delta)(1-r)}{(1-\delta)^2} \geq 2 (1-r)
	$$		
	and substituting $t=1-r$ we thus find that
	\begin{align*}
		& \int_{\B} \left( |\nabla u|^2 - |\del_\theta u|^2 \right) \, dx  \\
		\geq & (1-\eps) \int_0^1 \int_{-\pi}^{\pi} \int_0^\pi \cdots  
		\int_0^\pi  \left( |\del_t \tilde v|^2 + \sum \limits_{i=1}^{N-2}  
		|\del_{\vartheta_i} \tilde v|^2 + 2t |\del_\theta \tilde v|^2 \right)  
		\, d\vartheta_1 \cdots d\vartheta_{N-2} \, d\theta \, dt .
	\end{align*}
	with
        $$
        \tilde v: U \to \R, \qquad \tilde v(t,\vartheta_1,\ldots,\vartheta_{N-2},\theta)= u(P(1-t,\vartheta_1,\ldots,\vartheta_{N-2},\theta))
        $$

	Note that $u \in C^1_c(\Omega_{e_2,\delta})$ implies, by (\ref{eq:sufficient-geometry}), that $\tilde v$ is compactly supported in $(0,1) \times (0,\pi)^{N-1} \subset \R^N_+$, so we may regard $\tilde v$ as a function in $C_c^1(\R^N_+)$ and deduce that
	$$
	 \int_{\B} \left( |\nabla u|^2 - |\del_\theta u|^2 \right) \, dx \geq (1-\eps) \int_{\R^N_+} \Bigl( \sum \limits_{i=1}^{N-1} |\del_i \tilde v|^2 + 2 x_1 |\del_N \tilde v|^2 \Bigr) \, dx.
	$$
	Rather directly, we also find that, by a change of variables, 
        \begin{align*}
        \int_{\Omega} |u|^{{2_{\text{\tiny $1$}}^*}} \, dx &= \int_{U} 
        |v|^{{2_{\text{\tiny $1$}}^*}} g \, d(r, \theta, 
        \vartheta_1,\dots,d\vartheta_{N-2}) \le
                                            \int_{U} |v|^{{2_{\text{\tiny 
                                            $1$}}^*}} \, d(r, \theta, 
                                            \vartheta_1,\dots,d\vartheta_{N-2})\\
          &= \int_{U} |\tilde v|^{{2_{\text{\tiny $1$}}^*}} \, d(r, \theta, 
          \vartheta_1,\dots,d\vartheta_{N-2})= \int_{\R^N_+} |\tilde 
          v|^{{2_{\text{\tiny $1$}}^*}} \, dx .  
        \end{align*}
    Using \eqref{eq:Degenerate Inequality Scaling} with $\kappa= 2$, we 
    conclude that
	\begin{align*}
		\frac{\int_{\Omega} \left( |\nabla u|^2 - |\del_\theta u|^2 \right) \, dx}{	\|u\|_{L^{{2_{\text{\tiny $1$}}^*}}(\Omega)}^2} & 
		\geq (1-\eps) \frac{\int_{\R^N_+} \Bigl( \sum \limits_{i=1}^{N-1} 
		|\del_i \tilde v |^2 + 2 x_1 |\del_N \tilde v|^2 \Bigr) \, 
		dx}{\Bigl(\int_{\R^N_+} |\tilde v|^{{2_{\text{\tiny $1$}}^*}} \, dx 
		\Bigr)^\frac{2}{{2_{\text{\tiny $1$}}^*}}} \geq 
		(1-\eps)2^{\frac{1}{2}-\frac{1}{2^*_{\scaleto{1}{3pt}}}} \cS_1(\R^N_+)
	\end{align*}
	as claimed. 
\end{proof}
We can now prove the main result of this section. 
\begin{theorem} \label{degenerate Sobolev inequality - ball}
	For any $1 \leq p \leq {2_{\text{\tiny $1$}}^*}$ there exists $C>0$, such that any $u \in C_c^1(\B)$ satisfies
	$$
	\|u\|_{L^{p}(\B)}^2 \leq C \int_{\B} \left( |\nabla u|^2 - |\del_\theta u|^2  \right) \, dx .
	$$
\end{theorem}
\begin{proof}
Since $\B$ is bounded, it suffices to consider the case $p= {2_{\text{\tiny $1$}}^*}$. In the following, $C>0$ denotes a constant independent of $u$, which may change from line to line.
	Fix $\eps \in (0,\frac{1}{2})$ and let $\delta>0$ be given as in 
	Lemma~\ref{lemma: inequality close to boundary}.
    Take points $x_1, \ldots x_m \in \gamma$ such that the sets $U_k:=B_\delta(x_k)$ satisfy 
	$$
	\gamma \subset \bigcup_{k=1}^m U_k 
	$$
	and let $\delta_0 := \dist(\gamma, \B \setminus \bigcup \limits_{k=1}^m U_k)$.
	We then let $U_0:=\left\{x \in \B: \dist(x,\gamma)>\frac{\delta_0}{2} 
	\right\}$ and thus have $\B \subset \bigcup \limits_{k=0}^m U_k$. 
	We may then choose a partition of unity $\eta_0, \cdots, \eta_m$ subordinate to this covering.
	Then
	\begin{align*}
		\|u\|_{L^{{2_{\text{\tiny $1$}}^*}}(\B)} & \leq \sum \limits_{k=0}^m 	\|\eta_k u\|_{L^{{2_{\text{\tiny $1$}}^*}}(U_k)} \leq C \sum \limits_{k=0}^m \left( \int_{U_k} \left( |\nabla (\eta_k u)|^2 - |\del_\theta (\eta_k u)|^2 \right) \, dx \right)^\frac{1}{2} ,
	\end{align*}
   where we used Lemma~\ref{lemma: inequality close to boundary} and the fact that $v \mapsto \int_{U_0} \left( |\nabla v|^2 - |\del_\theta v|^2 \right) \, dx$ induces an equivalent norm on $H^1_0(U_0)$.
	Note that, for $k=0,\ldots,m$, we have
\begin{align*}
	\int_{U_k} \!\Bigl(  |\nabla (\eta_k u)|^2 - |\del_\theta (\eta_k u)|^2 \Bigr)dx 	&\leq 2 \left( \int_{U_k} \!\eta_k^2 \left(|\nabla u|^2 - |\del_\theta  u|^2\right)dx +  \int_{U_k} \!u^2 \left(|\nabla \eta_k|^2 - |\del_\theta \eta_k|^2\right) dx\right) \\
	&\leq  C \int_{U_k} \left( |\nabla u|^2 - |\del_\theta u|^2 +u^2 \right) \, dx ,
\end{align*}
with some fixed $C>0$. We conclude that
	\begin{align*}
		\|u\|_{L^{{2_{\text{\tiny $1$}}^*}}(\B)} &  \leq C \sum \limits_{k=0}^m \left(\int_{U_k} \left(|\nabla u|^2 - |\del_\theta u|^2 +u^2 \right) \, dx \right)^\frac{1}{2} ,
	\end{align*}
	and thus
	\begin{align}  \label{inequality with L2 term}
		\|u\|_{L^{{2_{\text{\tiny $1$}}^*}}(\B)}^2 &  \leq C \sum \limits_{k=0}^m \int_{U_k} \left( |\nabla u|^2 - |\del_\theta u|^2 + u^2 \right) \, dx = C \int_{\B} \left( |\nabla u|^2 - |\del_\theta u|^2 + u^2 \right) \, dx .
	\end{align}
	In order to complete the proof, we note that Proposition~\ref{lambda-1-alpha-pos} implies 
	$$
	\inf_{ u \in C_c^1(\B) \setminus \{ 0 \}} \frac{\int_{\B} \left( |\nabla u|^2 - |\del_\theta u|^2 \right)  \, dx }{\int_{\B} u^2 \, dx}= \lambda_1(\B) >0 
	$$
	and hence
	$$
	\int_{\B} u^2 \, dx \leq \frac{1}{\lambda_1(\B)} \int_{\B} \left( |\nabla u|^2 - |\del_\theta u|^2 \right)  \, dx .
	$$
	In view of \eqref{inequality with L2 term} this finishes the proof.
\end{proof}

\section{The variational setting for and preliminary results on ground state solutions} \label{preliminaries}

\subsection{The variational setting}
In this section, we set up the variational framework for \eqref{Reduced 
equation} and prove some preliminary estimates for the quantities 
$\scrC_{\alpha,0,2}(\B)$ and $R_{\alpha,m,p}$ defined in 
(\ref{eq:intro-minimization-problem}) and (\ref{eq:def-Raleigh-quotient}). We 
first show a Poincar\'e type estimate.
Recall here that $\lambda_1(\B)$ is the first Dirichlet eigenvalue of $-\Delta$ on the unit ball $\B$.

\begin{proposition}
	\label{lambda-1-alpha-pos}
	For $0 \leq \alpha \le 1$, we have 
\begin{equation}
        \label{lambda-1-alpha-pos-eq}
	\scrC_{\alpha,0,2}(\B) = \inf_{u \in C^1_c(\B) \setminus \{0\}} \frac{\int_{\B}\left(|\nabla u|^2 - \alpha^2 |\del_\theta u|^2\right)}{\int_\B u^2 \, dx} = 
	\lambda_1(\B).
\end{equation}
Moreover, minimizers are precisely the Dirichlet eigenfunctions of $-\Delta$ on $\B$ corresponding to the eigenvalue $\lambda_1(\B)$ and are therefore radial.
\end{proposition}

\begin{proof}
  By (\ref{eq-C-intro-new-1-1}) and since $\scrC_{0,0,2}(\B)= \lambda_1(\B)$ by the variational characterization of $\lambda_1(\B)$, it suffices to prove (\ref{lambda-1-alpha-pos-eq}) in the case $\alpha =1$.
	In the following, we let $\{Y_{\ell,k}\::\: \ell \in \N \cup \{0\}, \ 
	k=1,\ldots,d_\ell\}$ denote an $L^2$-orthonormal basis of 
	$L^2(\mathbb{S}^{N-1})$ of spherical harmonics of degree $\ell$. More 
	precisely, we can choose $Y_{\ell,k}$ in such a way that, for every $\ell 
	\in \N \cup \{0\}$, the functions  $Y_{\ell,k},\: k=1\ldots,d_\ell$ form a 
	basis of the eigenspace of the Laplace Beltrami operator 
	$-\Delta_{\mathbb{S}^{N-1}}$ corresponding to the eigenvalue 
	$\ell(\ell+N-2)$ and such that
	$$
	-\del_\theta^2 Y_{\ell,k} = \ell_k^2 Y_{\ell,k} \qquad \text{for $k=1\ldots,d_\ell$} 
	$$
	where $|\ell_k| \leq \ell$, see e.g. \cite{Higuchi}. Let $\phi \in C^1_c(\B)$, and let
	$\phi_{\ell,k} \in C^1([0,1])$ be the angular Fourier coefficient functions defined by
	$$
	\phi_{\ell,k}(r)= 
	\int_{\mathbb{S}^{N-1}}\phi(r\omega)Y_{\ell,k}(\omega)\,d\omega, \qquad 0 
	\le r \le 1.
	$$
For fixed $r \in [0,1]$, we then have the Parseval identities 
\begin{align*}
	\|\phi(r \,\cdot\,)\|_{L^2(\mathbb{S}^{N-1})}^2 &= \sum_{\ell,k} 
	|\phi_{\ell,k}(r)|^2 \|Y_{\ell,k}\|_{L^2(\mathbb{S}^{N-1})}^2,\\
  \|\partial_r \phi(r \,\cdot\,)\|_{L^2(\mathbb{S}^{N-1})}^2 &= \sum_{\ell,k} 
  |\partial_r \phi_{\ell,k}(r)|^2 \|Y_{\ell,k}\|_{L^2(\mathbb{S}^{N-1})}^2,\\
  \|\nabla_{\mathbb{S}^{N-1}} \phi(r \,\cdot\,)\|_{L^2(\mathbb{S}^{N-1})}^2 
  &= \sum_{\ell,k} (\ell +  N-2)|\phi_{\ell,k}(r)|^2                    
  \|Y_{\ell,k}\|_{L^2(\mathbb{S}^{N-1})}^2\quad 
            \text{and}\\
  \|\partial_\theta \phi(r \,\cdot\,)\|_{L^2(\mathbb{S}^{N-1})}^2 &= 
\sum_{\ell,k} \ell_k^2 |\phi_{\ell,k}(r)|^2 
\|Y_{\ell,k}\|_{L^2(\mathbb{S}^{N-1})}^2
\end{align*}
in $L^2(\mathbb{S}^{N-1})$. Here and in the following, we simply write $\sum 
\limits_{\ell,k}$ in place of $\sum \limits_{\ell=0}^\infty \sum 
\limits_{k=1}^{d_\ell}$.
Since $\frac{\ell(\ell+N-2)}{r^2} \ge \ell_k^2$ for $r \in [0,1]$ and every $\ell,k$, we estimate that
	\begin{align*}
	&\int_\B \left( |\nabla \phi|^2 -  |\del_\theta \phi|^2 \right) \, dx \\
	= & \int_0^1 r^{N-1} \int_{\mathbb{S}^{N-1}} \left( |\partial_r \phi(r 
	\omega)|^2 + \frac{1}{r^2}|\nabla_{\mathbb{S}^{N-1}} \phi(r \omega)|^2 - 
	|\del_\theta \phi(r \omega|^2 \right)  d\omega dr\\
	=& \sum_{\ell,k} \|Y_{\ell,k}\|_{L^2(\mathbb{S}^{N-1})}^2 \int_0^1 r^{N-1} 
	\Bigl(|\partial_r \phi_{\ell,k}(r)|^2+ 
	\left(\frac{\ell(\ell+N-2)}{r^2}-\ell_k^2 \right)|\phi_{\ell,k}(r)|^2\Bigr) 
	\,dr\\   
           \ge & \sum_{\ell,k}\|Y_{\ell,k}\|_{L^2(\mathbb{S}^{N-1})}^2 \int_0^1 
           r^{N-1} 
          |\partial_r \phi_{\ell,k}(r)|^2dr \\
          \ge& \lambda_1(\B)\sum_{\ell,k} 
          \|Y_{\ell,k}\|_{L^2(\mathbb{S}^{N-1})}^2 
          \int_0^1 r^{N-1} |\phi_{\ell,k}(r)|^2dr\\
          =& \lambda_1(-\Delta,\B) \int_0^1 r^{N-1} 	\|\phi(r 
          \,\cdot\,)\|_{L^2(\mathbb{S}^{N-1})}^2\,dr = \lambda_1(\B)\int_\B 
          |\phi|^2 \, dx  .
	\end{align*}
	Clearly, equality holds if and only if $\phi_{\ell,k} \equiv 0$ for $\ell \geq 1$ and $\phi_0$ corresponds to a first eigenfunction of the Dirichlet Laplacian on $\B$.
\end{proof}

\begin{corollary}$ $
\label{first-cor-section-2}  
    \begin{itemize}
    \item[(i)] We have $\scrC_{\alpha,m,2}(\B) =\scrC_{0,m,2}(\B)= \lambda_1(\B)+m$ for $\alpha \in [0,1]$, $m \in \R$.
    \item[(ii)] For $\alpha \in [0,1]$, $m > -\lambda_1(\B)$, $2 \le p < 2^*$ and $u \in H^1_0(\B) \setminus \{0\}$ we have $R_{\alpha,m,p}(u) > 0$.  
    \end{itemize}
  \end{corollary}

  \begin{proof}
    (i) This follows immediately from Proposition~\ref{lambda-1-alpha-pos}.\\
    (ii) For $\alpha \in [0,1]$, $m \ge -\lambda_1(\B)$, $2 \le p < 2^*$ and $u \in H^1_0(\B) \setminus \{0\}$ we have $R_{\alpha,m,p}(u)> R_{1,-\lambda_1(\B),p}(u)$ and
    $$
    R_{1,-\lambda_1(\B),p}(u) \left(\int_\B |u|^p \, dx \right)^\frac{2}{p}
    =\left(\int_\B |u|^p \, dx \right)^{-\frac{2}{p}} \int_\B \left( |\nabla u|^2 - |\del_\theta u|^2 - \lambda_1(\B) u^2\right)dx \ge 0 
      $$
    by Proposition~\ref{lambda-1-alpha-pos}.
  \end{proof}

\begin{remark}
  \label{remark-alpha-greater-1}
  {\bfseries(The case $\alpha>1$)} \\
 It is natural to ask what happens for $\alpha>1$. In fact, in this 
  case, the infimum $\scrC_{\alpha,m,p}(\B)$ in 
  (\ref{eq:intro-minimization-problem}) satisfies
  \begin{equation}
    \label{eq:infimum-minus-infinity}
  \scrC_{\alpha,m,p}(\B) = -\infty \qquad \text{for every $m \in \R$, $p \in [2,\infty)$.}
  \end{equation}
To see this, we fix $\eps \in (0,1)$ and nonzero functions $\phi \in 
C_c^1(1-\eps,1))$, $\psi \in C_c^1(\frac{\pi}{2}-\eps,\frac{\pi}{2}+\eps)$. 
Moreover, we consider the sequence of functions $u_k \in C^1_c(\B)$ which, in 
the polar coordinates from (\ref{polar-coordinates}), are given by
$$
(r,\theta,\vartheta_1,\dots,\vartheta_{N-2}) \mapsto \phi(r)\psi(\vartheta_1)\cdots \psi(\vartheta_{N-2}) X_k(\theta), \qquad \text{where $X_k(\theta)=\sin(k \theta)$.}
$$
Similarly as in (\ref{polar-coordinates-calc}), we then find, with $U_\eps:= 
(1-\eps,1) \times (-\pi,\pi) \times (\frac{\pi}{2}-\eps, 
\frac{\pi}{2}+\eps)^{N-2}$, that 
	\begin{align*}
		 \int_{\B} &\left( |\nabla u_k|^2 - \alpha^2|\del_\theta u_k|^2 \right) \, dx\\ 
          &=  \int_{U_\eps}\Bigl( |\phi'(r)|^2|X_k(\theta)|^2 \prod_{i=1}^{N-2}|\psi(\vartheta_i)|^2  + \frac{1}{r^2} \sum \limits_{i=1}^{N-2} g_i |\psi'(\vartheta_i)|^2|\phi(r)|^2|X_k(\theta)|^2 \prod_{\stackrel{j=1}{j \not= i}}^{N-2}|\psi(\vartheta_j)|^2\\
          \; & \quad + \Bigl(\frac{g_{N-1}}{r^2}-\alpha^2 \Bigr) 
          |X_k'(\theta)|^2|\phi(r)|^2 \prod_{i=1}^{N-2}|\psi(\vartheta_i)|^2 
          \Bigr) g \, d(r,\theta,\vartheta_1, \dots ,\vartheta_{N-2}) %
        \end{align*}
	with the functions $g,g_i:U \to \R$, $i=1,\dots,N-1$ given in (\ref{eq:def-g-g-i}). We may now choose $\eps= \eps(\alpha)>0$ so small that
        $$
        \frac{1}{2} \le g \le 1  \quad \text{and}\quad \alpha^2-\frac{g_{N-1}}{r^2} \ge \eps \qquad \text{on $U_\eps$.}
        $$
          Since also $|X_k| \le 1$ by definition, we estimate
         $$
         \int_{\B} \left( |\nabla u_k|^2 - \alpha^2|\del_\theta u_k|^2 \right) \, dx 
         \le c - d(k),
         $$
       where
       $$
           c := \int_{U_\eps}\Bigl( |\phi'(r)|^2 \prod_{i=1}^{N-2}|\psi(\vartheta_i)|^2  + \frac{1}{r^2} \sum \limits_{i=1}^{N-2} g_i |\psi'(\vartheta_i)|^2|\phi(r)|^2 \prod_{\stackrel{j=1}{j \not= i}}^{N-2}|\psi(\vartheta_j)|^2\Bigr)\,d(r,\theta,\vartheta_1, \dots ,\vartheta_{N-2})   
      $$
         and 
        \begin{align*}
          d(k)&:=\int_{U_\eps}\Bigl(\alpha^2-\frac{g_{N-1}}{r^2}\Bigr) |X_k'(\theta)|^2|\phi(r)|^2 \prod_{i=1}^{N-2}|\psi(\vartheta_i)|^2  g \, d(r,\theta,\vartheta_1, \dots ,\vartheta_{N-2})\\   
         &\ge \frac{k^2 \eps}{2} \int_{1-\eps}^1 |\phi(r)|^2dr \int_{-\pi}^{\pi} \cos^2(k \theta) d\theta \Bigl(\int_{\frac{\pi}{2}-\eps}^{\frac{\pi}{2}+\eps}|\psi(\vartheta)|^2 d\vartheta\Bigr)^{N-2} = \frac{\eps \pi}{2} d_2 k^2
        \end{align*}
        with $d_2:=  \int_{1-\eps}^1 |\phi(r)|^2dr  \Bigl(\int_{\frac{\pi}{2}-\eps}^{\frac{\pi}{2}+\eps}|\psi(\vartheta)|^2 d\vartheta\Bigr)^{N-2}$. Hence $d(k) \to \infty$ as $k \to \infty$. Moreover, for every $p \in [2,\infty)$ we have 
        $$
        \int_{\B} |u_k|^p \,dx = \int_{U_\eps}|\phi(r)|^p |X_k(\theta)|^p \prod_{i=1}^{N-2}|\psi(\vartheta_i)|^p g d(r,\theta,\vartheta_1, \dots,\vartheta_{N-2}) \le d_p
        $$
        with
        $$
        d_p:= 2 \pi \int_{1-\eps}^1 |\phi(r)|^pdr  
        \Bigl(\int_{\frac{\pi}{2}-\eps}^{\frac{\pi}{2}+\eps}|\psi(\vartheta)|^p 
        d\vartheta\Bigr)^{N-2} < \infty.
        $$
        It thus follows that
        $$
        \frac{\int_{\B} \left( |\nabla u_k|^2 - \alpha^2|\del_\theta u_k|^2 +m 
        |u_k|^2 \right) \, dx}{ \Bigl(\int_{\B} |u_k|^p 
        \,dx\Bigr)^{\frac{2}{p}}} \le \frac{c - d(k)-m d_2 
        }{\bigl(d_p\bigr)^{\frac{2}{p}}} \to - \infty \qquad \text{as $k \to 
        \infty$.}
        $$
for every $p \in [2,\infty)$, $m \in \R$. This shows (\ref{eq:infimum-minus-infinity}). 

Consequently, the study of ground state solutions of \eqref{Reduced equation} 
requires a completely different approach in the case $\alpha>1$. This is 
further treated in the forthcoming paper~\cite{kuebler}.
\end{remark}

In the following, we show that, for $\alpha \in [0,1)$ and $2<p< 2^*$, the 
value $\scrC_{\alpha,m,p}(\B)>0$ is attained in $H_0^1(\B) \setminus \{0\}$ and 
that any minimizer gives rise to a weak solution of \eqref{Reduced equation}.

\begin{lemma} \label{existence of minimiers - elliptic case}
  Let $0 \leq \alpha<1$, $2<p<2^*$ and $m >-\lambda_1(\B)$. Then the value 
  $\scrC_{\alpha,m,p}(\B)$ is positive and attained at a function $u_0 \in 
  H^1_0(\B) \setminus \{0\}$. Moreover, after multiplication by a positive 
  constant, $u_0$ is a weak solution of (\ref{Reduced equation}), and $u_0 \in 
  C^{2,\sigma}(\overline \B)$ for some $\sigma>0$.
\end{lemma}

\begin{proof}
We first note that 
$$
\int_\B \left( |\nabla u|^2 - \alpha^2  |\del_\theta u|^2 \right) \, dx \geq 
(1-\alpha^2) \int_\B |\nabla u|^2 \,  dx \qquad \text{for $u \in H^1_0(\B)$.}
$$
Since $\alpha \in [0,1)$, it therefore follows from Sobolev embeddings that
\begin{equation} \label{R-Sobolev-estimate}
R_{\alpha,m,p}(u) \geq C_{m,\alpha} \frac{\int_\B |\nabla u|^2  \, dx}{\left(\int_\B |u|^p \, dx \right)^\frac{2}{p}} \qquad \text{$u \in H_0^1(\B) \setminus\{0\}$}
\end{equation}
with a constant $C_{m,\alpha}>0$.
We take a minimizing sequence $(u_n)_n$ for the Rayleigh quotient 
$R_{\alpha,m,p}$, normalized such that $\int_\B |u|^p \, dx =1$ for all $n$. By 
\eqref{R-Sobolev-estimate}, $(u_n)_n$ remains bounded in $H^1_0(\B)$ and we may 
pass to subsequence that weakly converges to $u_0 \in H_0^1(\B)$. The 
compactness of the embedding $H^1_0(\B) \hookrightarrow L^p(\B)$ and the weak 
lower semicontinuity of the quadratic form $u \mapsto \int_\B \left( |\nabla 
u|^2 - \alpha^2  |\del_\theta u|^2 \right) \, dx$ then imply that $\int_\B 
|u_0|^p \, dx=1$ and $R_{\alpha,m,p}(u_0)= \scrC_{\alpha,m,p}(\B)$. Hence 
$\scrC_{\alpha,m,p}(\B)$ is attained, and $\scrC_{\alpha,m,p}(\B)>0$ by 
Corollary~\ref{first-cor-section-2}.
	
	Next, standard variational arguments show that every $L^p$-normalized minimizer $u_0$ must be a weak solution of
	\begin{equation*} 
	\left\{ 
	\begin{aligned} 
	-\Delta u + \alpha^2 \del_\theta^2 u + m u & = \scrC_{\alpha,m,p}(\B) |u|^{p-2} u \quad && \text{in $\B$} \\
	u & = 0 && \text{on $\del \B$.}
	\end{aligned}
	\right.
	\end{equation*}
	We then conclude that $\bigl[\scrC_{\alpha,m,p}(\B)\bigr]^\frac{1}{p-2} u_0$ solves (\ref{Reduced equation}). Finally, classical elliptic regularity theory yields $C^{2,\sigma}(\overline \B)$ since the operator $-\Delta -\alpha^2 \partial_\theta$ is uniformly elliptic in $\B$ in the case $0 \le \alpha <1$.
\end{proof}
\begin{definition}
  \label{ground-state-solution}
  Let $0 \leq \alpha <1$, $2<p<2^*$ and $m >-\lambda_1(\B)$. A weak solution $u 
  \in H^1_0(\B) \setminus \{0\}$ of (\ref{Reduced equation}) such that 
  $R_{\alpha,m,p}(u)=\scrC_{\alpha,m,p}(\B)$ will be called a \textbf{ground 
  state 
  solution}.  
\end{definition}
\subsection{The degenerate elliptic case $\alpha = 1$.}
\label{sec:limiting-case-alpha}

In the limiting case $\alpha=1$, problem~(\ref{Reduced equation}) becomes degenerate and requires to work in a function space different from $H^1_0(\B)$. From Proposition~\ref{lambda-1-alpha-pos}, we deduce that 
$$
(u,v) \mapsto \langle u, v \rangle_{\cH}:= \int_{\B} \left(\nabla u \cdot \nabla v - \del_\theta u \del_\theta v  \right) \, dx
$$
defines a scalar product on $C_c^1(\B)$. The induced norm will be denoted by $\|\cdot\|_\cH$. 
\begin{lemma}
  \label{def: degenerate space on ball}
  Let $\cH$ denote the completion of $C_c^1(\B)$ with respect to 
  $\|\cdot\|_\cH$. Then $\cH$ is a Hilbert space which is embedded in $L^p(\B)$ 
  for $p \in [2,2_{\text{\tiny $1$}}^*]$, where $2_{\text{\tiny $1$}}^* := 
  \frac{4 N+2}{2N-3}$ as before. Moreover, we have:
  \begin{itemize}
  \item[(i)] If $1\leq p<{2_{\text{\tiny $1$}}^*}$, then the embedding $\cH \hookrightarrow L^p(\B)$ is compact.
  \item[(ii)] If $m > -\lambda_1(\B)$ and $p \in [2,2_{\text{\tiny $1$}}^*]$, 
  then the Rayleigh quotient $R_{1,m,p}(u)$ is well defined by 
  (\ref{eq:def-Raleigh-quotient}) and positive for functions $u \in \cH 
  \setminus \{0\}$, 
  \end{itemize}
\end{lemma}

\begin{proof}
  The embedding $\cH \hookrightarrow L^p(\B)$ for $p \in [2,2_{\text{\tiny $1$}}^*]$ is an immediate consequence of Theorem~\ref{degenerate Sobolev inequality - ball}.

  To prove (i), we fix $p \in [1,{2_{\text{\tiny $1$}}^*})$, and we let 
  $(u_n)_n \subset \cH$ be a bounded sequence. Moreover, we put 
  $B_m:=B_{1-1/m}(0) \subset \B$ for $m \geq 2$. Then $u_n^m:=\1_{B_m} u_n$ 
  defines a bounded sequence in $H^1(B_m)$ for every $m \ge 2$. After passing 
  to a subsequence, $(u_n^m)_n$ converges in $L^p(B_m)$ by Rellich-Kondrachov.  
	After passing to a diagonal sequence we may therefore assume that there exists $u \in L^p(\B)$ with the property that 
	$u_n \to u$ for $m \in \N$.	Moreover, 
	$$
	\|u-u_n\|_{L^p(\B)} \leq \|u-u_n\|_{L^p(B_m)} + \|u-u_n\|_{L^q(\B \setminus B_m)} |\B \setminus B_m|^{\frac{1}{p}-\frac{1}{q}}.
	$$
	Since $\|u-u_n\|_{L^q(\B \setminus B_m)} \leq \|u-u_n\|_{L^q(\B)}$ remains bounded independently of $m$ and $n$, this gives
	$$
	\limsup_{n \to \infty} \|u-u_n\|_{L^p(\B)} \leq C |\B \setminus B_m|^{\frac{1}{p}-\frac{1}{q}} 
	$$
	for some $C>0$ independent of $m$, where the right hand side tends to zero 
	as $m \to \infty$. This proves that $u_n \to u$ in $L^p(\B)$.

Finally, we note that (ii) is an immediate consequence of 
Proposition~\ref{lambda-1-alpha-pos} and the embedding $\cH \hookrightarrow 
L^p(\B)$ for $p \in [2,2_{\text{\tiny $1$}}^*]$.       
\end{proof}

Lemma~\ref{def: degenerate space on ball} allows the following definition of a weak solution of (\ref{Reduced equation}) with $\alpha = 1$ in the case where $p \in [2,2_{\text{\tiny $1$}}^*]$. 
\begin{definition}
  \label{ground-state-solution-alpha=1}
  Let $m > -\lambda_1(\B)$, $p \in [2,2_{\text{\tiny $1$}}^*]$.
  \begin{itemize}
  \item[(i)] We call $u \in \cH$ a \textbf{weak solution} of (\ref{Reduced 
  equation}) with $\alpha = 1$ if
  $$
  \langle u, v \rangle_\cH = \int_{\B}\bigl(|u|^{p-2}u v - m uv\bigr)dx
  \qquad \text{for every $v \in \cH$.}
  $$
  \item[(ii)] A weak solution $u \in \cH$ of (\ref{Reduced equation}) with 
  $\alpha = 1$ will be called a \textbf{ground state solution} if $u$ is a 
  minimizer for $R_{1,m,p}$, i.e., we have $R_{1,m,p}(u)=\scrC_{1,m,p}(\B)$.
  \end{itemize}
\end{definition}

We then have the following existence result which replaces Proposition~\ref{existence of minimiers - elliptic case} in the degenerate elliptic case $\alpha =1$. 

\begin{proposition} \label{Prop: ground states for critical case}
  Let $1<p<{2_{\text{\tiny $1$}}^*}$ and $m >-\lambda_1(\B)$. Then we have
  \begin{equation}
    \label{eq:positivity-C-non-critical}
    \scrC_{1,m,p}(\B)>0,
  \end{equation}
  and there exists $u_0 \in \cH \setminus \{0\}$ with $R_{1,m,p}(u_0)= 
  \scrC_{1,m,p}(\B)$, i.e., $u_0$ minimizes $R_{1,m,p}$ in $\cH \setminus 
  \{0\}$. Furthermore, after multiplication by a positive constant, $u_0$ is 
  ground state solution of (\ref{Reduced equation}) with $\alpha=1$ and $u_0 
  \in 
  C^{2,\sigma}_\loc(\B)$ for some $\sigma>0$.
\end{proposition}
\begin{proof}
	Proving the existence of $u_0$ is completely analogous to the proof of 
	Lemma~\ref{existence of minimiers - elliptic case}, making use of the 
	Rellich-Kondrachov type result stated in Lemma~\ref{def: degenerate space 
	on ball}(i).
	
	In order to prove the regularity result, we first note that a Moser iteration scheme can be used to show that $u_0 \in L^\infty(\B)$, see Lemma~\ref{Lemma: Moser iteration} in the appendix for a detailed proof. For any fixed $s \in (0,1)$ we may then use the fact that the operator $-\Delta+\del_\theta^2$ is uniformly elliptic in the ball $B_s=\{x \in \R^N: |x|<s\}$ and classical elliptic regularity theory, to show $u_0 \in C_\loc^{2,\sigma}(B_s)$.
\end{proof}

Next, we treat the critical case $p={2_{\text{\tiny $1$}}^*}$, and first show 
that $\scrC_{1,m,{2^*_{\scaleto{1}{3pt}}}}(\B)$ is attained, provided it is 
small 
enough.

\begin{theorem} \label{existence of minimizers - minima comparison}
	Let $m>-\lambda_1(\B)$ such that $\scrC_{1,m,{2^*_{\scaleto{1}{3pt}}}}(\B) 
	<2^{\frac{1}{2}-\frac{1}{2^*_{\scaleto{1}{3pt}}}} \cS_1(\R^N_+)$.
	Then the value $\scrC_{1,m,{2^*_{\scaleto{1}{3pt}}}}(\B) $ is attained in 
	$\cH \setminus \{0\}$.
\end{theorem}
In particular, this proves the first part of Theorem~\ref{Theorem: Minima comparison - introduction}.
The strategy of the proof is inspired by \cite{Frank-Jin-Xiong} and first requires the following characterization of sequences in $\cH$:

\begin{lemma} \label{Minima comparison}
	Let
	$$
		Z(v):= \int_{\B} \left(|\nabla v|^2- |\del_\theta v|^2 + m v^2 \right) \, dx \qquad \text{and}\qquad 
		N(v) :=\int_{\B}|v|^{{2_{\text{\tiny $1$}}^*}} \, dx \qquad \text{for $v \in \cH$.}
	$$
Then we have
	$$
	2^{\frac{1}{2}-\frac{1}{2^*_{\scaleto{1}{3pt}}}} \cS_1(\R^N_+) \leq \inf 
	\left\{ \liminf_{n \to 
	\infty } Z(w_n): (w_n)_n 
	\subset \cH, \ N(w_n)=1, \ w_n \weakto 0 \ \text{in $\cH$} \right\} .
	$$
\end{lemma}
\begin{proof}
	Let $(w_n)_n \subset \cH$ such that $N(w_n)=1, \ w_n \weakto 0$ in $\cH$. 
	Let $\eps>0$ and choose $U_0,\ldots, U_m \subset \B$ as in the proof of Theorem~\ref{degenerate Sobolev inequality - ball}, so that
	$$
	\B \subset \bigcup_{k=0}^m U_k .
	$$
	We may then choose functions $\eta_0, \ldots, \eta_m \in C^2_c(\B)$ such that $\supp \,  \eta_k \subset U_k$ and $\sum \limits_{k=0}^m \eta_k^2 \equiv 1$ on $\B$. Then
	\begin{align*}
	 \int_{\B} \left(|\nabla (\eta_k w_n)|^2- |\del_\theta \eta_k w_n|^2 \right) \, dx 
	= &  \int_{\B}\left(  \eta_k^2|\nabla  w_n|^2 + 2 w_n \eta_k \nabla w_n \cdot \nabla \eta_k + w_n^2 |\nabla \eta_k|^2 \right) \, dx \\
	&	- \int_{\B} \left( \eta_k^2 |\del_\theta w_n|^2 + 2 w_n \eta_k \del_\theta w_n \cdot \del_\theta \eta_k + w_n^2 |\del_\theta \eta_k|^2   \right)  \, dx
	\end{align*}
	and thus
	\begin{align*}
		\int_{\B} \left(|\nabla w_n|^2- |\del_\theta w_n|^2 + m w_n^2 \right) \, dx 
		& \geq \sum_{k=0}^m \int_{\B} \left(|\nabla (\eta_k w_n)|^2- |\del_\theta \eta_k w_n|^2 \right) \, dx - C \int_{\B} w_n^2 \, dx
	\end{align*}	 
	with a constant $C>0$ independent of $n$. Here we used the fact that the mixed terms can be estimated as follows:
	\begin{align*}
		\int_\B w_n^2 \  \left(|\nabla \eta_k|^2 -|\del_\theta \eta_k |^2 \right) \, dx & \leq  2 \sup_{k \in \{0,\ldots,m\}} \|\nabla \eta_k\|_\infty^2 \int_\B w_n^2 \, dx \\
		\int_{\B} \eta_k w_n \left(\nabla w_n \cdot \nabla \eta_k - \del_\theta w_n \ \del_\theta \eta_k \right) \, dx 
		& \leq  \int_{\B} \eta_k  w_n^2 \left|-\Delta \eta_k+\del_\theta^2 \eta_k\right| \, dx \\
		& \leq \sup_{k \in \{0,\ldots,m\}} \left\|-\Delta \eta_k + \del_\theta^2 \eta_k \right\|_\infty  \int_{\B} |w_n|^2 \, dx .
	\end{align*}
	We first note that $w_n \to 0$ in $L^2(\B)$, since the embedding $\cH \hookrightarrow L^2(\B)$ is compact by Lemma~\ref{def: degenerate space on ball}(i).
         	Moreover, it is easy to see that $\|\cdot\|_\cH$ induces an equivalent norm on $H^1_0(U_0)$, 
	which implies that $\eta_0 w_n \weakto 0$ in $H^1_0(U_0)$. In particular, noting that by ${2_{\text{\tiny $1$}}^*} < 2^*$ the classical Rellich-Kondrachov theorem implies $\eta_0 w_n \to 0$ in $L^{{2_{\text{\tiny $1$}}^*}}(\B)$, we conclude
	$$
	\liminf_{n \to \infty} \int_{\B} \left(|\nabla (\eta_0 w_n)|^2- |\del_\theta (\eta_0 w_n)|^2 + m ( \eta_0 w_n)^2 \right) \, dx \geq \liminf_{n \to \infty}  \left(\int_{\B} |\eta_0 w_n|^{{2_{\text{\tiny $1$}}^*}} \, dx \right)^\frac{2}{{2_{\text{\tiny $1$}}^*}}  .
	$$ 
	On the other hand, Lemma~\ref{lemma: inequality close to boundary} gives
	\begin{align*} 
		\int_{\B} \left(|\nabla (\eta_k w_n)|^2- |\del_\theta \eta_k w_n|^2 
		\right) \, dx & \geq (1-\eps) 
		2^{\frac{1}{2}-\frac{1}{2^*_{\scaleto{1}{3pt}}}}  
		\cS_1(\R^N_+) \left(\int_{\B} |\eta_k 
		w_n|^{{2_{\text{\tiny $1$}}^*}} \, dx\right)^\frac{2}{{2_{\text{\tiny 
		$1$}}^*}}
	\end{align*}
    for $k=1,\ldots,m$ and hence
	\begin{align*}
		\liminf_{n \to \infty} \int_{\B} \left(|\nabla w_n|^2- |\del_\theta w_n|^2 + m w_n^2 \right) \, dx & \geq \liminf_{n \to \infty} \sum_{k=0}^m \int_{\B} \left(|\nabla (\eta_k w_n)|^2- |\del_\theta \eta_k w_n|^2 \right) \, dx \\
		& \geq (1-\eps)2^{\frac{1}{2}-\frac{1}{2^*_{\scaleto{1}{3pt}}}} 
		\cS_1(\R^N_+) \liminf_{n \to 
		\infty} \sum_{k=0}^m 
		\left(\int_{\B} |\eta_k w_n|^{{2_{\text{\tiny $1$}}^*}} \, 
		dx\right)^\frac{2}{2^*_{\scaleto{1}{3pt}}} \\
		&= (1-\eps) 2^{\frac{1}{2}-\frac{1}{2^*_{\scaleto{1}{3pt}}}} 
		\cS_1(\R^N_+) \liminf_{n \to 
		\infty} \sum_{k=0}^m \left\| 
		\eta_k^2 w_n^2 \right\|_{\frac{2^*_{\scaleto{1}{3pt}}}{2}} \\
        & \geq (1-\eps)2^{\frac{1}{2}-\frac{1}{2^*_{\scaleto{1}{3pt}}}} 
        \cS_1(\R^N_+) \liminf_{n \to 
        \infty} \left\|\sum_{k=0}^m 
        \eta_k^2 w_n^2 \right\|_{\frac{2^*_{\scaleto{1}{3pt}}}{2}} \\
        & = (1-\eps)2^{\frac{1}{2}-\frac{1}{2^*_{\scaleto{1}{3pt}}}}  
        \cS_1(\R^N_+) \liminf_{n \to 
        \infty} \left\| w_n 
        \right\|_{\frac{2^*_{\scaleto{1}{3pt}}}{2}} \\
        &=(1-\eps)2^{\frac{1}{2}-\frac{1}{2^*_{\scaleto{1}{3pt}}}}  
        \cS_1(\R^N_+).
	\end{align*}
	Since $\eps>0$ was arbitrary, we conclude that
	$$
	\liminf_{n \to \infty} \int_{\B} \left(|\nabla w_n|^2- |\del_\theta w_n|^2 
	+ m w_n^2 \right) \, dx \geq 
	2^{\frac{1}{2}-\frac{1}{2^*_{\scaleto{1}{3pt}}}}  \cS_1(\R^N_+)
	$$
	as claimed.
\end{proof}
We may now complete the proof of our main result.
\begin{proof}[Proof of Theorem~\ref{existence of minimizers - minima comparison}]
Consider a minimizing sequence $(u_n)_n \subset \cH$ for 
$\scrC_{1,m,{2^*_{\scaleto{1}{3pt}}}}(\B) $ with $\|u_n\|_{{2_{\text{\tiny 
$1$}}^*}}=1$. Then $(u_n)_n$ is bounded in $\cH$, hence, after passing to a 
subsequence, we may assume $u_n \weakto u_0$ in $\cH$. 
We set $v_n:=u_n-u_0$ and note that, by Sobolev embeddings, 
$$
v_n \to 0 \quad \text{in $L^q(B_s)$}
$$
for $1\leq q <{2_{\text{\tiny $1$}}^*}$ and $0<s<1$, where $B_s:=\{x \in \R^N: |x|<s\}$.
Weak convergence implies
$$
\scrC_{1,m,{2^*_{\scaleto{1}{3pt}}}}(\B)  = \lim_{n \to \infty} 
Z(u_n)=Z(u_0)+\lim_{n \to \infty}Z(v_n) ,
$$
whereas the Brezis-Lieb Lemma yields
$$
1=N(u_n)=N(u_0)+N(v_n)+o(1) .
$$
In particular, the limits
$$
T :=\lim_{n \to \infty} N(v_n), \qquad
M := \lim_{n \to \infty} Z(v_n)
$$
exist. If $T=0$, it follows that $N(u_0)=1$ and we are finished. For $T>0$, Lemma~\ref{Minima comparison} implies 
$$
M \geq \cS_1(\R^N_+) T^\frac{2}{2^*_{\scaleto{1}{3pt}}}
$$
and hence
\begin{align*}
\scrC_{1,m,{2^*_{\scaleto{1}{3pt}}}}(\B)  & = Z(u_0)+ M \geq Z(u_0)+ 
2^{\frac{1}{2}-\frac{1}{2^*_{\scaleto{1}{3pt}}}} \cS_1(\R^N_+) 
T^\frac{2}{2^*_{\scaleto{1}{3pt}}} 
\\
& \geq Z(u_0)+ \left(2^{\frac{1}{2}-\frac{1}{2^*_{\scaleto{1}{3pt}}}} 
\cS_1(\R^N_+)-\scrC_{1,m,{2^*_{\scaleto{1}{3pt}}}}(\B)  \right) 
T^\frac{2}{2^*_{\scaleto{1}{3pt}}} + \scrC_{1,m,{2_{\text{\tiny 
$1$}}^*}}(\B)  \left(1-N(u_0)\right)^\frac{2}{2^*_{\scaleto{1}{3pt}}} \\
& \geq Z(u_0)+ (2^{\frac{1}{2}-\frac{1}{2^*_{\scaleto{1}{3pt}}}} 
\cS_1(\R^N_+)-\scrC_{1,m,{2^*_{\scaleto{1}{3pt}}}}(\B)  ) 
T^\frac{2}{2^*_{\scaleto{1}{3pt}}} + \scrC_{1,m,{2_{\text{\tiny 
$1$}}^*}}(\B)  - \scrC_{1,m,{2^*_{\scaleto{1}{3pt}}}}(\B)  
N(u_0)^\frac{2}{2^*_{\scaleto{1}{3pt}}} ,
\end{align*}
where we used the inequality $(a-b)^\tau \geq a^\tau-b^\tau$ for $a \geq b \geq 0$ and $0 \leq \tau \leq 1$. 
It follows that
$$
Z(u_0)+ (2^{\frac{1}{2}-\frac{1}{2^*_{\scaleto{1}{3pt}}}} 
\cS_1(\R^N_+)-\scrC_{1,m,{2_{\text{\tiny 
$1$}}^*}}(\B)  ) 
T^\frac{2}{2^*_{\scaleto{1}{3pt}}}  - \scrC_{1,m,{2_{\text{\tiny 
$1$}}^*}}(\B)  N(u_0)^\frac{2}{2^*_{\scaleto{1}{3pt}}} \leq 0 , 
$$
and therefore
\begin{equation} \label{optimizer inequality}
	\begin{aligned} 
& \int_{\B} \left(|\nabla u_0|^2- |\del_\theta u_0|^2 + m u_0^2 \right) \, dx  
- 
\scrC_{1,m,{2^*_{\scaleto{1}{3pt}}}}(\B)  \left( \int_{\B} 
|u_0|^{{2_{\text{\tiny $1$}}^*}} \, dx \right)^\frac{2}{2^*_{\scaleto{1}{3pt}}} 
\\
& + (2^{\frac{1}{2}-\frac{1}{2^*_{\scaleto{1}{3pt}}}} 
\cS_1(\R^N_+)-\scrC_{1,m,{2_{\text{\tiny 
$1$}}^*}}(\B)  ) 
T^\frac{2}{2^*_{\scaleto{1}{3pt}}} \leq 0 .
\end{aligned} 
\end{equation}
By definition, we have $\int_{\B} \left(|\nabla u_0|^2- |\del_\theta u_0|^2 + m 
u_0^2 \right) \, dx  - \scrC_{1,m,{2^*_{\scaleto{1}{3pt}}}}(\B)  \left( 
\int_{\B} |u_0|^{{2_{\text{\tiny $1$}}^*}} \, dx 
\right)^\frac{2}{2^*_{\scaleto{1}{3pt}}} \geq 0$ and since 
$2^{\frac{1}{2}-\frac{1}{2^*_{\scaleto{1}{3pt}}}} 
\cS_1(\R^N_+)-\scrC_{1,m,{2_{\text{\tiny 
$1$}}^*}}(\B)  > 0$ by assumption, we 
must have $T=0$, i.e. $v_n \to 0$ in $L^p(\Omega)$. It follows that $u_0 \not 
\equiv 0$ and $\int_{\B} |u_0|^{{2_{\text{\tiny $1$}}^*}} \, dx=1$, and 
\eqref{optimizer inequality} gives
$$
\int_{\B} \left(|\nabla u_0|^2- |\del_\theta u_0|^2 + m u_0^2 \right) \, dx 
\leq \scrC_{1,m,{2^*_{\scaleto{1}{3pt}}}}(\B)  \left( \int_{\B} 
|u_0|^{{2_{\text{\tiny $1$}}^*}} \, dx \right)^\frac{2}{2^*_{\scaleto{1}{3pt}}} 
,
$$
which implies that $u_0$ is a minimizer.
\end{proof}

We note the following consequence of Theorem~\ref{existence of minimizers - minima comparison}, which extends (\ref{eq:positivity-C-non-critical}) to the critical case. 

\begin{corollary}
  \label{existence of minimizers - minima comparison-cor}
We have $\scrC_{1,m,{2^*_{\scaleto{1}{3pt}}}}(\B)>0$.
\end{corollary}

\begin{proof}
  If the value $\scrC_{1,m,{2^*_{\scaleto{1}{3pt}}}}(\B) $ is attained in $\cH 
  \setminus \{0\}$, then we have $\scrC_{1,m,{2^*_{\scaleto{1}{3pt}}}}(\B)>0$ 
  by Lemma~\ref{def: degenerate space on ball}(ii). If not, we have  
  $\scrC_{1,m,{2^*_{\scaleto{1}{3pt}}}}(\B) \ge \cS_1(\R^N_+)>0$ by 
  Theorem~\ref{existence of minimizers - minima comparison} and 
  Theorem~\ref{general-degenerate-Sobolev-inequality-whole space}. 
\end{proof}

In general, the existence of ground state solutions in the case $\alpha=1$, $p = {2_{\text{\tiny $1$}}^*}$ remains an open problem and might depend on the parameter $m > -\lambda_1(\B)$. We have the following partial existence result in the critical case.

\begin{theorem} \label{thm: existence of solutions in critical case}
	There exists $\eps>0$, such that for $m \in (-\lambda_1(\B),-\lambda_1(\B)+\eps)$
	there exists $u_0 \in \cH \setminus \{0\}$ such that 
	$$
	R_{1,m,{2_{\text{\tiny $1$}}^*}}(u_0)= \inf_{u \in \cH \setminus \{0\}} R_{1,m,{2_{\text{\tiny $1$}}^*}}(u),
	$$
	i.e. $u_0$ minimizes $R_{1,m,{2_{\text{\tiny $1$}}^*}}$. Furthermore, after multiplication by a positive constant, $u_0$ is a weak solution of
	\begin{equation*}
	\left\{ 
	\begin{aligned} 
	-\Delta u + \del_\theta^2 u + m u & = |u|^{{2_{\text{\tiny $1$}}^*}-2} u \quad && \text{in $\B$,} \\
	u & = 0 && \text{on $\del \B$,}
	\end{aligned}
	\right.
	\end{equation*}
	i.e., $u_0$ satisfies
	$$
	\int_{\B } \nabla u \cdot \nabla \phi - \del_\theta u \, \del_\theta \phi + m u \phi \, dx = \int_{\B} |u|^{{{2_{\text{\tiny $1$}}^*}}-2}u \phi \, dx
	$$
	for all $\phi \in \cH$.
\end{theorem}
\begin{proof}
	For a (necessarily radial) eigenfunction $\phi_1$ of $-\Delta$ on $\B$ corresponding to $\lambda_1(\B)$,  we have
	$$
	\scrC_{1,m,{2^*_{\scaleto{1}{3pt}}}}(\B) \leq R_{1,m,{2_{\text{\tiny 
	$1$}}^*}}(\phi_1) = \frac{(\lambda_1(\B)+m)\int_{\B}  \phi_1^2  \, 
	dx}{\left(\int_{\B}|\phi_1|^p \, dx\right)^\frac{2}{p}}
	$$
	which implies $\scrC_{1,m,{2^*_{\scaleto{1}{3pt}}}}(\B) \to 0$ as $m \to 
	-\lambda_1(\B)^+$. In particular, it follows that there exists $\eps>0$ 
	such that 
	$$
	\scrC_{1,m,{2^*_{\scaleto{1}{3pt}}}}(\B)  < 
	2^{\frac{1}{2}-\frac{1}{2^*_{\scaleto{1}{3pt}}}} 
	\cS_1(\R^N_+) 
	$$
	holds for $m \in (-\lambda_1(\B),-\lambda_1(\B)+\eps)$. By Theorem~\ref{existence of minimizers - minima comparison}, this finishes the proof.
\end{proof}
Note that this completes the proof of Theorem~\ref{Theorem: Minima comparison - introduction}.

\subsection{Radiality versus $x_1$-$x_2$-nonradiality of ground state solutions.}
\label{sec:radiality-versus-r}

By classical results due to McLeod and Serrin~\cite{McLead-Serrin}, 
Kwong~\cite{Kwong}, Kwong and Li~\cite{Kwong-Li} (see also references in 
\cite{Damascelli-Grossi-Pacella}), the problem
  \begin{equation}
    \label{alpha=0-problem}
	\left\{ 
	\begin{aligned} 
	-\Delta u +  m u & = |u|^{p-2} u \quad && \text{in $\B$} \\
	u & = 0 && \text{on $\del \B$}, 
	\end{aligned}
	\right.
	\end{equation}
has a unique radial positive solution $u_{rad} \in H^1_0(\B)$ which is a minimizer for $\scrC_{0,m,p}(\B)$. Clearly, $u_{rad}$ is also a weak solution of (\ref{Reduced equation}) for every $\alpha>0$, but it might not be a ground state solution. In fact, we have the following.

\begin{lemma} \label{continuity of inf}
	Let $2<p<2^*$ and $m>-\lambda_1(\B)$ be fixed.
        \begin{enumerate}
        \item[(i)] The map 
	$$
	[0,1] \to \R, \qquad \alpha \mapsto \scrC_{\alpha,m,p}(\B)
	$$ 
	is continuous and nonincreasing. 
      \item[(ii)] Let $\alpha \in (0,1]$, and suppose that $p \le 2_{\text{\tiny $1$}}^*$ in the case $\alpha =1$. Then the following properties are equivalent:
        \begin{itemize}
        \item[$(ii)_1$] $\scrC_{\alpha,m,p}(\B) < \scrC_{0,m,p}(\B)$.
        \item[$(ii)_2$] Every ground state solution of (\ref{Reduced equation}) is $x_1$-$x_2$-nonradial.  
        \end{itemize}
        \end{enumerate}
\end{lemma}

\begin{proof}
(i) The monotonicity of $\scrC_{\alpha,m,p}(\B)$ in $\alpha$ follows immediately from the definition. In order to prove continuity, we first consider $\alpha_0 \in (0,1]$ and let $\eps>0$. Moreover, we let $u_0 \in H^1_0(\B) \setminus \{0\}$ be a function with $R_{\alpha_0,m,p}(u_0)<\scrC_{\alpha_0,m,p}(\B)+\eps$. For $\alpha \leq \alpha_0$, we then have
\begin{align*}
	\scrC_{\alpha_0,m,p}(\B) &\leq \scrC_{\alpha,m,p}(\B) \leq R_{\alpha,m,p}(u_0)\\
        &\le R_{\alpha_0,m,p}(u_0) + (\alpha_0^2-\alpha^2)  \frac{\int_\B 
        |\del_\theta u_0|^2  \, dx}{\left(\int_\B |u_0|^p \, dx 
        \right)^\frac{2}{p}} \\
        & \le \scrC_{\alpha_0,m,p}(\B) + (\alpha_0^2-\alpha^2)  \frac{\int_\B 
        |\del_\theta u_0|^2  \, dx}{\left(\int_\B |u_0|^p \, dx 
        \right)^\frac{2}{p}}
\end{align*}
which implies that $\limsup \limits_{\alpha \to \alpha_0^-}|\scrC_{\alpha,m,p}(\B)-\scrC_{\alpha_0,m,p}(\B)| \le \eps.$ This shows the continuity from the left in $\alpha_0$.

Next we let $\alpha_0 \in [0,1)$ and show continuity from the right in 
$\alpha_0$. For this we fix $\delta>0$ such that $(\alpha_0,\alpha_0+\delta) 
\subset (0,1)$. For $\alpha \in (\alpha_0,\alpha_0+\delta)$,  
Lemma~\ref{existence of minimiers - elliptic case} implies that the value 
$\scrC_{\alpha,m,p}(\B)$ is attained at a function $u_\alpha \in H^1_0(\B) 
\setminus \{0\}$ with $\int_\B |u_\alpha|^p \, dx=1$. Therefore
	\begin{align*}
	\scrC_{\alpha_0,m,p}(\B) &\geq \scrC_{\alpha,m,p}(\B) = R_{\alpha,m,p}(u_\alpha) = R_{\alpha_0,m,p}(u_\alpha) + (\alpha_0^2-\alpha^2)\int_\B |\del_\theta u_\alpha|^2  \, dx \\
	& \geq \scrC_{\alpha_0,m,p}(\B) - |\alpha_0^2-\alpha^2| \int_\B |\nabla u_\alpha|^2  \, dx ,
	\end{align*}
	whence, using the fact that
        $$
        (1-\alpha^2) \int_\B |\nabla u_\alpha|^2  \, dx \leq
        \int_\B \left( |\nabla u_\alpha|^2 - \alpha^2  |\del_\theta u_\alpha|^2\right)dx =\scrC_{\alpha,m,p}(\B) \leq \scrC_{0,m,p}(\B),
          $$
          we conclude
	$$
	\begin{aligned} 
	\scrC_{\alpha_0,m,p}(\B) & \geq \scrC_{\alpha,m,p}(\B) \geq 
	\scrC_{\alpha_0,m,p}(\B) - \frac{|\alpha_0^2-\alpha^2|}{1-\alpha^2} 
	\scrC_{0,m,p}(\B) \\
	&\geq \scrC_{\alpha_0,m,p}(\B) - 
	\frac{|\alpha_0^2-\alpha^2|}{1-(\alpha_0+\delta)^2} \scrC_{0,m,p}(\B).
	\end{aligned}
    $$
This shows the continuity from the right in $\alpha_0$.
	
        (ii) As noted above, $\scrC_{0,m,p}(\B)$ is attained by a radial 
        positive solution $u_{rad}$ of (\ref{alpha=0-problem}) and we have
        $R_{0,m,p}(u_{rad}) = R_{\alpha,m,p}(u_{rad})$. 
        Hence, if $\scrC_{0,m,p}(\B) = \scrC_{\alpha,m,p}(\B)$, then $u_{rad}$ is also a radial ground state solution of (\ref{Reduced equation}).
        Hence $(ii)_2$ and (i) imply that $\scrC_{\alpha,m,p}(\B)<\scrC_{0,m,p}(\B)$. If, conversely, there exists a radial ground state solution $u$ of (\ref{Reduced equation}), then we have 
$$
\scrC_{0,m,p}(\B) \le R_{0,m,p}(u) = R_{\alpha,m,p}(u)= \scrC_{\alpha,m,p}(\B)
$$
and therefore equality holds by (i). Consequently, the $\scrC_{\alpha,m,p}(\B)<\scrC_{0,m,p}(\B)$ implies that every ground state solution of (\ref{Reduced equation}) is $x_1$-$x_2$-nonradial.
      \end{proof}

      The second part of this section is devoted to the proof of Theorem~\ref{sec:introduction-second-theorem}, which yields
radiality of ground state solutions for $\alpha$ close to zero. For this, we fix $m  \ge 0$ and $2<p<2^*$. Moreover, we consider a sequence of numbers $\alpha_n \in (0,1)$, $\alpha_n \to 0$ and, for every $n \in \N$, a positive ground state solution $u_n \in H^1_0(\B)$ of (\ref{Reduced equation}) with $\alpha= \alpha_n$. Recall that the existence of $u_n$ is proved in Lemma~\ref{existence of minimiers - elliptic case}. 
To prove Theorem~\ref{sec:introduction-second-theorem}, it then suffices to show that
\begin{equation}
  \label{eq:sufficient-radial}
\text{$u_n= u_{rad}\;$ for $n$ sufficiently large,}  
\end{equation}
where $u_{rad}$ is the unique positive solution of (\ref{alpha=0-problem}). We first claim the following.

\begin{lemma} \label{convergence to radial}
	$u_n \to u_{rad}$ in $H_0^1(\B)$ as $n \to \infty$.
\end{lemma}

\begin{proof}
We put $v_n:= \frac{u_n}{\|u_n\|_{L^p(\B)}}$, so $v_n$ is an $L^p$-normalized minimizer for $\scrC_{\alpha_n,m,p}(\B)$. Then $(v_n)_n$ is bounded in $H_0^1(\B)$ by definition of $\scrC_{\alpha_n,m,p}(\B)$.
Consequently, we have $v_n \weakto v_0$ in $H_0^1(\B)$ after passing to a 
subsequence, which implies that $v_n \to v_0$ in $L^p(\B)$ and therefore 
$\int_{\B} |v_0|^p\,dx=1$. We show that $v_0$ is minimizer for 
$\scrC_{0,m,p}(\B)$. Indeed, by weak lower semicontinuity, we have
\begin{align*}
\scrC_{0,m,p}(\B) &\le R_{0,m,p}(v_0) \le \liminf_{n \to \infty} R_{0,m,p}(v_n) \le \lim_{n \to \infty} \Bigl(R_{\alpha_n,m,p}(v_n) + \alpha_n^2 \|\del_\theta u_n\|_{L^2(\B)}^2 \Bigr) \\
              & \leq \lim_{n \to \infty}\scrC_{\alpha_n,m,p}(\B) + \alpha_n \|u_n\|_{H^1(\B)}^2 = \scrC_{0,m,p}(\B),
\end{align*}
where we used Lemma~\ref{continuity of inf} in the last step. Hence $v_0$ is a minimizer of $\scrC_{0,m,p}(\B)$, and a posteriori we find that
\begin{align*}
  \|\nabla v_n\|_{L^2(\B)}^2 + m \|v_n\|_{L^2(\B)}^2 &= R_{\alpha_n,m,p}(v_n) + 
  \alpha_n^2 \|\del_\theta v_n\|_{L^2(\B)}^2\\
  &\to R_{0,m,p}(v_0) = \|\nabla v_0\|_{L^2(\B)}^2+ m \|v_0\|_{L^2(\B)}^2\qquad \text{as $n \to \infty$.}
\end{align*}
By uniform convexity of $H^1(\B)$, we thus conclude that $v_n \to v_0$ in $H^1_0(\B)$. Next we recall that, as noted in the proof of Lemma~\ref{existence of minimiers - elliptic case}, we have
$$
u_n := \bigl[\scrC_{\alpha_n,m,p}(\B)\bigr]^\frac{1}{p-2} v_n \qquad \text{and, by uniqueness,}\qquad \quad u_{rad} := \bigl[\scrC_{\alpha_n,m,p}(\B)\bigr]^\frac{1}{p-2} v_{0}. 
$$
Hence Lemma~\ref{continuity of inf} implies that $u_n \to u_{rad}$ in $H^1_0(\B)$. Although we have proved this only after passing to a subsequence, the convergence of the full sequence now follows from the uniqueness of $u_{rad}$. The proof is thus finished.
\end{proof}

Next, we improve Lemma~\ref{convergence to radial} by noting that
\begin{equation}
\label{second order convergence to radial}
u_n \to u_{rad} \qquad \text{in $H^2(\B)$.}
\end{equation}
This follows in a standard way from Lemma~\ref{convergence to radial} and standard elliptic regularity theory (see e.g. \cite[Theorem 8.12]{Gilbarg-Trudinger}), since $u_n = u_{rad}-u_n \in H_0^1(\B)$ is a weak solution of 
	$$
	\left\{ 
	\begin{aligned}
	-\Delta w_n + \alpha_n^2 \del_\theta w_n + m w_n & = |v_{rad}|^{p-2} v_{rad} - |v_n|^{p-2} v_n \quad && \text{in $\B$} \\
	w_n & = 0 && \text{on $\del \B$,}
	\end{aligned}
	\right.
	$$
	and the coefficients of the differential operator $-\Delta  + \alpha_n^2 \del_\theta$ are uniformly bounded and elliptic in $n \in \N$.

        We may now complete the proof of our main result as follows.

        \begin{proof}[Proof of Theorem~\ref{sec:introduction-second-theorem}]
To complete the proof of (\ref{eq:sufficient-radial}), we consider the map
	$$
	F: (-1,1) \times H^2(\B) \cap H_0^1(\B) \to L^2(\B), \quad F(\alpha,u) := 
	-\Delta u + \alpha^2 \del_\theta^2 u + m u  -|u|^{p-2} u  ,
	$$
and we note that weak solutions of \eqref{Reduced equation} correspond to zeroes of $F$. We also note that $F(\alpha, u_{rad}) =0$ for all $\alpha$. We wish to apply the implicit function theorem at $(0,u_{rad})$, so we need to check that
	$$
	[\del_u F](0,u_{rad}) = -\Delta + m  - (p-1) |u_{rad}|^{p-2}
	$$
    is invertible as a map $ H^2(\B) \cap H_0^1(\B) \to L^2(\B)$. This is 
    equivalent to the nondegeneracy of $u_{rad}$ as a solution of 
    (\ref{alpha=0-problem}) which is noted e.g. in \cite[Theorem 
    4.2]{Damascelli-Grossi-Pacella} for $m=0$ and in  \cite[Theorem 
    1.1]{Aftalion-Pacella} in the case $m>0$. Now the implicit function theorem 
    yields $\eps>0$ with the following property: If $u \in H^2(\B) \cap 
    H^1_0(\B)$ satisfies $\|u-u_{rad}\|_{H^2(\B)}< \eps$ and $F(\alpha,u) = 0$ 
    for some $\alpha \in (-\eps,\eps)$, then $u=u_{rad}$.

        Hence Lemma~\ref{second order convergence to radial} implies that
        $u_n=u_{rad}$ for $n$ sufficiently large, which shows (\ref{eq:sufficient-radial}), as claimed.          
        \end{proof}

In the remainder of this section, we show $x_1$-$x_2$-nonradial ground states for large $m$, as claimed in Theorem~\ref{Thm: m large}. We restate this theorem here in an equivalent form.
\begin{theorem} \label{Thm: m large extended}
	Let $\alpha \in (0,1)$ and $2<p<2^*$. Then there exists $\eps_0>0$, such that the ground states of
	\begin{equation} \label{eps equation}
	\left\{
	\begin{aligned}
	-\Delta u + \alpha^2 \del_\theta^2  u + \frac{1}{\eps^2} u &= |u|^{p-2} u \quad &&\text{in $\B$,} \\
	u & = 0 && \text{on $\del \B$,}
	\end{aligned}
	\right. 
	\end{equation}
	are $x_1$-$x_2$-nonradial for $\eps \in (0,\eps_0)$. Moreover, if $p<{2_{\text{\tiny $1$}}^*}$, the same result holds for $\alpha=1$.
\end{theorem}
\begin{proof}
  We first treat the case $\alpha \in (0,1)$. In the following, for $u \in H^1_0(\B)$ and $\eps>0$, we consider $B_{1/\eps}:= B_{1/\eps}(0)$ and the rescaled function $u_\eps \in H^1_0(B_{1/\eps})$, $u_\eps(x)= u(\eps x)$. A direct computation then shows that
  \begin{equation}
    \label{eq:rescaling-rule}
\frac{\int_{B_{1/\eps}} \left( |\nabla u_\eps|^2 - \alpha^2 \eps^2 |\del_\theta 
u_\eps|^2 + u_\eps^2 \, \right)dx}{\left(\int_{B_{1/\eps} } |u_\eps|^p \, dx 
\right)^\frac{2}{p}}  = \eps^{2-N+\frac{2N}{p}} 
R_{\alpha,\frac{1}{\eps^{2}},p}(u).
  \end{equation}
  As a consequence, we have 
  \begin{equation*}
  \scrC_{\alpha \eps ,1,p}(B_{1/\eps}):= \inf_{v \in H^1_0(B_{1/\eps}) \setminus \{0\}} \frac{\int_{B_{1/\eps}} \left( |\nabla v|^2 - \alpha^2 \eps^2 |\del_\theta v|^2 + v^2 \, \right)dx}{\left(\int_{B_{1/\eps} } |v|^p \, dx \right)^\frac{2}{p}}  = \eps^{2-N+\frac{2N}{p}}   \scrC_{\alpha,\frac{1}{\eps^2},p}(\B).
  \end{equation*}
It therefore suffices to show that there exists $\eps_0>0$ such that all 
minimizers for $\scrC_{\alpha \eps ,1,p}(B_{1/\eps})$ in $H^1_0(B_{1/\eps}) 
\setminus \{0\}$ are $x_1$-$x_2$-nonradial if $\eps \in (0,\eps_0)$. We argue by 
contradiction and suppose that there exists a sequence $\eps_n \to 0$ and, for 
every $n \in \N$, a minimizer $v_{\eps_n} \in H^1_0(B_{1/\eps_n}) \setminus 
\{0\}$ for $\scrC_{\alpha \eps_n ,1,p}(B_{1/\eps_n})$ which satisfies
  \begin{equation}
    \label{eq:zero-theta-deriv}
    \partial_\theta v_{\eps_n} \equiv 0 \qquad \text{in $B_{1/\eps_n}$.}
  \end{equation}
To simplify the notation, we continue writing $\eps$ in place of $\eps_n$ in the following. From (\ref{eq:zero-theta-deriv}) and the inclusion $H_0^1(B_{1/\eps}) \subset H^1(\R^N)$, we then deduce that 
	\begin{align} \label{eps-lower bound}
  \scrC_{\alpha \eps ,1,p}(B_{1/\eps}) &= \frac{\int_{B_{1/\eps}} \left( 
  |\nabla v_\eps|^2  + v^2 \, \right)dx}{\left(\int_{B_{1/\eps} } |v|^p \, dx 
  \right)^\frac{2}{p}} \nonumber\\
  &\ge  \inf_{v \in H^1({\R^N}) \setminus \{0\}} \frac{\int_{\R^N} \left( 
  |\nabla v|^2  + v^2 \, \right)dx}{\left(\int_{\R^N } |v|^p \, dx 
  \right)^\frac{2}{p}} =: \scrC_{0,1,p}(\R^N).
	\end{align}
We will now derive a contradiction to this inequality by constructing suitable 
functions in $H^1_0(B_{1/\eps} \setminus \{0\})$ to estimate $\scrC_{\alpha 
\eps ,1,p}(B_{1/\eps})$.  To this end, we first note that the value 
$\scrC_{0,1,p}(\R^N)$ is attained by any translation of the unique positive 
radial solution $\tilde u_0 \in H^1(\R^N)$ of the nonlinear Schrödinger equation
	 	\begin{equation*} 
	 	-\Delta u  + u   =|u|^{p-2} u \qquad \text{in $\R^N$.}
	 	\end{equation*}
   Now take a radial function $\eta \in C_c^1(\B)$ such that $0 \leq \eta \leq 1$ and $\eta \equiv 1$ in $B_{1/2}$, and let $u_0(x):=\tilde u_0(x-e_1)$ where $e_1=(1,0,\ldots,0)$. We then define 
   $$
   \eta_\eps, \; w_\eps \in C_c^1(B_{1/\eps})\qquad \text{by}\qquad \eta_\eps(x)=\eta(\eps x),\quad w_\eps(x)=\eta_\eps(x)u_0(x).
   $$
   Then we have $w_\eps \equiv u_0$ in $B_{1/(2\eps)}$, and 
   \begin{align} 
    &\scrC_{\alpha \eps ,1,p}(B_{1/\eps}) \le \frac{\int_{B_{1/\eps}} \left( |\nabla w_\eps|^2 - \alpha^2 \eps^2 |\del_\theta w_\eps|^2 + w_\eps^2 \, \right)dx}{\left(\int_{B_{1/\eps} } |w_\eps|^p \, dx \right)^\frac{2}{p}} \label{split quotient}\\
   =& \frac{\int_{B_{1/\eps}} \eta_\eps^2 \left( |\nabla u_0|^2  + u_0^2 \, \right)dx}{\left(\int_{B_{1/\eps} } \eta_\eps^p |u_0|^p \, dx \right)^\frac{2}{p}} 
  + \frac{\int_{B_{1/\eps}} \left( u_0^2|\nabla_\eps \eta|^2 +2 \eta_\eps u_0 \nabla \eta_\eps \cdot \nabla u_0 - \alpha^2 \eps^2 \eta_\eps^2 |\del_\theta u_0|^2  \, \right)dx}{\left(\int_{B_{1/\eps} } \eta_\eps^p |u_0|^p \, dx \right)^\frac{2}{p}}. \nonumber
   \end{align}
   We first estimate the second term and note that classical results (see \cite{Berestycki-Lions}) imply that there exist $C_0,\delta_0>0$, such that
   \begin{equation} \label{Exponential decay}
   |u_0(x)|, |\nabla u_0(x)| \leq C_0 e^{-\delta_0 |x|} \quad \text{for $x \in \R^N$.}
   \end{equation}
   Noting that $\nabla \eta_\eps \equiv 0$ on $B_{1/(2\eps)}$, this readily 
   implies
   $$
   \int_{B_{1/\eps}} \left( u_0^2|\nabla \eta_\eps|^2 +2 \eta_\eps u_0 \nabla \eta_\eps \cdot \nabla u_0   \, \right)dx \leq C_1 e^{-\frac{\delta_1}{\eps}}
   $$
   for some constants $C_1,\delta_1>0$. Moreover, for $\eps \in (0,\frac{1}{2})$ we have 
   $$
   \alpha^2 \eps^2 \int_{B_{1/\eps}}   \eta_\eps^2 |\del_\theta u_0|^2  \, dx \geq C_2 \eps^2 \qquad \text{with}\quad
   C_2:= \alpha^2 \int_{\B}   |\del_\theta u_0|^2  \, dx >0,
   $$
since $u_0$ is an $x_1$-$x_2$-nonradial function. After possibly modifying $C_1, C_2>0$, this gives 
   \begin{align*}
   \frac{\int_{B_{1/\eps}} \left( u_0^2|\nabla \eta_\eps|^2 +2 \eta_\eps u_0 
   \nabla \eta_\eps \cdot \nabla u_0 - \alpha^2 \eps^2 \eta_\eps^2 |\del_\theta 
   u_0|^2  \, \right)dx}{\left(\int_{B_{1/\eps} } \eta_\eps^p|u_0|^p \, dx 
   \right)^\frac{2}{p}} \leq C_1 e^{-\frac{\delta_1}{\eps}} - C_2 \eps^2.
   \end{align*}   
   Next we consider the first term in \eqref{split quotient} and note that
   $$
   \frac{\int_{B_{1/\eps}} \eta_\eps^2 \left( |\nabla u_0|^2  + u_0^2 \, \right)dx}{\left(\int_{B_{1/\eps} } \eta_\eps^p |u_0|^p \, dx \right)^\frac{2}{p}}
    \leq \frac{\int_{\R^N} \left( |\nabla u_0|^2  + u_0^2 \, \right)dx}{\left(\int_{B_{1/(2\eps)} }  |u_0|^p \, dx \right)^\frac{2}{p}},
   $$
   while \eqref{Exponential decay} implies
   $$
   \int_{\R^N \setminus B_{1/(2\eps)} } |u_0|^p \, dx \leq C_3 e^{-\frac{\delta_2}{\eps}}
   $$
   for some $C_3,\delta_2>0$. It thus follows that
   \begin{align*} 
   \frac{\int_{B_{1/\eps}} \eta_\eps^2 \left( |\nabla u_0|^2  + u_0^2 \, \right)dx}{\left(\int_{B_{1/\eps} } \eta_\eps^p |u_0|^p \, dx \right)^\frac{2}{p}} 
   &\leq \frac{\int_{\R^N} \left( |\nabla u_0|^2  + u_0^2 \, \right)dx}{\left(\int_{B_{1/(2\eps)} } |u_0|^p \, dx \right)^\frac{2}{p}} 
   \leq \frac{\int_{\R^N}  \left( |\nabla u_0|^2  + u_0^2 \, \right)dx}{\left(\int_{\R^N }  |u_0|^p \, dx- C_3 e^{-\frac{\delta_2}{\eps}} \right)^\frac{2}{p}} \\
   &\le  \frac{\int_{\R^N}  \left( |\nabla u_0|^2  + u_0^2 \, \right)dx}{\left(\int_{\R^N }  |u_0|^p \, dx\right)^\frac{2}{p}} +C_4 e^{-\frac{\delta_2}{\eps}} = \scrC_{0,1,p}(\R^N) + C_4 e^{-\frac{2\delta_2}{p\eps}}
   \end{align*}
   for $\eps>0$ sufficiently small with some constant $C_4>0$, since $u_0$ attains $\scrC_{0,1,p}(\R^N)$. In view of \eqref{eps-lower bound} and \eqref{split quotient}, this yields that
   $$
  \scrC_{0,1,p}(\R^N) \leq   \scrC_{\alpha \eps ,1,p}(B_{1/\eps}) \le \scrC_{0,1,p}(\R^N) - C_2\eps^2 +C_1 e^{-\frac{\delta_1}{\eps}} + C_4 e^{-\frac{2\delta_2}{p\eps}}, 
   $$
and the right hand side of this inequality is smaller than 
$\scrC_{0,1,p}(\R^N)$ if $\eps>0$ is sufficiently small. This is a 
contradiction, and thus the claim follows in this case. 

In the case $\alpha=1$, the argument is the same up to replacing $H^1_0(\B)$ by 
$\cH$ and by considering the corresponding rescaled function space $\cH_\eps$ 
on $B_{1/\eps}$. Then the contradiction argument can be carried out in the same 
way, since radial functions in $\cH_\eps$ belong to $H^1_0(B_{1/\eps}) 
\subset H^1(\R^N)$.
   \end{proof}
%
%

%

   \section{The case of an annulus}
   \label{sec:case-an-annulus}
   In this section, we consider rotating solutions of (\ref{NLKG}) in the case where $\B$ is replaced by an annulus 
$$
\A_r:= \{ x \in \R^N: r < |x|<1\}
$$
for some $r \in (0,1)$. The ansatz \eqref{Spiraling wave ansatz} then leads to the reduced problem
\begin{equation} \label{Reduced equation annulus}
	\left\{ 
	\begin{aligned} 
		-\Delta u + \alpha^2 \del_\theta^2 u + m u & = |u|^{p-2} u \quad && \text{in $\A_r$,} \\
		u & = 0 && \text{on $\del \A_r$}
	\end{aligned}
	\right.
\end{equation}
where $m >-\lambda_1(\A_r)$, $p \in (2,\frac{2N}{N-2})$ and $\del_\theta = 
x_{N-1} \del_{x_N} - x_{N} \del_{x_{N-1}}$ as before. Here, $\lambda_1(\A_r)$ 
denotes the first Dirichlet eigenvalue of $-\Delta$ on $\A_r$. As in 
(\ref{eq:intro-minimization-problem}), we may then define
\begin{equation}
  \label{eq:intro-minimization-annulus}
\scrC_{\alpha,m,p}(\A_r) :=\inf_{u \in H^1_0(\A_r) \setminus \{0\}} R_{\alpha,m,p}(u)
\end{equation}
with the Rayleigh quotient $R_{\alpha,m,p}(u)$ given by (\ref{eq:def-Raleigh-quotient}) for functions $u \in H^1_0(\A_r)$.
In the following, a weak solution of (\ref{Reduced equation annulus}) will be 
called a ground state solution if it is a minimizer for 
(\ref{eq:intro-minimization-annulus}). We then have the following analogue of Theorem~\ref{main theorem}.
\begin{theorem}
\label{main-theorem-annulus}  
  Let $r \in (0,1)$, $m > -\lambda_1(\A_r)$ and $p \in (2,2^*)$.
  \begin{itemize}
  \item[(i)] If $\alpha \in (0,1)$, then there exists a ground state solution of (\ref{Reduced equation annulus}).
  \item[(ii)] We have 
        \begin{equation*}
        \scrC_{1,m,p}(\A_r) =0 \quad \text{for} \   p>2_{\text{\tiny $1$}}^*, \qquad \text{and} \quad \scrC_{1,m,p}(\A_r)>0 \quad \text{for} \  p \leq 2_{\text{\tiny $1$}}^*.	
        \end{equation*}
	Moreover, for any $p \in (2_{\text{\tiny $1$}}^*,2^*)$, there exists $\alpha_p \in (0,1)$ with the property that
        $$
        \scrC_{\alpha,m,p}(\A_r)< \scrC_{0,m,p}(\A_r) \qquad \text{for $\alpha \in (\alpha_p,1]$}
        $$
        and therefore every ground state solution of (\ref{Reduced equation annulus})
         is $x_1$-$x_2$-nonradial for $\alpha \in (\alpha_p,1]$.
      \end{itemize}
    \end{theorem}

This theorem does not come as a surprise and is proved by precisely the same 
arguments as Theorem~\ref{main theorem}, so we omit the proof. Instead, we now 
discuss an interesting additional feature of the annulus $\A_r$. Unlike in the 
case of the ball, we can formulate {\em explicit} sufficient conditions for the 
parameters $p,\alpha, m$ and $r$ which guarantee that every ground state 
solution of (\ref{Reduced equation annulus}) is $x_1$-$x_2$-nonradial. This is the 
content of the following theorem.

\begin{theorem} \label{Theorem: Symmetry breaking on annulus}
	Let $N \geq 2$, $m \geq 0$, $r,\alpha \in (0,1)$ and assume $p> \frac{N-1-r^2 \alpha^2}{\kappa(r,m)} + 2$ with 
        \begin{equation}
          \label{eq:def-tilde-mu}
        \kappa(r,m)= \left \{
          \begin{aligned}
            &m r^2 + \max \Bigl\{\Bigl( \frac{N-2}{2} \Bigr)^2, \Bigl(\frac{\pi}{1-r}\Bigr)^2 r^{N-1}\Bigr\},&&\quad N \ge 3;\\
            &mr^2 + \Bigl(\frac{\pi}{1-r}\Bigr)^2 r^{N}, &&\quad N = 2.
          \end{aligned}
        \right. 
	\end{equation}
        Then every ground state solution of \eqref{Reduced equation annulus} 
        is $x_1$-$x_2$-nonradial.
      \end{theorem}

      We point out that $\kappa (m,r) \to \infty$ if $m \to \infty$ or $r \to 1^-$. Consequently, for given $p>2$, ground states of
      \eqref{Reduced equation annulus} are nonradial if either $m$ is large or the annulus is thin, i.e. $r$ is close to $1$. The proof is based on the following lemma.
\begin{lemma} \label{Lemma: Nonradiality condition on annulus}
  Suppose that $m \ge 0$, and that there exists a function $v\in H^1_0(\A_r)$ satisfying 
        \begin{equation}
          \label{eq:first-con-annulus}
        	\int_{\mathbb{S}^{N-1}} v(s,\cdot) \, d\sigma  = 0 \qquad \text{for 
        	every $s \in (r,1)$}  
         \end{equation}
         and 
	\begin{equation}
          \label{eq:second-con-annulus}
        	\int_{\A_r} \left( |\nabla v |^2 -\alpha^2 |\del_\theta v|^2 + m v^2  \right) \, dx  - (p-1)    \int_{\A_r} |u_0|^{p-2} v^2 \,dx  <0.
	\end{equation}
	Then we have 
        \begin{equation}
          \label{eq:conclusion-lemma-annulus}
        \scrC_{\alpha,m,p}(\A_r)< R_{\alpha,m,p}(u_0),
	\end{equation}
	where $u_0 \in H^1_0(\A_r)$ is the unique positive radial solution of \eqref{Reduced equation annulus}.
\end{lemma}

Here we note that in the case $m = 0$, the uniqueness of the positive radial 
solution $u_0$ of \eqref{Reduced equation annulus} has been first proved by Ni 
and Nussbaum~\cite{Ni-Nussbaum}. In the case $m>0$, the uniqueness is due to 
Tang~\cite{tang} and Felmer, Mart\'inez and Tanaka~\cite{Felmer et al} for $N 
\geq 3$ and $N=2$, respectively.

\begin{proof}
  We argue by contradiction and assume that equality holds in (\ref{eq:conclusion-lemma-annulus}). Then $u_0$ is a minimizer for $\scrC_{\alpha,m,p}(\A_r)$, which implies, in particular, that
  \begin{equation}
    \label{eq:contra-consequence}
   R_{\alpha,m,p}'(u_0)v=0 \quad \text{and}\quad R_{\alpha,m,p}''(u_0)(v,v) \ge 0.
   \end{equation}
In the following, we write $R_{\alpha,m,p} = \frac{Z(u)}{N(u)}$ for $u \in 
H^1(\A_r) \setminus \{0\}$ with 
	$$
		Z(u)  := \int_{\A_r} \left( |\nabla u|^2 -\alpha^2 |\del_\theta u|^2 + m u^2 \right) \, dx  \quad \text{and}\quad 
		N(u) := {\left( \int_{\A_r} |u|^p \,dx \right)^\frac{2}{p}} .
	$$
	The first property in (\ref{eq:contra-consequence}) then implies  
	$N(u_0)Z'(u_0)v= Z(u_0)N'(u_0)v$ and consequently
	$$
	N(u_0)^3 [R_{\alpha,m,p}]''(u_0)(v,v)  = N(u_0)^2 Z''(u_0)(v,v) - Z(u_0) 
	N(u_0) N''(u_0)(v,v)
	$$
	for $v \in H^1_0(\A_r)$.
Therefore, the second property in (\ref{eq:contra-consequence}) yields
	$$
  Z''(u_0)(v,v) - \frac{Z(u_0)}{N(u_0)} N''(u_0)(v,v) \ge 0 .
	$$
	Moreover, noting that $u_0$ is a weak solution of \eqref{Reduced equation 
	annulus} and therefore $Z(u_0)=N(u_0)^\frac{p}{2}$, we conclude that 
	\begin{align*} 
		0 \le & \frac{1}{2} \left(  Z''(u_0)(v,v) - \frac{Z(u_0)}{N(u_0)} N''(u_0)(v,v) \right) \\
		= &  \int_{\A_r} \left( |\nabla v |^2 -\alpha^2 |\del_\theta v|^2 + m v^2  \right) \, dx  - (p-1)    \int_{\A_r} |u_0|^{p-2} v^2 \,dx \\
		&  +  (p-2) N(u_0)^{-\frac{p}{2}}  \left( \int_{\A_r} |u_0|^{p-2} u_0 \,v \, dx \right)^2 .
	\end{align*}
	This, however, contradicts (\ref{eq:second-con-annulus}) since $ \int_{\A} 
	|u_0|^{p-2} u_0\, v \, dx=0$ by (\ref{eq:first-con-annulus}). The proof is 
	thus finished.
\end{proof}
\begin{proof}[Proof of Theorem~\ref{Theorem: Symmetry breaking on annulus}] Our goal is to construct a function that satisfies the conditions of Lemma~\ref{Lemma: Nonradiality condition on annulus}.
	To this end, let $\mu_1$ be the first eigenvalue of the weighted eigenvalue problem 
	$$
	\left\{ 
	\begin{aligned}
		-w_{rr} - \frac{N-1}{r} + m w - (p-1) |u_0(r)|^{p-2} w & = 
		\frac{\mu}{r^2} w \qquad \text{in $(r,1)$,} \\
		w(r)=w(1) & = 0 ,
	\end{aligned}
	\right.
	$$
	and let $w$ the up to normalization unique positive eigenfunction.
	Moreover, let $Y \in C^\infty(\mathbb{S}^{N-1})$ be a spherical harmonic 
	of degree 
	1 such that $\del_\theta^2 Y=-Y$
	and set $v(r,\omega):=w(r) Y(\omega)$. Then condition 
	(\ref{eq:first-con-annulus}) of Lemma~\ref{Lemma: Nonradiality condition on 
	annulus} is satisfied. By construction, $v$ also satisfies
	$$
	-\Delta v +\alpha^2 \del_\theta^2 v + (m-\alpha^2) v - (p-1)|u_0|^{p-2} v = \frac{\mu_1+N-1}{|x|^2} v -\alpha^2 v
	$$
	and testing this equation with $v$ itself then yields
	\begin{align} 
		&\int_{\A_r} \left( |\nabla v|^2 -\alpha^2 |\del_\theta v|^2 + m v^2 - (p-1)|u_0|^{p-2} v^2 \right) \, dx \label{first-mu-ineq}\\
		=& (\mu_1+(N-1)) \int_{\A_r} \frac{v^2}{|x|^2} \, dx - \alpha^2 \int_{\A_r} v^2 \, dx 
		\leq (\mu_1+(N-1) - r^2 \alpha^2)  \int_{\A_r} \frac{v^2}{|x|^2} \, dx .
		\nonumber
	\end{align}
	We recall that $\mu_1$ can be characterized by
	\begin{align*}
		\mu_1 =	\min_{ \phi \in H^1_{0,rad}(\A_r) \setminus \{0\}}\frac{\int_{\A_r} \left(|\nabla \phi|^2 + m \phi^2 \right) \, dx - (p-1) \int_{\A_r} |u_0|^{p-2} \phi^2 \, dx }{\int_{\A_r} \frac{\phi^2}{|x|^2} \, dx  } .
	\end{align*}
	Taking $\phi=u_0$ in this quotient, we obtain the estimate
	\begin{align}
		\mu_1 & \leq	\frac{\int_{\A_r} \left(|\nabla u_0|^2 + m u_0^2 \right) \, dx - (p-1) \int_{\A_r} |u_0|^{p}  \, dx }{\int_{\A_r} \frac{u_0^2}{|x|^2} \, dx  } \label{mu-1-first-est}\\
                      & = - (p-2)	\frac{\int_{\A_r} \left(|\nabla u_0|^2 + m u_0^2 \right) \, dx  }{\int_{\A_r} \frac{u_0^2}{|x|^2} \, dx  }  \le -(p-2) \Bigl( \frac{\int_{\A_r}|\nabla u_0|^2 dx  }{\int_{\A_r} \frac{u_0^2}{|x|^2} \, dx}+mr^2\Bigr)
                        \nonumber
        \end{align}
	We now distinguish the cases $N \ge 3$ and $N=2$. 
        If $N \geq 3$, Hardy's inequality gives
        \begin{equation}
          \label{eq:hardy}
        \int_{\A_r} |\nabla u_0|^2  \, dx  \geq \left( \frac{N-2}{2}\right)^2 \int_{\A_r} \frac{u_0^2}{|x|^2} \, dx.
	\end{equation}
	Alternatively, we may also estimate, since $u_0$ is radial,
        \begin{align}
        \int_{\A_r} |\nabla u_0|^2  \, dx &= \int_{r}^1 \rho^{N-1}|\partial_r u_0(\rho)|^2 d\rho \ge
                                               r^{N-1} \int_{r}^1|\partial_r 
                                               u_0(\rho)|^2 d\rho  \nonumber
                                               \\ 
                                              &\ge 
                                              \Bigl(\frac{\pi}{1-r}\Bigr)^2r^{N-1}
                                               \int_{r}^1 u_0^2(\rho) d\rho 
                                              \nonumber
          \ge
\Bigl(\frac{\pi}{1-r}\Bigr)^2 r^{N-1} \int_{r}^1 \rho^{N-3} u_0^2(\rho) d\rho \\
& = \Bigl(\frac{\pi}{1-r}\Bigr)^2 r^{N-1}         
\int_{\A_r} \frac{u_0^2}{|x|^2} \, dx.   \label{eq:hardy-alternative}
        \end{align}
        Thus (\ref{mu-1-first-est}) gives $\mu_1 < -(p-2) \kappa(r,m)$ with 
        $\kappa(r,m)$ given in~(\ref{eq:def-tilde-mu}) for $N \ge 3$. Inserting 
        this into (\ref{first-mu-ineq}) yields
	\begin{align*} 
		& \int_{\A_r} \left( |\nabla v|^2 -\alpha^2 |\del_\theta v|^2 + m v^2 - (p-1)|u_0|^{p-2} v^2 \right) \, dx \\
		<& - (p-2)\kappa + N-1 - r^2 \alpha^2  ,
	\end{align*}
	i.e., condition (\ref{eq:second-con-annulus}) of Lemma~\ref{Lemma: Nonradiality condition on annulus} is satisfied if $p> \frac{N-1 - r^2 \alpha^2}{\kappa} +2$, which holds by assumption.

        Hence $v$ satisfies the assumptions of Lemma~\ref{Lemma: Nonradiality condition on annulus}, which implies that (\ref{eq:conclusion-lemma-annulus}) holds and therefore every minimizer for (\ref{eq:intro-minimization-annulus}) is nonradial. Let $u$ denote such a nonradial ground state solution, and suppose by contradiction that $\del_\theta u_0\equiv 0$ The nonradiality of $u$ implies that there exists an isometry $A \in O(N)$ such that $\tilde u := u \circ A \in H^1_0(\A_r) $ satisfies $ \del_\theta \tilde u \not \equiv 0$. Since $A$ is an isometry, this implies
	$$
	R_{\alpha,m,p}(\tilde u)=  R_{\alpha,m,p}(u) - \alpha^2  \frac{\int_{\A_r}  |\del_\theta \tilde u|^2  \, dx}{\left(\int_{\A_r} |u|^p \, dx \right)^\frac{2}{p}} < R_{\alpha,m,p}(u)=	\scrC_{1,m,p}(\A_r),
	$$
	which contradicts (\ref{eq:intro-minimization-annulus}). Consequently, we have $\del_\theta u_0 \not \equiv 0$, which yields that $u_0$ is $x_1$-$x_2$-nonradial. This finishes the proof in the case $N \ge 3$.
        
        It remains to consider the case $N=2$. In this case, we replace the estimates (\ref{eq:hardy}) and (\ref{eq:hardy-alternative}) by
        $$
        \int_{\A_r} |\nabla u_0|^2  \, dx  \ge \Bigl(\frac{\pi}{1-r}\Bigr)^2 r^{N-1} \int_{r}^1 u_0^2(\rho) d\rho \ge
         \Bigl(\frac{\pi}{1-r}\Bigr)^2 r^{N}         
\int_{\A_r} \frac{u_0^2}{|x|^2} \, dx.
$$
Combining this with (\ref{mu-1-first-est}) we again get $\mu_1 < -(p-2) \kappa(r,m)$ with $\kappa(r,m)$ given in~(\ref{eq:def-tilde-mu}) for $N=2$. We may thus complete the proof as above.
\end{proof}

\section{Riemannian models}
\label{sec:riemannian-models}

So far we only used the inequality stated in Theorem~\ref{general-degenerate-Sobolev-inequality-whole space} in the case $s=1$.
We shall now consider an application for general $s \in (0,2]$ by considering 
\eqref{NLKG} on a special class of Riemannian manifolds with boundary.

Indeed, consider a Riemannian model, i.e., a Riemannian manifold $(M,g)$ of dimension $N \geq 2$ admitting a pole $o \in M$ and whose metric is (locally) given by
\begin{equation} \label{Riemannian model metric}
ds^2 = dr^2 + (\psi(r))^2 d\Theta^2 
\end{equation}
for $r>0$, $\Theta \in \mathbb{S}^{N-1}$, where $d\Theta^2$ denotes the 
canonical metric 
on $\mathbb{S}^{N-1}$ and $\psi$ is a smooth function that is positive on 
$(0,\infty)$. 
Moreover, we assume
\begin{equation} \label{psi conditions Riemannian models}
\psi'(0)=1 \quad \text{and} \quad \psi^{(2k)}(0)=0 \qquad \text{for $k \in \N_0$.} 
\end{equation}
For such a Riemannian model, the associated Laplace-Beltrami operator becomes
$$
\Delta_g f= \frac{1}{\psi^{N-1}} \del_r \left( \psi^{N-1} \del_r f \right) + 
\frac{1}{\psi^2} \Delta_{\mathbb{S}^{N-1}} f
$$
where $\Delta_{\mathbb{S}^{N-1}}$ denotes the Laplace-Beltrami operator on 
$\mathbb{S}^{N-1}$.  Riemannian models are of independent geometric interest, 
we refer to \cite{Berchio-Ferrero-Grillo} and the references therein for a 
broader overview.

In the following, we focus on the case $M=\B$, $o=0 \in \R^N$ and again study the problem
\begin{equation}  \label{NLKG Riemannian model}
\left\{ 
\begin{aligned}
\del_t^2 v - \Delta_g v +m v & = |v|^{p-2} v \quad && \text{in $M$} \\
v & = 0 && \text{on $\del M$}
\end{aligned}
\right.
\end{equation}
where $2<p<\frac{2N}{N-2}$ and $m>-\lambda_1(M))$ with $\lambda_1(M)$ denoting 
the first Dirichlet eigenvalue of $-\Delta_g$ on $M$. We stress that the case 
$\psi(r)=r$ corresponds to the classical flat metric on $\B$ considered in 
detail in the previous sections. As a further example, the hemisphere of radius 
$\frac{2}{\pi}$ given by $\mathbb{S}^{N}_{2/\pi,+} := \{ x \in 
\mathbb{S}^N_{2/\pi}: x_N>0\}$ can be interpreted as a 
Riemannian model. Indeed, using polar coordinates $(r,\omega) \in (0,1) 
\times \mathbb{S}^{N-1}$, a parametrization $\B \to \mathbb{S}^N_{2/\pi,+}$ is 
given by
$(r,\omega) \mapsto \frac{2}{\pi}(\sin (\frac{\pi}{2} r) \, \omega, \cos 
(\frac{\pi}{2} r) )$. This 
yields a Riemannian model with $\psi(r) = \sin (\frac{\pi}{2} r)$.
Similarly, Riemannian models can also be used for spherical caps.

As in the flat case, we restrict our attention to solutions of (\ref{NLKG 
Riemannian model}) of the form $v(t,x)=u(R_{\alpha t} (x))$, where $R_{\theta} 
\in O(N+1)$ denote a planar rotation in $\R^N$ with angle $\theta$. We may 
again assume, without loss of generality, that 
$$
R_{\theta}(x)= (x_{1} \cos \theta  + x_2 \sin \theta , -x_{1} \sin \theta  + x_2 \cos \theta ,x_3,\dots,x_N) \qquad \text{for $x \in \R^N$},
$$
so $R_{\theta}$ is the rotation in the $x_{1}-x_2$-plane. This leads to the 
reduced equation
\begin{equation} \label{Reduced equation Riemannian model}
\left\{ 
\begin{aligned} 
-\Delta_g u + \alpha^2 \del_\theta^2 u + m u & = |u|^{p-2} u \quad && \text{in $M$} \\
u & = 0 && \text{on $\del M$}
\end{aligned}
\right.
\end{equation}
with the differential operator $\partial_\theta  = x_1 \partial_{x_2} - x_2 \partial_{x_1}$ associated to the Killing vector field $x \mapsto (-x_2,x_1,0,\dots,0)$ on $M$. We may then again study the quotient
$$
\qquad R^M_{\alpha,m,p}: H_0^1(M) \setminus \{0\} \to \R, \qquad \qquad R^M_{\alpha,m,p}(u):= \frac{\int_M \left(|\nabla_g u|^2 - \alpha^2 |\del_\theta u|^2 + m u^2 \right) \, dg}{\|u\|_{L^p(M)}^2}
$$ 
and its minimizers, i.e.
$$
\scrC_{\alpha,m,p}(M) :=\inf_{u \in C_c^1(\B) \setminus \{0\}} 
R_{\alpha,m,p}^M(u).
$$
Analogously to Theorem~\ref{main theorem}, we can use the general inequality 
stated in Theorem~\ref{general-degenerate-Sobolev-inequality-whole space} to 
give the following result.

\begin{theorem}
  \label{main-theorem-rm}
  Let $s \in (0,2]$, and let $(M,g)$ be a Riemannian model with $M=\B$ and 
  associated function $\psi \in C^\infty[0,1)$ satisfying
  (\ref{psi conditions Riemannian models}) and 
	\begin{equation} \label{psi condition main theorem}
	c_1(1-r)^s \le 1- \psi(r) \le c_2 (1-r)^s \qquad \text{for $r \in (0,1)$ with constants $c_1,c_2>0$.}
	\end{equation}
Moreover, let $m > -\lambda_1(M)$, and let $2<p<2^*$.
  \begin{itemize}
  \item[(i)] If $\alpha \in (0,1)$, then there exists a ground state solution of (\ref{Reduced equation}).
  \item[(ii)] We have 
    \begin{equation}
      \label{eq:C-identities-rm}
	\scrC_{1,m,p}(M) =0 \quad \text{for} \   p>2_s^*, \qquad \text{and} \quad 
	\scrC_{1,m,p}(M)>0 \quad \text{for} \  p \leq 2_s^*.
    \end{equation}
	Moreover, for any $p \in (2_s^*,2^*)$, there exists $\alpha_p \in (0,1)$ with the property that
        $$
        \scrC_{\alpha,m,p}(M)< \scrC_{0,m,p}(M) \qquad \text{for $\alpha \in (\alpha_p,1]$}
        $$
        and therefore every ground state solution of (\ref{Reduced equation Riemannian model}) is $x_1$-$x_2$-nonradial for $\alpha \in (\alpha_p,1)$.
      \end{itemize}
    \end{theorem}
    \begin{proof}
Since the proof is completely parallel to the proof of Theorem~\ref{main 
theorem}, we omit some details and focus our attention on showing where 
condition (\ref{psi condition main theorem}) enters. It is again useful to 
introduce polar coordinates $(r,\theta,\vartheta_1,\dots,\vartheta_{N-2}) \in 
U:= (0,1) \times (-\pi,\pi) \times (0,\pi)^{N-2}$ given by 
  \begin{align*}
    (x_1,\dots,x_N) =& (r \cos \theta \sin\vartheta_1 \cdots \sin \vartheta_{N-2},\: r \sin \theta \sin\vartheta_1 \cdots \sin \vartheta_{N-2},\nonumber\\
    &r \cos \vartheta_{1}, r \sin \vartheta_1 \cos \vartheta_2, \dots, r \sin 
    \vartheta_1 \dots \sin \vartheta_{N-3} \cos \vartheta_{N-2},r \sin 
    \vartheta_1 \dots  \vartheta_{N-2} ). \label{polar-coordinates-rm}
  \end{align*}
  In the following, we will abbreviate the coordinates
  $(\theta,\vartheta_1,\ldots,\vartheta_{N-2})$ to $\Theta$ for simplicity.
In these coordinates, the metric \eqref{Riemannian model metric} becomes
$$
g=dr^2 + (\psi(r))^2 \left( \sum \limits_{i=1}^{N-2} \left( \prod_{k=1}^{i-1} \sin^2 \vartheta_k \right) d\vartheta_k^2 + \left( \prod_{k=1}^{N-1} \sin^2 \vartheta_k \right) d\theta^2 \right),
$$
and therefore the quadratic form associated to the operator $-\Delta_g + \alpha^2 \del_\theta^2$ is given by 
	\begin{align*}
	& \int_{M} \left( |\nabla_g u|^2 - |\del_\theta u|^2 \right) \, dg  \\
	= & \int_U  \left( |\del_r u|^2 + \frac{1}{\psi^2} \sum \limits_{i=1}^{N-2} 
	h_i |\del_{\vartheta_i} u|^2 + \left(\frac{h_{N-1}}{\psi^2}-1\right) 
	|\del_\theta u|^2 \right) |g| \, d(r,\Theta) 	\end{align*}
for $u \in C^1_c(M)$ with
        \begin{equation*}
        |g|(r,\Theta) = (\psi(r))^{N-1}  \prod_{k=1}^{N-2} 
        \sin^{N-1-k}\vartheta_k ,\qquad 
		h_i(r,\Theta) = \prod_{k=1}^{i-1} 
		\frac{1}{\sin^2 \vartheta_k}.
        \end{equation*}
Moreover,
$$
\int_{M} |u|^p\,dg = \int_{U} g |u|^p\,d(r,\Theta)\qquad \text{for $u \in 
C^1_c(M)$ and $p>1$.}
$$
Next we note that, as a consequence of  (\ref{psi condition main theorem}), we have
      \begin{equation}
        \label{eq:limit-Theta-0}
      |g|(\Theta_0)=1 \quad \text{and}\quad h^i(\Theta_0)=1\qquad \text{for $i 
      =1,\dots,N$ with $\Theta_0:= 
      \left(1,0,\frac{\pi}{2},\dots,\frac{\pi}{2}\right)$.}
      \end{equation}
Moreover, by assumption~(\ref{psi condition main theorem}), the function $\frac{h_{N-1}}{\psi^2}-1$ satisfies
\begin{equation}
  \label{eq:key-two-sided-est}
\tilde c_1 \Bigl((1-r)^s + \sum_{k=1}^{N-2} 
\bigl(\vartheta_k-\frac{\pi}{2}\bigr)^2\Bigr)\le   
\frac{h_{N-1}}{\psi^2}(r,\Theta)-1 \le \tilde c_2 \Bigl((1-r)^s + 
\sum_{k=1}^{N-2} \bigl(\vartheta_k-\frac{\pi}{2}\bigr)^2 \Bigr)
\end{equation}
for $(r,\theta,\vartheta_1,\dots,\vartheta_{N-2}) \in U_0$ with suitable constants $\tilde c_1, \tilde c_2>0$, where
$$
U_0 := \left(\frac{1}{2},1\right) \times \bigl(-\pi,\pi \bigr) \times 
\left(\frac{\pi}{4},\frac{3}{4}\pi\right)^{N-2} \subset U.
$$
We now consider a fixed function $u \in C^1_c(U_0) \setminus \{0\} \subset C^1_c(U) \setminus \{0\}$, which, regarded as a function of polar coordinates, gives rise to a function in $C^1_c(M)$. For $\lambda \in (0,1)$ we consider the map
      $$
      \Lambda_\lambda : U_0 \to U_0, \qquad  (r,\Theta) \mapsto
      \left(1 + \lambda(1-r), \lambda^{1+\frac{s}{2}} \theta, 
      \frac{\pi}{2}+\lambda\left(\vartheta_1 - 
      \frac{\pi}{2}\right),\dots,\frac{\pi}{2}+\lambda\left(\vartheta_2 - 
      \frac{\pi}{2}\right)\right),
      $$
      and we define $u_\lambda := u \circ \Lambda_\lambda^{-1} \in C_c^1(U_0) 
      \setminus \{0\}$ for $\lambda \in (0,1).$
 
Using (\ref{eq:limit-Theta-0}) and (\ref{eq:key-two-sided-est}), we find that 
      \begin{align*}
        \lambda^{-\frac{2N+s}{p}}\Bigl( \int_{U} g |u_\lambda|^p 
        \,d(r,\Theta)\Bigr)^{\frac{2}{p}}  
        &= \Bigl( \int_{U} g \circ \Lambda_\lambda |u|^p 
        \,d(r,\Theta)\Bigr)^{\frac{2}{p}} 
        \nonumber\\
        & \to \Bigl(\int_{U} |u|^p     
        \,d(r,\Theta)\Bigr)^{\frac{2}{p}} =:  
        c_u(p)
      \end{align*}
as $\lambda \to 0^+$ and
        \begin{align}
          &\limsup_{\lambda \to 0^+}\lambda^{2-\frac{s}{2}-N} \int_U  \Bigl( 
          |\del_r u_\lambda |^2 + \frac{1}{\psi^2} \sum \limits_{i=1}^{N-2} h_i 
          |\del_{\vartheta_i} u_\lambda|^2 + 
          \left(\frac{h_{N-1}}{\psi^2}-1\right) |\del_\theta u_\lambda|^2 
          \Bigr) |g| \, d(r,\Theta)
          \nonumber \\
          =& \limsup_{\lambda \to 0^+}\int_U  \Bigl( |\del_r u |^2 + 
          \frac{1}{\psi^2} \circ  \Lambda_\lambda \sum \limits_{i=1}^{N-2} h_i 
          \circ \Lambda_\lambda |\del_{\vartheta_i} u|^2 + 
          \lambda^{-s}\Bigl(\frac{h_{N-1}}{\psi^2}\circ \Lambda_\lambda 
          -1\Bigr) |\del_\theta u|^2 \Bigr) |g|\circ \Lambda \, 
          d(r,\Theta)\nonumber\\
          \le& d_u^1 + d_u^2     \label{second-lim-rm}
	\end{align}
        with
        $$
        d_u^1 := \int_U  \Bigl( |\del_r u |^2 + \sum \limits_{i=1}^{N-2} 
        |\del_{\vartheta_i} u|^2\Bigr) d(r,\Theta)
        $$
        and
        \begin{align*}
          d_u^2 &= \tilde c_2 \limsup_{\lambda \to 0^+} \int_U   \Bigl((1-r)^s + \lambda^{2-s}
                  \sum_{k=1}^{N-2} \bigl(\vartheta_k-\frac{\pi}{2}\bigr)^2 
                  \Bigr)|\del_\theta u|^2 d(r,\Theta)\\
                &= \left \{
                  \begin{aligned}
                    &\tilde c_2 \int_U (1-r)^s|\del_\theta u|^2\, 
                    d(r,\Theta),&&\qquad s \in (0,2),\\
                    &\tilde c_2 \int_U \Bigl((1-r)^2 + \sum \limits_{i=1}^{N-2} 
                    \bigl(\vartheta_k-\frac{\pi}{2}\bigr)^2\Bigr) |\del_\theta 
                    u|^2\, d(r,\Theta),&&\qquad s = 2.
                  \end{aligned}
                  \right.                  
        \end{align*}
      It thus follows that
      $$
      \scrC_{1,m,p}(M) \le \limsup_{\lambda \to 0^+} R_{1,m,p}^M(u_\lambda) = 
      \limsup_{\lambda \to 0^+}\frac{
        \lambda^{N+ \frac{s}{2}-2}(d_u^1+d_u^2) + \lambda^{\frac{2N+s}{2}}c_u(2)}{\lambda^{\frac{2N+s}{p}}c_u(p)}
      = 0 \quad \text{if $p > 2_s^*$.}
      $$
      This shows the first identity in (\ref{eq:C-identities-rm}). To see the second identity in (\ref{eq:C-identities-rm}), we argue as in Section~\ref{sec:degen-sobol-ineq}. More precisely, we first note that it is sufficient to consider the case $p=2_s^*$, and then we show the inequality
      \begin{align*}
        \Bigl(\int_{U} g |u|^{2^*_s}\,d(r,\Theta)\Bigr)^{\frac{2}{2^*_s}} \le 
        C \int_U  \left( |\del_r u|^2 + \frac{1}{\psi^2} \sum 
        \limits_{i=1}^{N-2} h_i |\del_{\vartheta_i} u|^2 + 
        \left(\frac{h_{N-1}}{\psi^2}-1\right) |\del_\theta u|^2 \right) |g| \, 
        d(r,\Theta)
      \end{align*}
      for functions $u \in C^1_c(U_0)$ with a suitable constant $C>0$. For this, we use Theorem~\ref{existence of minimizers - whole space - intro} and the first inequality in (\ref{eq:key-two-sided-est}). The argument is then completed by using the rotation invariance of the problem and a partition of unity argument to localize the problem.
    \end{proof}

    \begin{remark}
      \label{final-remark-rm}
      \begin{enumerate}
      	\item[(i)]
      	In the case of a hemisphere mentioned earlier, i.e. 
      	$\psi(r) = 
      	\frac{2}{\pi}\sin (\frac{\pi}{2} r)$, Theorem~\ref{main-theorem-rm}
        yields nonradial ground 
      	state solutions for 
      	$p>2^*_s=\frac{2(N+1)}{N-1}$. Notably, this corresponds to the critical 
      	exponent for generalized travelling waves on the sphere 
      	$\mathbb{S}^N$ found in \cite{Taylor,Mukherjee,Mukherjee2}. 
      	In fact, our approach based on Theorem~\ref{existence of minimizers - whole space - intro} can be used to give an alternative proof for the 
      	existence of nontrivial solutions and the embeddings stated in 
      	\cite[Proposition 3.2]{Taylor} and \cite[Proposition 1.2 + Lemma 
      	1.3]{Mukherjee2}.
    \item[(ii)]
	      Theorem~\ref{main-theorem-rm} leaves open the case $s>2$. 
	      Note that the two-sided estimate \eqref{eq:key-two-sided-est} 
	       needs to be analyzed more carefully if $s>2$ 
	      and $N \geq 3$, as the leading order term is then $2$ in place of 
	      $s$. In this case, if \eqref{psi condition main 
	      theorem} 
	      holds for some $s>2$, the conclusion will instead be
	      \begin{equation*}
	      \scrC_{1,m,p}(M) =0 \quad \text{for} \   p>2_{\text{\tiny $2$}}^*, 
	      \qquad 
	      \text{and} \quad \scrC_{1,m,p}(M)>0 \quad \text{for} \  p \leq 
	      2_{\text{\tiny $2$}}^*.
	      \end{equation*}
          For $N=2$, on the other hand, \eqref{eq:key-two-sided-est} 
          suggests that Theorem~\ref{main-theorem-rm} also holds for $s>2$.
    \end{enumerate}
    \end{remark}

%
%
%
%
%
%
%
%
%
%
\appendix
\section{Boundedness of solutions}
In the proof of the regularity properties of ground states in the case 
$\alpha=1$ stated in Proposition~\ref{Prop: ground states for critical case}, 
we used the following:
\begin{lemma} \label{Lemma: Moser iteration}
	Let $2<p<2_{\text{\tiny $1$}}^*$, $m>-\lambda_1$ and let $u \in \cH$ be a 
	weak solution of 
	\begin{equation} \label{app: equation}
		-\Delta u + \del_\theta^2 u + m u = |u|^{p-2} u \quad \text{in $\B$}.
	\end{equation}
	Then $u \in L^\infty(\B)$. Furthermore, there exist constants 
	$C=C(N,m),\sigma>0$ such that
	\begin{equation} \label{Linfty bound}
		|u|_\infty \leq C \|u\|_{\cH}^\sigma .
	\end{equation}
	For $m \geq 0$, the constant $C=C(N)>0$ can be chosen independent of $m$.
\end{lemma}
\begin{proof}
	The proof is based on Moser iteration, cf. \cite[Appendix B]{Struwe} and 
	the references therein. 
	
	We fix $L,s \geq 2$ and consider auxiliary functions $h, g \in 
	C^1([0,\infty))$ defined by 
	$$
	h(t) := s \int_0^t \min \{\tau^{s-1},L^{s-1}\}\,d\tau \qquad 
	\text{and}\qquad g(t) := \int_0^t [h'(\tau)]^2 \, d\tau 
	$$
	We note that 
	\begin{equation}
		\label{eq:additional-appendix}
		h(t)= t^s\quad \text{for $t \le L$}\qquad \text{and}\qquad g(t) \leq t 
		g'(t)= t(h'(t))^2 \quad \text{for $t \geq 0$,}  
	\end{equation}
	since the function $t \mapsto h'(t) = s \min \{t^{s-1},L^{s-1}\}$ is 
	nondecreasing.  We shall now show that $w := u^+ \in L^\infty(\B)$, and 
	that $\|w\|_\infty$ is bounded by the right hand side of \eqref{Linfty 
		bound}. Since we may replace $u$ with $-u$, the claim will then follow.
	
	We note that $w \in \cH$ and $\phi:= g(w) \in \cH$ with 
	$$
	\nabla w = \1_{\{u>0\}} \nabla u, \qquad \nabla \phi= g'(w) \nabla w , 
	\qquad \del_\theta w  = \1_{\{u>0\}} \del_\theta u, \qquad \del_\theta 
	\phi= g'(w) \del_\theta w .
	$$
	This follows from the boundedness of $g'$ and the estimate $g(t) \leq s^2 
	t^{2s-1}$ for $t \ge 0$. Testing \eqref{app: equation} with $\phi$ gives
	$$
	\int_{\B} \Bigl(\nabla u \cdot \nabla \phi - (\del_\theta u \ \del_\theta 
	\phi) + m u \phi\Bigr)dx = \int_{\B} |u|^{p-2} u \phi  \, dx ,
	$$
	from where we estimate
	\begin{align}
		\int_{\B} \Bigl(|\nabla h(w) |^2 - (\del_\theta h(w))^2 + mw g(w)\Bigr) 
		dx &=\int_{\B} \Bigl(g'(w) \Bigl(|\nabla w |^2 -(\del_\theta w)^2 
		\Bigr) + mu g(w)\Bigr) dx \nonumber\\
		&=  \int_{\B} |u|^{p-2} u g(w) \, dx \nonumber\\
		& \leq \int_{\B} w^p  (h'(w))^2  \, dx.\label{appendix-first-inequality}
	\end{align}
	Here we used (\ref{eq:additional-appendix}) in the last step. We now fix $r 
	>1$ with $\frac{(p-2)r}{r-1} \geq 2$ and $q>4r$. Combining 
	(\ref{appendix-first-inequality}) with Proposition~\ref{lambda-1-alpha-pos} 
	and Theorem~\ref{degenerate Sobolev inequality - ball}, we obtain the 
	inequality 
	\begin{equation}
		\label{eq:appendix-second-inequality}
		|h(w)|_{p^*}^2    \leq c_0  \int_{\B} w^p  (h'(w))^2  \, dx 
	\end{equation}
	with a constant $c_0=c_0(N,m)>0$. Note that for $m \geq 0$, $c_0$ only 
	depends on $N$. Since 
	$$
	h(t)= t^s,\quad h'(t) = st^{s-1}\quad \text{and}\quad g(t)= s^2 \int_0^t 
	\tau^{2s-2}\,d\tau = \frac{s^2}{2s-1} t^{2s-1}\qquad \text{for $t \le L$,}
	$$
	we may let $L \to \infty$ in (\ref{eq:appendix-second-inequality}) and 
	apply Lebesgue's theorem to obtain 
	$$
	|w^s|_{p^*}^2 \leq c_0 s^2 \int_{\B} w^{p+2s-2} \, dx \le c_0 s^2 
	|w|_{p^*}^{p-2} |w|_{2sq}^{2s} ,
	$$
	where $q=\frac{p^*}{p^*-p+2}$ is the the conjugated exponent to 
	$\frac{p^*}{p-2}$. 
	This yields 
	\begin{equation}
		\label{eq:appendix-third-inequality}
		|w|_{sp^*} \le (c_1s)^{\frac{1}{s}}  |w|_{2sq}\qquad \text{with $c_1:= 
		\Bigl(c_0 
			|w|_{p^*}^{p-2}\Bigr)^{\frac{1}{2}}$,} 
	\end{equation}
	whenever $w \in L^{2sq}(\B)$. 
	We now consider $s=s_n= \rho^n$ for $n \in \N$ with $\rho := 
	\frac{p^*}{2q}= \frac{2+p^*-p}{2} >1$, so that 
	$$
	2s_1 q = p^* \qquad \text{and}\qquad 2s_{n+1} q = s_n p^*  \quad \text{for 
	$n \in \N$.}
	$$
	Iteration of (\ref{eq:appendix-third-inequality}) then gives 
	$$
	|w|_{\rho^n p^*} = |w|_{s_n p^*} \leq |w|_{p^*} \prod_{j=1}^n (c_1 
	\rho^j)^{\rho^{-j}}  \leq c_1^{\frac{\rho}{\rho-1}} c_2 |w|_{p^*}
	$$
	for all $n$ with
	$$
	c_2 : =  \rho^{\sum \limits_{j=1}^\infty j \rho^{-j}} < \infty.
	$$
	It follows that
	\begin{equation} \label{p-norm limit}
		|w|_\infty = \lim_{n \to \infty} |w|_{ \rho^n p^*} \leq 
		c_1^{\frac{\rho}{\rho-1}} c_2 |w|_{p^*} . 
	\end{equation}
	Moreover, by Theorem~\ref{degenerate Sobolev inequality - ball}, we have 
	$$
	c_1 \leq c_1' \|w\|_{\cH}^{\frac{p-2}{2}} \le c_1' 
	\|u\|_{\cH}^{\frac{p-2}{2}}\qquad \text{and}\qquad  |w|_q\le \tilde c  
	\|w\|_{\cH} \le \tilde c  \|u\|_{\cH}  
	$$
	with constants $c_1', \tilde c>0$ depending only on $N$. It thus follows 
	from \eqref{p-norm limit} that
	$$
	|w|_\infty \leq C \|u\|_{\cH}^{\frac{(p-2)\rho}{2(\rho-1)} +1}\qquad 
	\text{with}\qquad C:= c_2 (c_1')^{\frac{\rho}{\rho-1}}\tilde c. 
	$$
	The proof is thus finished.	
\end{proof}
%
%
%
%
%


\begin{thebibliography}{50}
	
	\bibitem{Aftalion-Pacella}
	\newblock A. Aftalion and F. Pacella,
	\newblock Uniqueness and nondegeneracy for some nonlinear elliptic problems 
	in a ball, 
	\newblock \emph{J. Differential Equations} \textbf{195} (2003), 380--397.
	
	\bibitem{Agudelo-Kuebler-Weth}
	\newblock O. Agudelo, J. K\"ubler and T. Weth,
	\newblock Spiraling solutions of nonlinear Schr\"odinger equations, 
	\newblock to appear in \emph{Proc. Roy. Soc. Edinburgh Sect. A}.
	
	\bibitem{Ambrosetti-malchiodi}
	\newblock A. Ambrosetti and A. Malchiodi,
	\newblock Perturbation methods and semilinear elliptic problems on ${\bf 
	R}^n$.
	\newblock Progress in Mathematics, 240. Birkhäuser Verlag, Basel, 2006. 
	xii+183 pp. 
	
	
	\bibitem{Beckner}
	\newblock W. Beckner,
	\newblock On the Grushin operator and hyperbolic symmetry,
	\newblock \emph{Proc. Amer. Math. Soc.} \textbf{129} (2001), 1233--1246. 
	
	\bibitem{Ben-Naoum-Mahwin}
	\newblock A.K. Ben-Naoum and J. Mawhin, 
	\newblock Periodic solutions of some semilinear wave equation on balls and 
	on 
	spheres,
	\newblock \emph{Topol. Methods Nonlinear Anal.} \textbf{1} (1993), 113--137.
	
	\bibitem{Berchio-Ferrero-Grillo}
	\newblock E. Berchio, A. Ferrero and G. Grillo, 
	\newblock Stability and qualitative properties of radial solutions of the 
	Lane-Emden-Fowler equation on Riemannian models,
	\newblock \emph{J. Math. Pures Appl.} \textbf{102} (2014), 1--35. 
	
	\bibitem{Berestycki-Lions}
	\newblock H. Berestycki and P.-L. Lions, 
	\newblock Nonlinear scalar field equations. I. Existence of a ground state,
	\newblock \emph{Arch. Rational Mech. Anal.} \textbf{82} (1983), no. 4, 
	313--345.
	
	
	\bibitem{chen-evequoz-weth}
	\newblock H. Chen, G. Ev\'equoz and T. Weth,  
	\newblock Complex solutions and stationary scattering for the nonlinear 
	Helmholtz equation,
	\newblock \emph{SIAM J. Math. Anal.} \textbf{53} (2021), 2349--2372.
	
	
	\bibitem{Chen-Zhang}
	\newblock J. Chen and Z. Zhang, 
	\newblock Infinitely many periodic solutions for a semilinear wave equation 
	in 
	a ball in $\R^n$,
	\newblock \emph{J. Differential Equations} \textbf{256} (2014), 1718--1734. 
	
	\bibitem{Chen-Zhang-2}
	\newblock J. Chen and Z. Zhang, 
	\newblock Existence of infinitely many periodic solutions for the radially 
	symmetric wave equation with resonance,
	\newblock \emph{J. Differential Equations} \textbf{260} (2016), 6017--6037.
	
	\bibitem{Chen-Zhang-3}
	\newblock J. Chen and Z. Zhang, 
	\newblock Existence of multiple periodic solutions to asymptotically linear 
	wave equations in a ball,
	\newblock \emph{Calc. Var. Partial Differential Equations} \textbf{56} 
	(2017),  
	Paper No. 58, 25 pp. 
	
	\bibitem{Damascelli-Grossi-Pacella}
	\newblock L. Damascelli, Massimo Grossi, and F. Pacella,
	\newblock Qualitative properties of positive solutions of semilinear 
	elliptic equations in symmetric domains via the maximum  principle,
	\newblock \emph{Annales de l'I.H.P. Analyse non linéaire} \textbf{16} 
	(1999), 631--652.
	
	\bibitem{del Pino-Musso-Pacard}
	\newblock M. del Pino, M. Musso and F. Pacard,
	\newblock Solutions of the Allen–Cahn equation which are invariant under 
	screw-motion,
	\emph{ Manuscripta Math.} \textbf{138} (2012), 273--286.
	
	\bibitem{evequoz-weth}
	\newblock G. Ev\'equoz and T. Weth, 
	\newblock Dual variational methods and nonvanishing for the nonlinear 
	Helmholtz 
	equation,
	\newblock \emph{Adv. Math.} \textbf{280} (2015), 690--728. 
	
	\bibitem{Felmer et al}
	\newblock P. Felmer, S. Mart\'inez and K. Tanaka, 
	\newblock Uniqueness of radially symmetric positive solutions for $-\Delta 
	u + 
	u = u^p$ in an annulus,
	\newblock \emph{J. Differential Equations} \textbf{245} (2008), 1198--1209.
	
	\bibitem{Filippas et al}
	S. Filippas, L. Moschini and A. Tertikas,
	\newblock On a class of weighted anisotropic Sobolev inequalities,
	\newblock \emph{J. Funct. Anal.} \textbf{255} (2008), 90--119.
	
	\bibitem{Frank-Jin-Xiong}
	\newblock R. L. Frank, T. Jin and J. Xiong,
	\newblock Minimizers for the fractional Sobolev inequality on domains,
	\newblock \emph{Calc. Var. Partial Differential Equations} \textbf{57} 
	(2018), 43.
	
	\bibitem{Gagliardo}
	\newblock E. Gagliardo, 
	\newblock Propriet\`a di alcune classi di funzioni in pi\`uvariabili, 
	\emph{Ricerche Mat.} \textbf{7} (1958), 102--137.
	
	\bibitem{Gidas-Ni-Nirenberg}
	\newblock B. Gidas, W.-M. Ni, and L. Nirenberg, 
	\newblock \emph{Symmetry and related properties via the maximum principle},
	\newblock Comm. Math. Phys. \textbf{68} (1979), 209--243.
	
	\bibitem{Gilbarg-Trudinger}
	\newblock D. Gilbarg and N.S. Trudinger,
	\newblock Elliptic partial differential equations of second order,
	\newblock Springer, 2015.
	
	\bibitem{gutierrez}
	\newblock S. Guti{\'e}rrez,
	\newblock Non trivial $L^q$-solutions to the Ginzburg-Landau equation,
	\newblock {Math. Ann.} \textbf{328} (2004), no. 1-2, 1–25. 
	
	\bibitem{Hajlasz-Koskela}
	\newblock P. Haj\l asz and P. Koskela, 
	\newblock Sobolev met Poincar\'e,
	\newblock \emph{Mem. Amer. Math. Soc.} \textbf{145} (2000), x+101 pp.
	
	\bibitem{Higuchi}
	\newblock A. Higuchi,
	\newblock Symmetric tensor spherical harmonics on the N-sphere and their 
	application to the de Sitter group $SO(N,1)$, 
	\newblock \emph{J. Math. Phys.} \textbf{28} (1987), 1553--1566. 
	
	\bibitem{hirsch-reichel}
	\newblock A. Hirsch and W. Reichel, 
	\newblock Real-valued, time-periodic localized weak solutions for a 
	semilinear 
	wave equation with periodic potentials,
	\newblock \emph{Nonlinearity} \textbf{32} (2019), 1408--1439.
	
	\bibitem{kuebler}
	\newblock J. K\"ubler,
	\newblock \emph{in preparation.}
	
	\bibitem{Kwong}
	\newblock M. K. Kwong, 
	\newblock Uniqueness of positive solutions of $\Delta u - u + u^p = 0$ in 
	$\R^n$,
	\newblock \emph{Arch. Rational Mech. Anal.} \textbf{105} (1989), 243--266.
	
	\bibitem{Kwong-Li}
	\newblock M. K. Kwong and L.Q. Zhang, 
	\newblock Uniqueness of the positive solution of $\Delta u + f(u) = 0$ in 
	an annulus,
	\newblock \emph{Differential Integral Equations} \textbf{4} (1991), 
	583--599.
	
	\bibitem{Lieb-Seiringer}
	\newblock E. Lieb and R. Seiringer, 
	\newblock Derivation of the Gross-Pitaevskii equation for rotating Bose 
	gases, 
	\newblock \emph{Comm. Math. Phys.} \textbf{264} (2006), 505--537. 
	
	\bibitem{Lions}
	\newblock P.-L. Lions,
	\newblock The concentration-compactness principle in the calculus of 
	variations. The locally compact case, part 1,
	\newblock \emph{Annales de l'I.H.P. Analyse non linéaire} \textbf{1} 
	(1984), 109--145.
	
	\bibitem{mandel-scheider-yesil}
	\newblock R. Mandel, D. Scheider and T. Ye\c{s}il, 
	\newblock Dual variational methods for a nonlinear Helmholtz equation with 
	sign-changing nonlinearity,
	\newblock \emph{Calc. Var. Partial Differential Equations} \textbf{60} 
	(2021),  
	13 pp. 
	
	\bibitem{mandel-montefusco-pellaci}
	\newblock R. Mandel, E. Montefusco and B. Pellacci, 
	\newblock Oscillating solutions for nonlinear Helmholtz equations,
	\newblock \emph{Z. Angew. Math. Phys.} \textbf{68} (2017), 19 pp.
	
	\bibitem{mandel-scheider}
	\newblock R. Mandel, and D. Scheider, 
	\newblock Variational methods for breather solutions of nonlinear wave 
	equations,
	\newblock \emph{Nonlinearity} \textbf{34} (2021), 3618--3640.
	
	
	\bibitem{McLead-Serrin}
	\newblock K. McLeod and J. Serrin, 
	\newblock Uniqueness of positive radial solutions of $\Delta u + f(u) = 0$ 
	in $\R^n$,
	\newblock \emph{Arch. Rational Mech. Anal.} \textbf{99} (1987), 115--145.
	
	\bibitem{Monti}
	\newblock R. Monti, 
	\newblock Sobolev inequalities for weighted gradients,
	\newblock \emph{Comm. Partial Differential Equations} \textbf{31} (2006), 
	1479--1504.
	
	\bibitem{Monti-Morbidelli}
	\newblock R. Monti and D. Morbidelli, 
	\newblock Kelvin transform for Grushin operators and critical semilinear 
	equations,
	\newblock \emph{Duke Math. J.} \textbf{131} (2006), 167--202.
	
	\bibitem{Mukherjee}
	\newblock M. Mukherjee,
	\newblock Nonlinear travelling waves on complete Riemannian manifolds,
	\newblock \emph{Adv. Differential Equations} \textbf{23} (2018), 65--88.
	
	\bibitem{Mukherjee2}
	\newblock M. Mukherjee,
	\newblock A special class of nonlinear hypoelliptic equations on spheres,
	\newblock \emph{NoDEA Nonlinear Differential Equations Appl.} \textbf{24} 
	(2017), 25 pp.
	
	\bibitem{Ni-Nussbaum}
	\newblock W.-M. Ni and R. D. Nussbaum, 
	\newblock Uniqueness and nonuniqueness for positive radial solutions of 
	$\Delta 
	u + f (u, r) = 0$,
	\newblock \emph{Comm. Pure Appl. Math.} \textbf{38} (1985), 67--108.
	
	
	\bibitem{Seiringer}
	\newblock R. Seiringer, 
	\newblock Gross-Pitaevskii theory of the rotating Bose gas,
	\newblock \emph{Comm. Math. Phys.} \textbf{229} (2002), 491--509. 
	
	
	\bibitem{Struwe}
	\newblock M. Struwe,
	\newblock Variational Methods,
	\newblock Springer, Berlin, 2008.
	
	\bibitem{tang}
	\newblock M. Tang, 
	\newblock Uniqueness of positive radial solutions for 
	$\Delta u - u+u^p=0$ on an annulus, 
	\newblock \emph{J. Differential Equations} \textbf{189} (2003), 148--160.
	
	\bibitem{Taylor}
	\newblock M. Taylor,
	\newblock Traveling wave solutions to NLS and NLKG equations in 
	non-Euclidean settings,
	\newblock \emph{Houston J. Math.} \textbf{42} (2016), 143--165.
	
\end{thebibliography}
\end{document}